\documentclass[11pt]{article}%

\usepackage[colorinlistoftodos]{todonotes}

\usepackage{amssymb}
\usepackage{amsfonts}
\usepackage{amsmath}
\usepackage{mathtools}
\usepackage{dsfont}

\usepackage{comment}
\usepackage{amsthm}

\usepackage{hyperref}
\hypersetup{
    colorlinks=true,
    linkcolor=blue,
    citecolor=red,
    urlcolor=blue,
    pdfborder={0 0 0}
}

\usepackage{cleveref}

\usepackage{float}
\usepackage{graphicx}
\usepackage{wrapfig}
\usepackage{subfig}
\usepackage{fancybox}
\usepackage{framed}
\usepackage{caption}
\usepackage{epstopdf} 
\usepackage[all]{xy} 
\usepackage{tikz} 
\usepackage{tabularx}
\usepackage{afterpage}
\usepackage{arydshln}

\usepackage[titletoc, title]{appendix} 
\usepackage{titling} 
\usepackage{url} 
\usepackage{color}
\usepackage{enumitem}
\usepackage{enumerate, multirow, longtable} 

\usepackage{pdfpages}

\usepackage[latin1]{inputenc}
\usepackage[a4paper]{geometry}
\usepackage{microtype}

\providecommand{\U}[1]{\protect\rule{.1in}{.1in}}
\newtheorem{theo}{Theorem}[section]
\newtheorem{conjecture}{Conjecture}
\newtheorem{prop}[theo]{Proposition}
\newtheorem{lem}[theo]{Lemma}
\newtheorem{cor}[theo]{Corollary}

\newtheorem{rem}[theo]{Remark}

\numberwithin{equation}{section}

\newcommand{\EE}{\mathbb{E}}
\newcommand{\E}{\mathbb{E}}

\newcommand{\RR}{\mathbb{R}}

\def\E{\EE}

\newcommand{\Ia}{ {\mathcal I }}

\newcommand{\Pa}{ {\mathcal P }}

\newcommand{\R}{\mathbb{R}}

\def\bl{\boldsymbol{\lambda}}
\def\ef{\mathbf{f}}

\newcommand{\lhk}{\mathbf{p}^{\gamma, D}} 

\newcommand{\bdec}[3]{\mathbf{E}_{#1 \overset{#3}{\to} #2}}

\newcommand{\pdec}[3]{\mathbf{P}_{#1 \overset{#3}{\to} #2}}

\title{Weyl's law in Liouville quantum gravity}
\author{
Nathana\"el Berestycki\thanks{University of Vienna, \href{mailto:nathanael.berestycki@univie.ac.at}{nathanael.berestycki@univie.ac.at}}\:
and Mo Dick Wong\thanks{Durham University, \href{mailto:mo-dick.wong@durham.ac.uk}{mo-dick.wong@durham.ac.uk}}}
\date{\today}
\begin{document}
\maketitle

\begin{abstract}
    Can you hear the shape of Liouville quantum gravity? We obtain a Weyl law for the
 eigenvalues of Liouville Brownian motion: the $n$-th eigenvalue grows
 linearly with $n$, with the proportionality constant given by the
 Liouville area of the domain and a certain deterministic constant $c_\gamma$
 depending on $\gamma \in (0, 2)$. The constant $c_\gamma$, initially a complicated function of Sheffield's quantum cone, can be evaluated explicitly and is strictly greater than the equivalent Riemannian constant. 
 
 At the heart of the proof we obtain
 sharp asymptotics of independent interest for the small-time behaviour of the
 on-diagonal heat kernel. Interestingly, we show that the scaled heat
 kernel displays nontrivial pointwise fluctuations. Fortunately, 
 at the level of the heat trace these pointwise fluctuations
 cancel each other, which leads to the result.

 We complement these results with
 a number of conjectures on the spectral geometry of Liouville quantum gravity, notably suggesting a connection with quantum chaos.
\end{abstract}

\tableofcontents 

\section{Problem setting and result}

\subsection{Weyl's law}

Let $D \subset \mathbb{R}^2 \cong \mathbb{C}$ be a simply connected\footnote{This assumption is probably not necessary but is convenient for some estimates. We have chosen not to make the assumptions on the domain as general possible in order to keep the paper to a reasonable length. With some effort it should be possible to prove the results assuming only that $D$ is a bounded domain with at least one boundary regular point. To avoid any confusion, recall that a point $z \in \partial D$ is called regular if, for a planar Brownian motion $(W_t)_{t \ge 0}$ starting from $z$, we have $\mathbb{P}_z(\inf \{t > 0: W_t \not \in D\} = 0) = 1$, i.e., $W$ leaves $D$ immediately. 
}, 
bounded domain 
and let $h(\cdot)$ be the Gaussian free field on $D$ with Dirichlet boundary condition, i.e. $h(\cdot)$ is a centred Gaussian field on $D$ with covariance kernel given by
\begin{align*}
\mathbb{E}[h(x) h(y)] = G_0^D(x, y) \qquad \forall x, y \in D
\end{align*}

\noindent where $G_0^D(x, y)$ is the Dirichlet-boundary Green's function on $D$. In other words, for all $x \ne y$ in $D$ we have
\begin{align*}
G_0^D(x, y) = \pi\int_0^\infty p_t^D(x,y) dt
\end{align*}

\noindent where $p_t^D(\cdot, \cdot)$ is the Dirichlet heat kernel on $D$, with our time parametrisation chosen such that it represents the transition density of a standard (two-dimensional) Brownian motion (with killing at the boundary). In particular, for any $x \in D$ we have
\begin{align*}
p_t^D(x,x) \overset{t \to 0+}{\sim} (2\pi t)^{-1} \qquad
\text{and} \qquad G_0^D(x, y) \overset{y \to x}{=} - \log|x-y| + \mathcal{O}(1).
\end{align*}
Note that there is no factor of two or $\pi$ in the logarithmic blow-up on the right hand side above, which is a result of our conventions on the Green function and the Gaussian free field (these are consistent with other works on Liouville quantum gravity).

For $\gamma \in (0, 2)$, we denote by $\mu_{\gamma} (d\cdot)$ the \textbf{Liouville  measure} (or Gaussian multiplicative chaos measure) associated to $h(\cdot)$, i.e. 
\begin{align} \label{eq:def_LQG}
\mu_{\gamma}(dx) 
& = \lim_{\epsilon \to 0^+} \epsilon^{\frac{\gamma^2}{2}}e^{\gamma h_{\epsilon}(x)} dx\\
\notag & = \lim_{\epsilon \to 0^+} R(x; D)^{\frac{\gamma^2}{2}}e^{\gamma h_{\epsilon}(x) - \frac{\gamma^2}{2} \mathbb{E}[h_{\epsilon}(x)^2]} dx, \qquad x \in D
\end{align}

\noindent where $R(x; D)$ is the conformal radius of $D$ from $x$. The Liouville measure plays a central role in the emerging theory of Liouville quantum gravity (LQG) \cite{KPZ, DuplantierMillerSheffield}, or equivalently (but with a slightly different perspective), Liouville conformal field theory \cite{DKRV, DOZZ}; see again \cite{BP} for a survey including a discussion of the physical motivations and references.

\begin{figure}[h!]
\begin{center}
\includegraphics[width=0.49\textwidth]{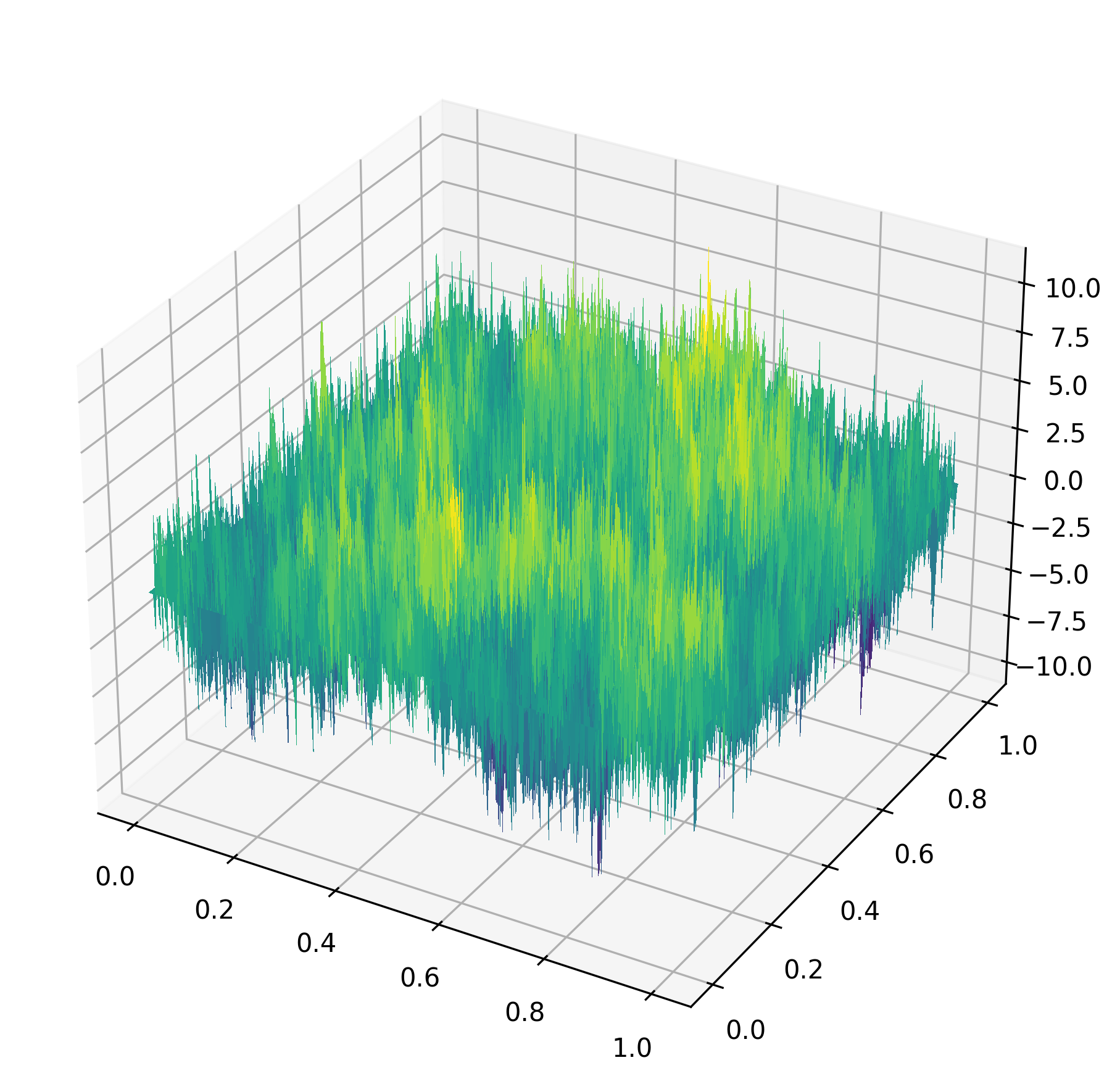}
\includegraphics[width=0.49\textwidth]{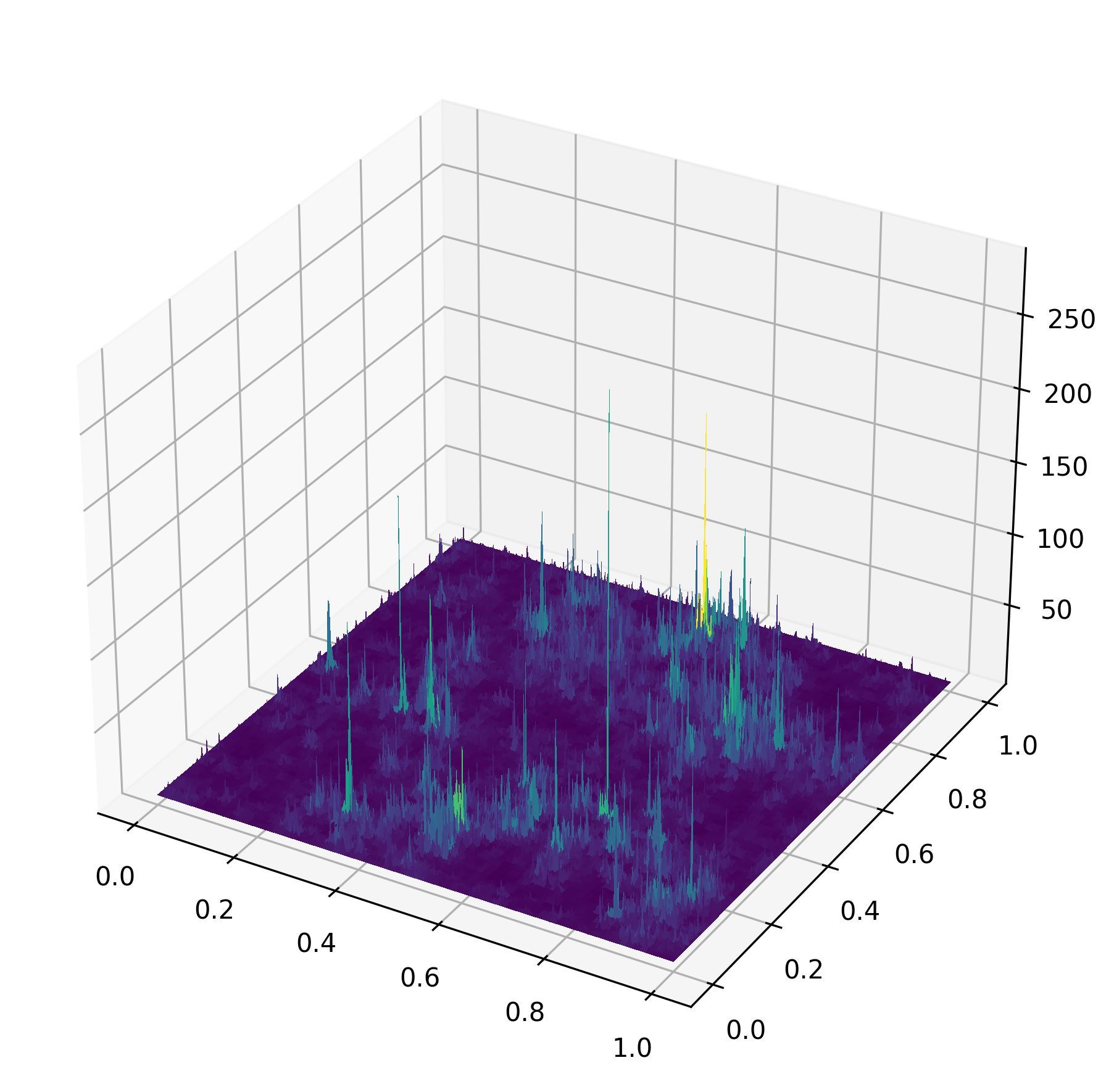}
\end{center}
\caption{\label{fig:GFFGMC} Left: realisation of a mollified GFF $h_\epsilon$. Right: density profile of $e^{\gamma h_\epsilon}$ with $\gamma = 0.5$. The mollification/discretisation scale is chosen to be of order $\epsilon \approx 10^{-3}$ on $D = [0,1]^2$.}
\end{figure}

In this article we are interested in some fundamental questions pertaining to the geometry of Liouville quantum gravity. The basic problem which motivates us is the following analogue of Mark Kac's celebrated question \cite{Kac1966}:

\begin{center}
\text{Can one hear the shape of Liouville quantum gravity?}
\end{center}

In Mark Kac's original question, the setting is the following: we are given a bounded domain $D \subset \R^d$, and the sequence of eigenvalues $(\lambda_n)_{n\ge 0}$ corresponding to $-\tfrac{1}{2}\Delta$ with Dirichlet boundary conditions in $D$, and ask if this sequence determines $D$ up to isometry (i.e., up to translation, reflection and rotation). Kac's question has served as a motivation for a remarkable body of work. As is well known since the fundamental work of Weyl \cite{Wey1911}, the eigenvalues determine at least the volume of $D$, since if we call 
$N_0(\lambda) = \sum_{n\ge 0}1_{\{ \lambda_n \le \lambda\}}$ 
the eigenvalue counting function, then the celebrated Weyl law asserts that
\begin{equation}
\frac{N_0(\lambda)}{(2\lambda)^{d/2} } \to \frac{\omega_d}{(2\pi)^d}\text{Leb}(D)    
\end{equation}
where $\omega_d$ is the volume of the unit ball in $\mathbb{R}^d$. Weyl's law is known to hold in a great degree of generality including Neumann boundary conditions and can be extended to the setting of Riemannian geometry (see e.g. \cite{Chavel}). However, it is also known that the answer to Kac's question in general is negative (counterexamples were obtained first by Milnor for five-dimensional surfaces \cite{Mil1964}, and by Gordon, Webb and Wolpert for concrete bounded planar domains \cite{GWW1992}).\\

In this paper we initiate the study of this problem in the context of Liouville quantum gravity, and more generally we begin an investigation of the spectral geometry of LQG, see Figure \ref{F:eigen}. Given a bounded domain $D$, let $(\mathbf{B}_t)_{t\ge 0}$ denote the Liouville Brownian motion on $D$ (\cite{Ber2015}, \cite{GRV}) which we recall is the canonical diffusion in the geometry of LQG. While the infinitesimal generator of this process may not be easily described, the Green measure $\mathbf{G}(x, dy)$ associated to it is rather straightforward, since by construction $\mathbf{B}$ is a time-change of ordinary Brownian motion. This leads to the expression (\cite{GRV_hk}):
\begin{equation}
\mathbf{G}(x, dy) = G_0^D(x, y) \mu_\gamma(dy).
\end{equation}
It is not hard to check that for a fixed $x \in D$, the right hand side is a finite measure on $D$ when $\gamma<2$, and this can also be made sense a.s. for all $x\in D$ simultaneously. The spectral theorem can then be applied (see \cite[Section 3]{MRVZ} on the torus, and \cite[Proposition 5.2]{AK2016} for the case of a bounded domain with Dirichlet boundary conditions, which is of interest here; see also \cite{GRV_hk} for the definition of the Liouville Green function). By definition (\cite{MRVZ, AK2016}) the eigenvalues $\boldsymbol{\lambda}_n = \boldsymbol{\lambda}_n(\gamma)$ of Liouville Brownian motion are the inverses of the eigenvalues of $\mathbf{G}$; we also call $\mathbf{f}_n(\cdot) = \mathbf{f}_n(\cdot; \gamma)$ the corresponding eigenfunctions, normalised to have unit $L^2 ( \mu_\gamma)$ norms. (The eigenvalues and eigenfunctions are fundamentally related to the \textbf{Liouville heat kernel} via a trace formula -- see in particular \cite{MRVZ} and \cite{AK2016} for a careful discussion -- this will play an important role in our paper but will be discussed later in \Cref{sec:introheat}). Equivalently, the eigenpairs $(\boldsymbol{\lambda}_n, \mathbf{f}_n)$ could be defined from the Dirichlet form associated to Liouville Brownian motion \cite{GRV}: we have
$$
\int_ D (\nabla g \cdot \nabla \mathbf{f}_n)\ dx = \boldsymbol{\lambda}_n \int_D g \mathbf{f}_n \mu_\gamma(dx) \qquad \forall g \in L^2(\mu_\gamma) \cap H_0^1(D).
$$

We are now ready to state our main conjecture concerning the analogue of Kac's question for Liouville quantum gravity:

\begin{conjecture}\label{C:Kac}
One \textbf{can} almost surely hear the shape of Liouville quantum gravity. More precisely, the Gaussian free field $h$ is a measurable function of the eigenvalues: that is, 
there exists a measurable function $\phi$ such that 
$$
h = \phi ( (\boldsymbol{\lambda}_n)_{n\ge 0}),
$$
almost surely.
\end{conjecture}

In this conjecture the domain $D$ was fixed and assumed to be known. If we do not assume $D$ to be known then it is natural to ask whether the sequence $(\boldsymbol{\lambda}_n)_{n\ge 0}$ determines both the domain $D$ and the Gaussian free field $h$ living on it. However, one quickly realises that if two pairs  $(D_1,h_1)$ and $(D_2, h_2)$ are equivalent in the sense of random surfaces (see \cite{DS2011}) then they generate the same eigenvalue sequence. 
A slightly stronger form of Conjecture \ref{C:Kac} is therefore:

\begin{conjecture}\label{conj:Kac2}
The eigenvalue sequence $(\boldsymbol{\lambda}_n)_{n\ge 0}$ determines the pair $(D, h)$ modulo equivalence of random surfaces.
\end{conjecture}

In fact, it is not hard to see that Conjecture \ref{C:Kac} implies the stronger form Conjecture \ref{conj:Kac2}. These conjectures are partly motivated by the results of Zelditch \cite{Zelditch} which show that spectral determination is ``generically'' possible subject to analyticity conditions on the boundary and some extra symmetries.

\medskip In this paper we will not aim to prove this conjecture but instead show that the analogue of Weyl's law for Liouville quantum gravity holds: that is, $(\boldsymbol{\lambda}_n)_{n\ge 0}$ determines at least the LQG volume $\mu_\gamma(D)$ of $D$. More precisely, our main result is the following. Suppose the eigenvalues $(\boldsymbol{\lambda}_n)_{n \ge 0}$ are sorted in increasing order, and define the eigenvalue counting function by
\begin{equation}
\mathbf{N}_{\gamma}(\lambda) := \sum_{n\ge 0} 1_{\{\boldsymbol{\lambda}_n \le \lambda\}}.
\end{equation}

\begin{theo}\label{T:Weyl_intro} Let $0< \gamma <2$.
We have
\begin{align}\label{eq:Weyl}
\frac{\mathbf{N}_{\gamma}(\lambda)}{\lambda} \xrightarrow[\lambda \to \infty]{p} c_\gamma \mu_{\gamma}(D).
\end{align}
Here, the constant $c_\gamma = c_\gamma(Q-\gamma)$, where $Q = \frac{\gamma}{2} + \frac{2}{\gamma}$ and for $m>0$, $c_\gamma (m)$ is defined as follows:
\begin{align}\label{eq:constant}
c_\gamma(m) :=  \frac{1}{\pi}\Bigg\{\mathbb{E}\left[\int_{0}^\infty
\mathcal{I}\left(e^{\gamma (B_t - mt)}\right)dt\right] 
+\mathbb{E}\left[\int_{0}^{\infty} 
\mathcal{I}\left(e^{\gamma \mathcal{B}_t^{m}} \right)dt\right]\Bigg\}
\end{align}
where
\begin{align}\label{eq:Ical}
\mathcal{I}(x) := x e^{-x}, \qquad x \in \mathbb{R},
\end{align}
$(B_t)_{t \ge 0}$ is a standard (1-dimensional) Brownian motion, and $(\mathcal{B}_t^m)_{t \ge 0}$ is a Brownian motion with drift $m>0$ conditioned to be non-negative at all times $t \ge 0$.
\end{theo}

Readers familiar with Sheffield's theory of quantum cones (\cite{zipper}, see also \cite{DuplantierMillerSheffield}) will recognise the constant $c_\gamma$ as a somewhat complicated functional of the so-called $\gamma$-quantum cone. Perhaps surprisingly, this constant can be evaluated explicitly: 
\begin{theo}\label{theo:constant}
For any $\gamma\in (0,2), m > 0$, we have
$c_\gamma(m) = 1/({\pi \gamma m})$.
In particular, 
\begin{equation}
c_\gamma= \frac1{\pi ( 2 - \gamma^2/2)}.
\end{equation}
Moreover, $\lim_{\gamma \to 0^+} c_\gamma =c_0 := 1/(2\pi)$ and $c_\gamma > c_0$.
\end{theo}

\begin{figure}[h!]
\begin{center}
\includegraphics[width=.99\textwidth]{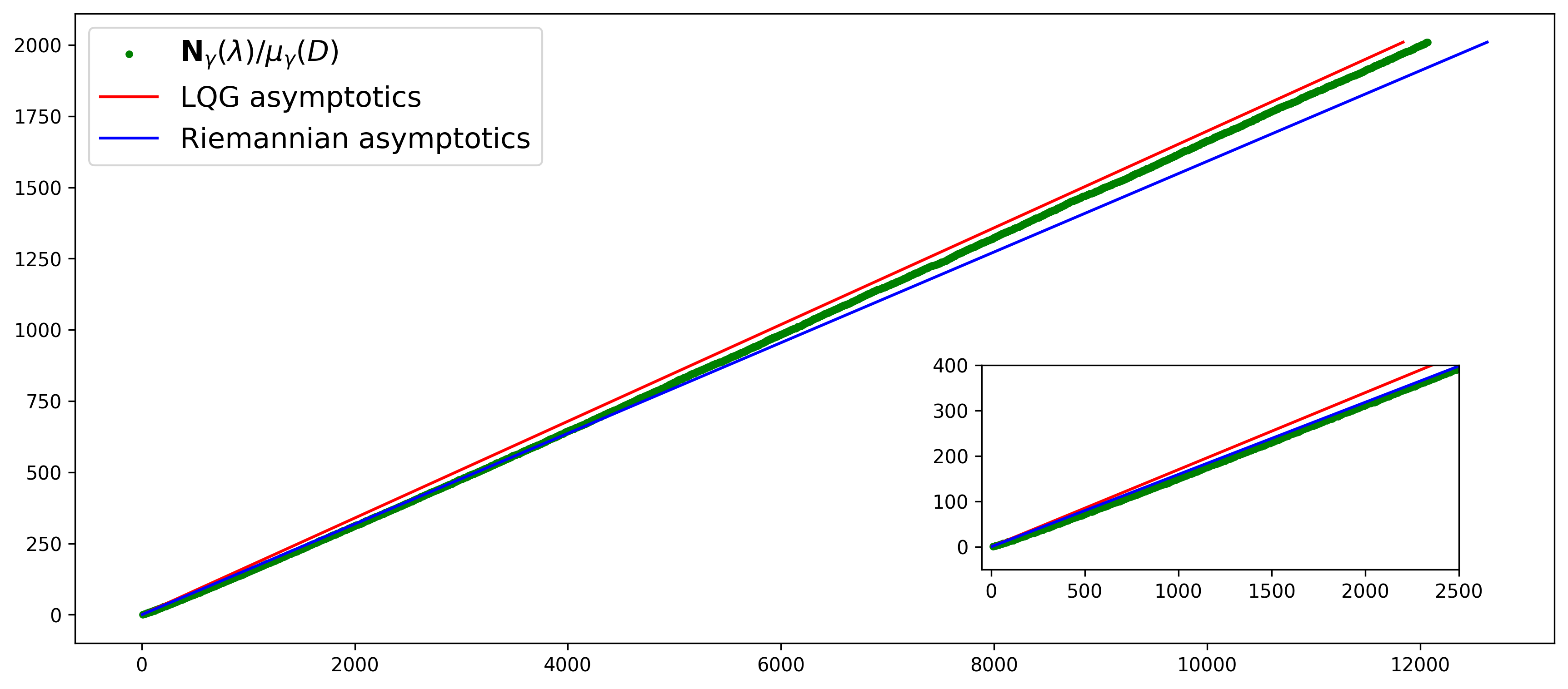}
\end{center}
\caption{Weyl's law for LQG with $\gamma = 0.5$. Green: volume-normalised eigenvalue counting function $\lambda \mapsto \mathbf{N}_\gamma(\lambda) / \mu_{\gamma}(D)$. Red: prediction from \Cref{T:Weyl_intro} ($\lambda \mapsto c_\gamma \lambda)$. Blue: Riemannian prediction ($\lambda \mapsto c_0\lambda $).}
\label{F:eigen}
\end{figure}

Theorem \ref{T:Weyl_intro} corresponds to a Weyl law where the dimension $d$ is taken to be $d=2$. (Note that taking the limit $\gamma \to 0^+$ we recover, at least formally, the classical Weyl's law for Euclidean domains). This corresponds to the fact that the spectral dimension of Liouville quantum gravity is equal to two (see \cite{RhodesVargas_spectral}, conjectured earlier by Ambj\o{}rn \cite{ABNRW1998}). At the same time, the fact that $c_\gamma>c_0$ shows that one cannot merely naively extrapolate the Riemannian result to LQG. This should probably be viewed as a consequence of the highly disordered, multifractal nature of the geometry in LQG; see Figure \ref{F:eigen}.

Finally, it is known that the Liouville measure $\mu_\gamma$ determines the Gaussian free field $h$ (see \cite{BerestyckiSheffieldSun}). This, however, does not imply Conjecture \ref{C:Kac} since we would need to know not only the LQG-mass of the domain $D$ but also that of any (say, open) subset of $D$ in order to entirely determine the measure $\mu_\gamma$.

\subsection{Conjectures and questions on the spectral geometry of LQG}
In addition to Conjectures \ref{C:Kac} and \ref{conj:Kac2} above, we record in this section a number of conjectures on the spectral geometry of Liouville quantum gravity.
\Cref{F:eigen} shows the growth of the volume-normalised eigenvalue counting function $\lambda \mapsto \mathbf{N}_{\gamma}(\lambda) / \mu_\gamma(D)$ associated to the realisation of GFF in \Cref{fig:GFFGMC} and compares it against theoretical predictions from \Cref{T:Weyl_intro} as well as Weyl's law for Riemannian manifolds. It is curious to see that the Riemannian prediction provides a better fit for the initial eigenvalues.  This may be explained by the fact that the low-frequency eigenpairs computed do not ``feel'' the roughness of $\gamma$-LQG surface (which could be an artefact of the numerical experiment as it involves mollified Gaussian free field on a discretised domain); see \Cref{fig:efcontour} for a comparison of eigenfunctions.

\begin{figure}[h!]
\begin{center}
\includegraphics[width=.99\textwidth]{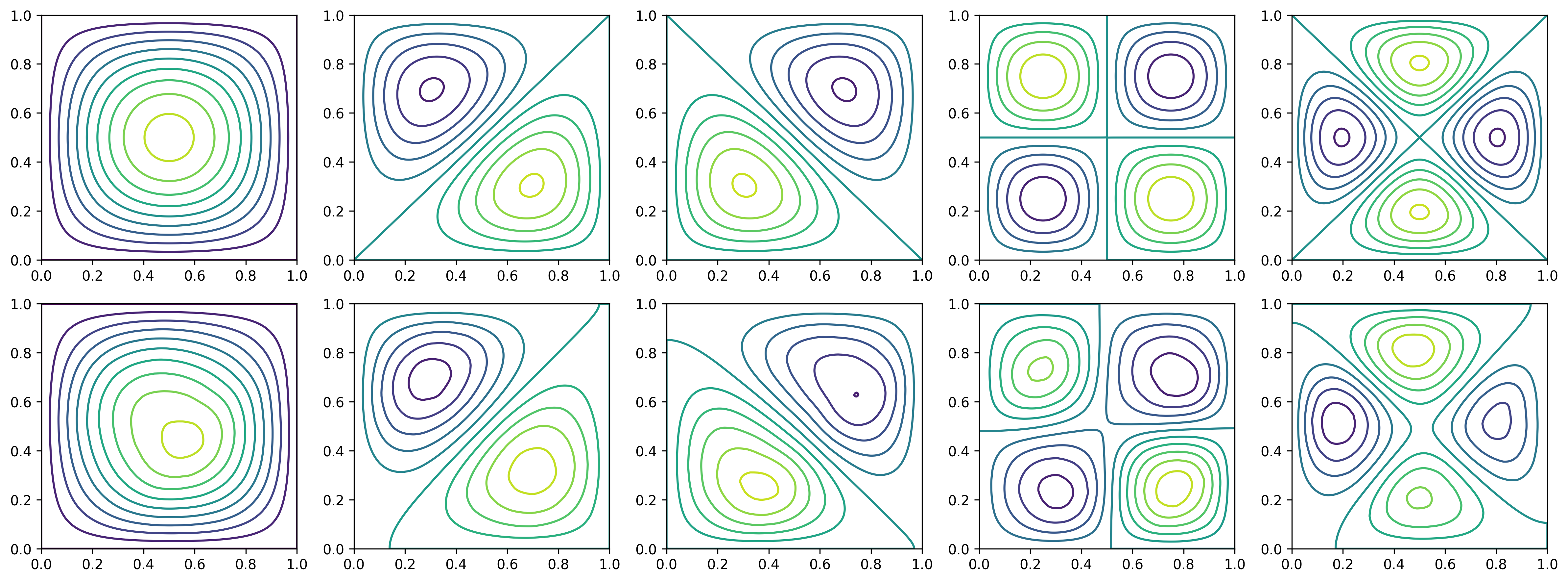}
\end{center}
\caption{Contour maps of the first $5$ eigenfunctions - Euclidean (top) versus LQG (bottom).}
\label{fig:efcontour}
\end{figure}

These simulations and others below suggest a rich picture for the spectral geometry of LQG; we view the results in this article as the first step of an in-depth study in this direction.

\paragraph{P\'olya's conjecture.} To begin with, we note that the eigenvalue counting function appears to always stay below the linear function predicted by its Weyl's law.

\begin{conjecture}
With probability one, $\mathbf{N}_\gamma(\lambda) \le c_\gamma \mu_\gamma(D) \lambda$ for all $\lambda\ge 0$. 
    \end{conjecture}

This is the LQG analogue of a famous conjecture of P\'olya \cite{Pol1954} for Euclidean domains, which is open in general (P\'olya proved it for the so-called tiling domains \cite{Polya_tilingdomains}, whereas the case for Euclidean balls has been established by Filonov et al. \cite{FLPS2023} only very recently). A closely related result is the Berezin--Li--Yau inequality \cite{Berezin, LiYau} which, informally, says that the conjecture holds for Euclidean domains in a Cesaro sense. Note that this conjecture is only plausible because $c_\gamma > c_0$. \\

\paragraph{Second term in Weyl's law.} A fascinating question concerns the second order term for the asymptotics of $\mathbf{N}_\gamma(\lambda)$ as $\lambda \to \infty$. In the Euclidean world, Weyl famously conjectured that this is of order $\sqrt{\lambda}$ for smooth domains $D$; more precisely (under our normalisation)
$$
N_0(\lambda) = c_0 \text{Leb} (D) \lambda - \frac1{2\pi} |\partial D| \sqrt{\lambda} + o( \sqrt{\lambda})
$$
where $|\partial D|$ denotes the length of the boundary of $D$. Surprisingly this conjecture is still open in general, as it has been established under an additional geometric assumption by Ivrii \cite{Ivrii100} (essentially, there should not be ``too many'' periodic geodesics). While this assumption is believed to hold for any smooth domains, it remains to be verified.

In the LQG context, it would be interesting to understand what the correct order of $ c_\gamma \mu_\gamma(D) \lambda - \mathbf{N}_\gamma(\lambda)$ should be, and whether one could ``hear the perimeter" of the domain. Answers to these questions could be subtle, as it was observed in the literature of random fractals that there could be competitions between boundary corrections and random fluctuations (see e.g. \cite{CharmoyCroydonHambly}). The choice of the Dirichlet variant of GFF here may also affect the subleading order, since the mass distribution with respect to $\mu_\gamma$ has a rapid decay near the boundary $\partial D$. In our simulation with $\gamma = 0.5$, $\mathbf{N}_{\gamma}(\lambda)$ behaves like $c_{\gamma} \mu_\gamma(D) \lambda + \mathcal{O}(\lambda^{b})$ with $b$ being much smaller than $1/2$, and the deviation from the best fitting power-law curve appears to follow some central limit theorem, see \Cref{fig:error}.

\begin{figure}[h!]
\begin{center}
\includegraphics[width = \textwidth]{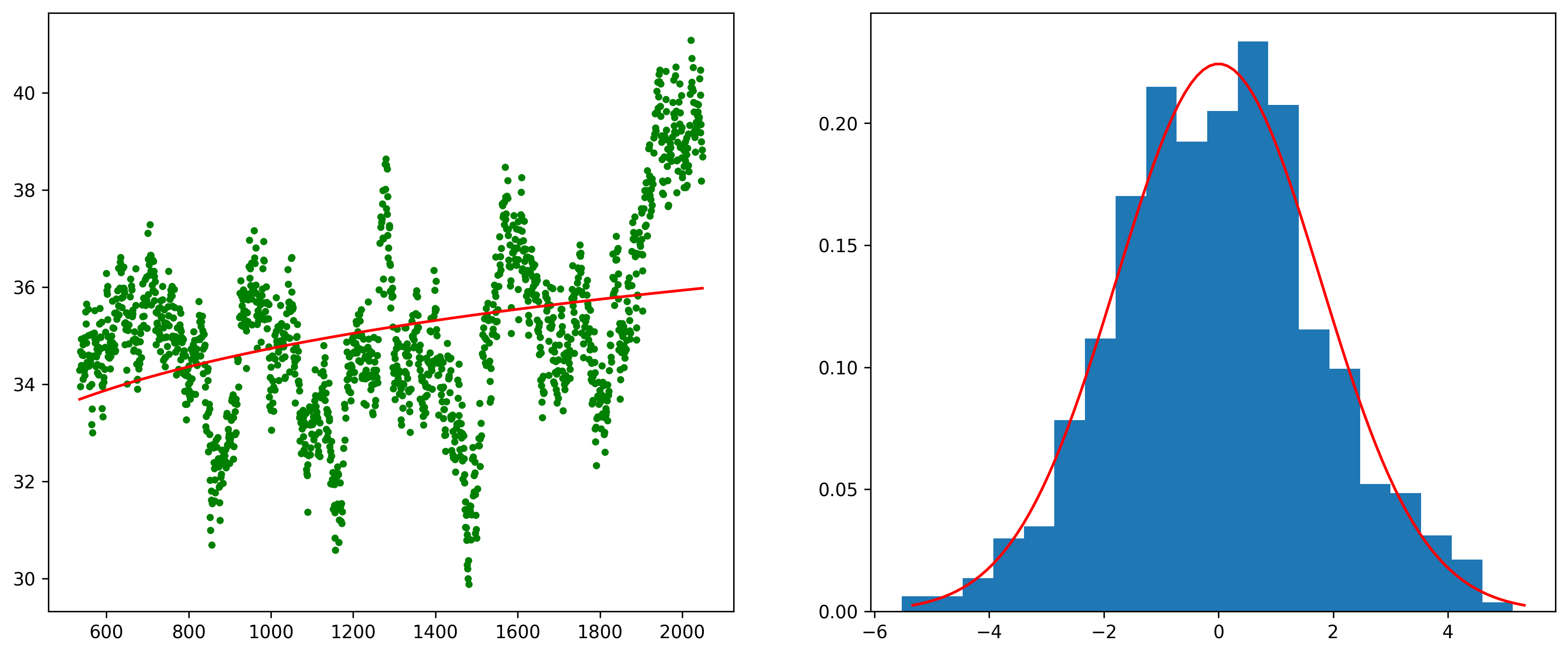}
\end{center}
\caption{Subleading order of the eigenvalue counting function in the window $\lambda \in [\boldsymbol{\lambda}_{501}, \boldsymbol{\lambda}_{2000}]$. Left: scatter plot of $c_\gamma \mu_{\gamma}(D)\lambda - \mathbf{N}_\gamma(\lambda)$ (green) versus fitted power-law curve (red). Right: histogram of deviations from the power-law curve (blue) versus fitted Gaussian density.
}
\label{fig:error}
\end{figure}

\paragraph{Delocalisation of eigenfunctions; quantum chaos.} Another natural question concerns the behaviour of eigenfunctions in the high energy (semiclassical) limit. As we increase the energy levels $\bl_n$, do the corresponding eigenfunctions $\ef_n$ typically become delocalised in the sense that their $L^2$ mass is spread out (as is the case for standard planar Brownian motion, the eigenfunctions of which are akin to sine waves with high frequency), or do they remain localised in some given region (as can happen e.g. in \textbf{Anderson localisation} owing to medium impurities)? 

\begin{figure}[h!]
\begin{center}
\includegraphics[width = .52\textwidth]{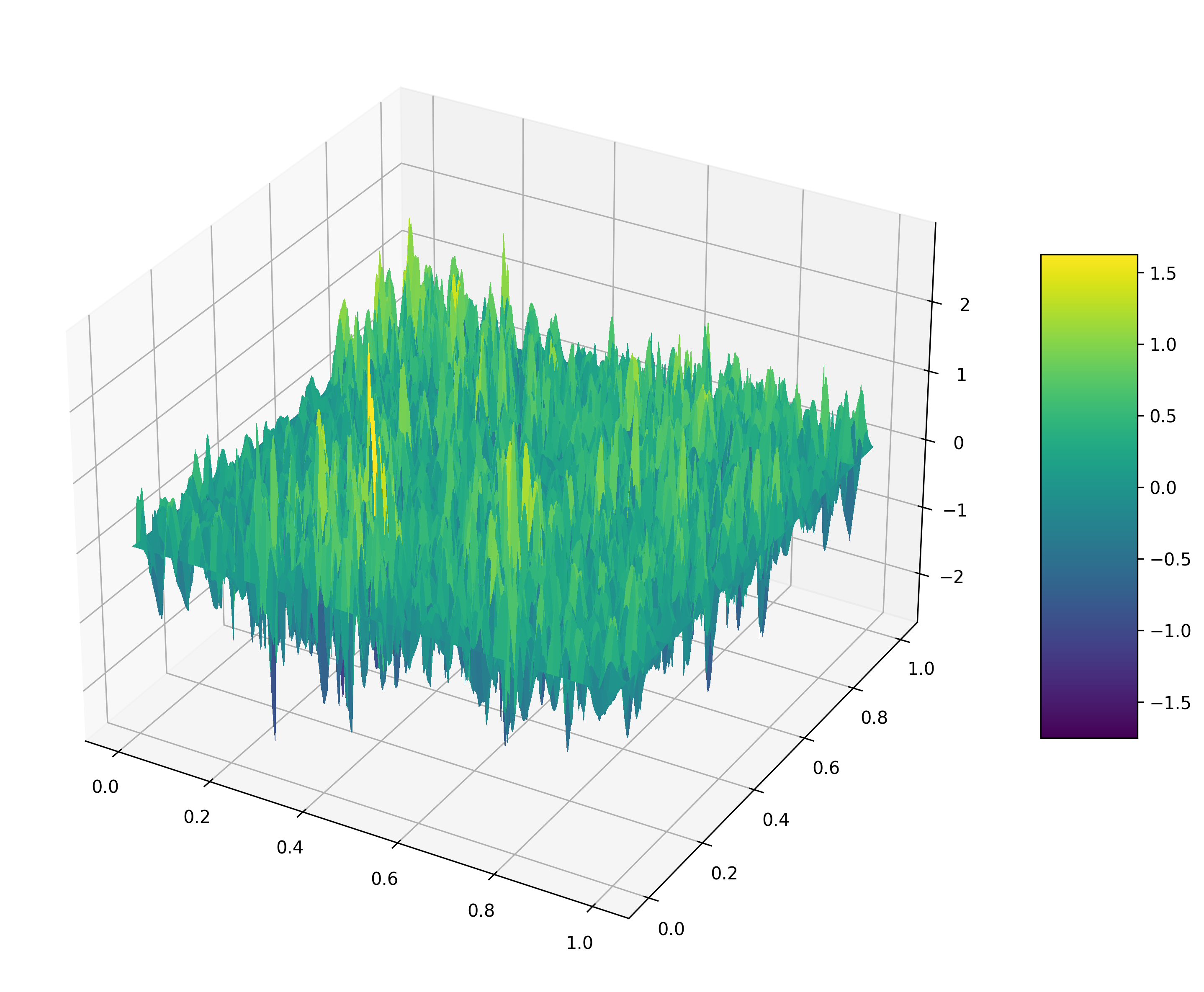}
\includegraphics[width = .47\textwidth]{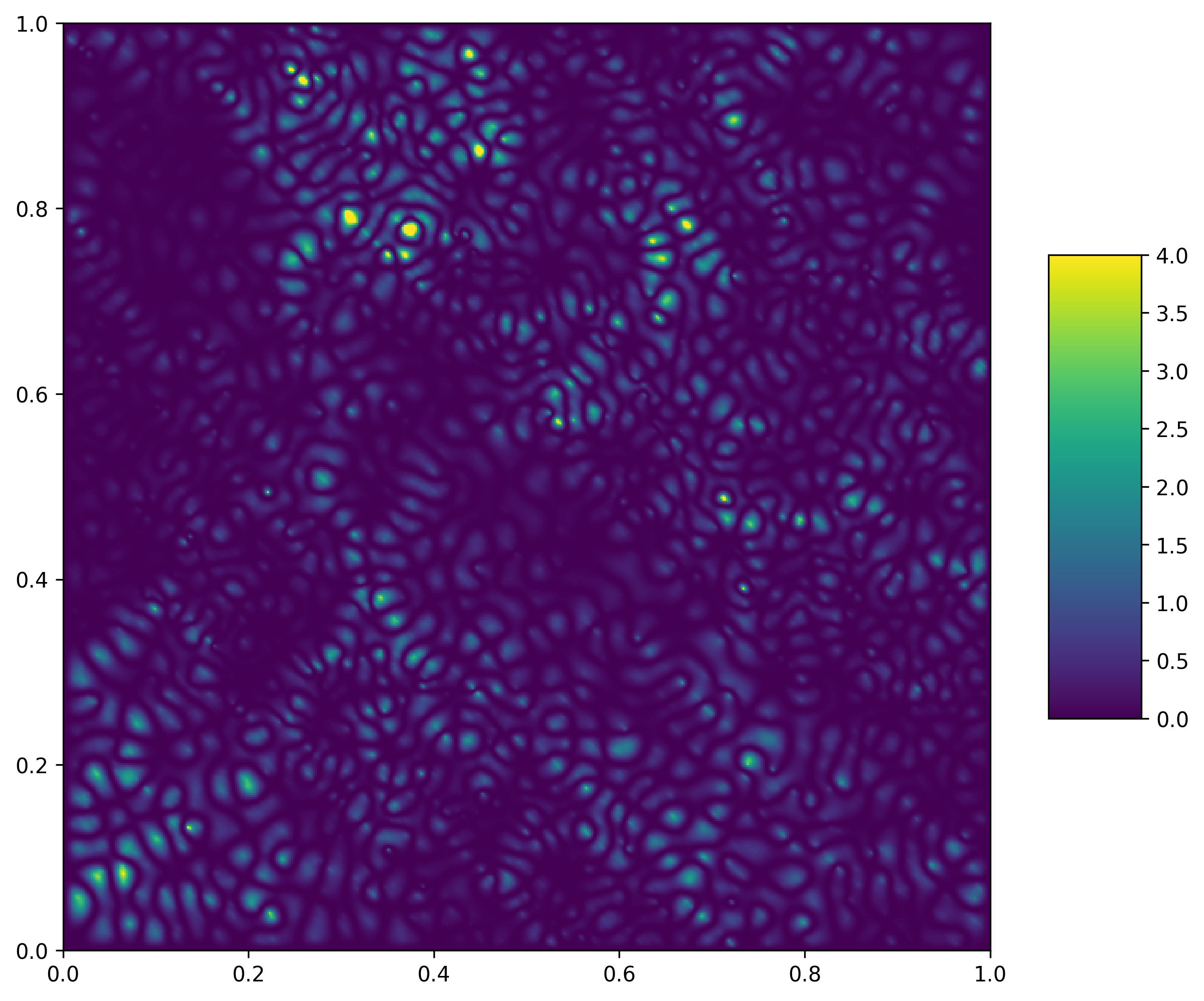}
\end{center}
\caption{Plot of the 2000th LQG eigenfunction $\mathbf{f}_{2000}$ (left) and heatmap for $|\mathbf{f}_{2000}|^2$ (right).
}
\label{F:eigenfunction}
\end{figure}

We conjecture that eigenfunctions are typically delocalised, see Figure \ref{F:eigenfunction}. In fact, by analogy with \textbf{quantum chaos} (see e.g. \cite{Ber1977}) and more precisely the celebrated \textbf{quantum unique ergodicity} conjecture of Rudnick and Sarnak  \cite{RudnickSarnak}, we make the following conjecture:

\begin{conjecture}\label{C:QUE}
 Fix $\gamma \in (0,2)$, and suppose the eigenfunctions $\ef_n$ are normalised to have unit $L^2 ( \mu_\gamma)$-norm. Then as $n \to \infty$,
$$
|\ef_n(x)|^2 \mu_\gamma(dx) \Rightarrow \frac{\mu_\gamma(dx)}{\mu_{\gamma}(D)} 
$$
 in the weak-$*$ topology in probability.
\end{conjecture}

The reason for making such a conjecture is that Liouville conformal field theory is, to the first order, a theory of random \emph{hyperbolic} surfaces, as emphasised by the fact that the ground state of the Polyakov action is given by solutions to Liouville's equation, which have constant negative curvature (see \cite{LacoinRhodesVargas}, and also \cite[Chapter 5.7]{BP}). The above can therefore be seen as an extension of the aforementioned quantum chaos conjectures to the Liouville CFT setting. \\

\paragraph{Eigenvalue spacing.} Also motivated by the literature on quantum chaos is the question of eigenvalue fluctuations. Following the Bohigas-Giannoni-Schmit conjecture on spectral statistics  \cite{BGS1984} (see also a celebrated conjecture of Sarnak \cite{Sarnak} for deterministic hyperbolic surfaces), we conjecture that level fluctuations of LQG eigenvalues should resemble those of Gaussian Orthogonal Ensemble (GOE) of random matrices (see e.g. \cite{AGZ, Meh2004} for an introduction). For instance, in the concrete example of level spacing distribution of eigenvalues, we conjecture:

\begin{conjecture}\label{C:Sar} 
For each $x \ge 0$, 
\begin{equation}\label{CSar}
\frac1{N} \sum_{j=1}^{N} 1_{\{ c_\gamma \mu_\gamma(D)(\boldsymbol{\lambda}_{j+1} - \boldsymbol{\lambda}_j) \le x\} } \xrightarrow[N\to\infty]{p} F_{\emph{GOE}}(x),
\end{equation}
where $F_{\emph{GOE}} (x)$ is the GOE \emph{level-spacing distribution}. 
\end{conjecture}

Note that the rescaled eigenvalue gap $c_\gamma \mu_\gamma(D)(\boldsymbol{\lambda}_{j+1} - \boldsymbol{\lambda}_j)$ is considered above since it is approximately  equal to $1$ on average in the long run, as established by our Weyl's law (\Cref{T:Weyl_intro}). 
The spacing distribution $F_{\text{GOE}}$, also known as Gaudin distrbution (for $\beta = 1$) in the literature, may be expressed in terms of a Fredholm determinant involving the Sine kernel \cite{Gau1961} as well as the Painlev\'e transcendents  \cite{ForresterWitte}. 
See  \Cref{fig:ev-spacing} for a comparison between the empirical LQG eigenvalue spacing distribution and our GOE conjecture.\\

\begin{figure}[h!]
\begin{center}
\includegraphics[width=0.98\textwidth]{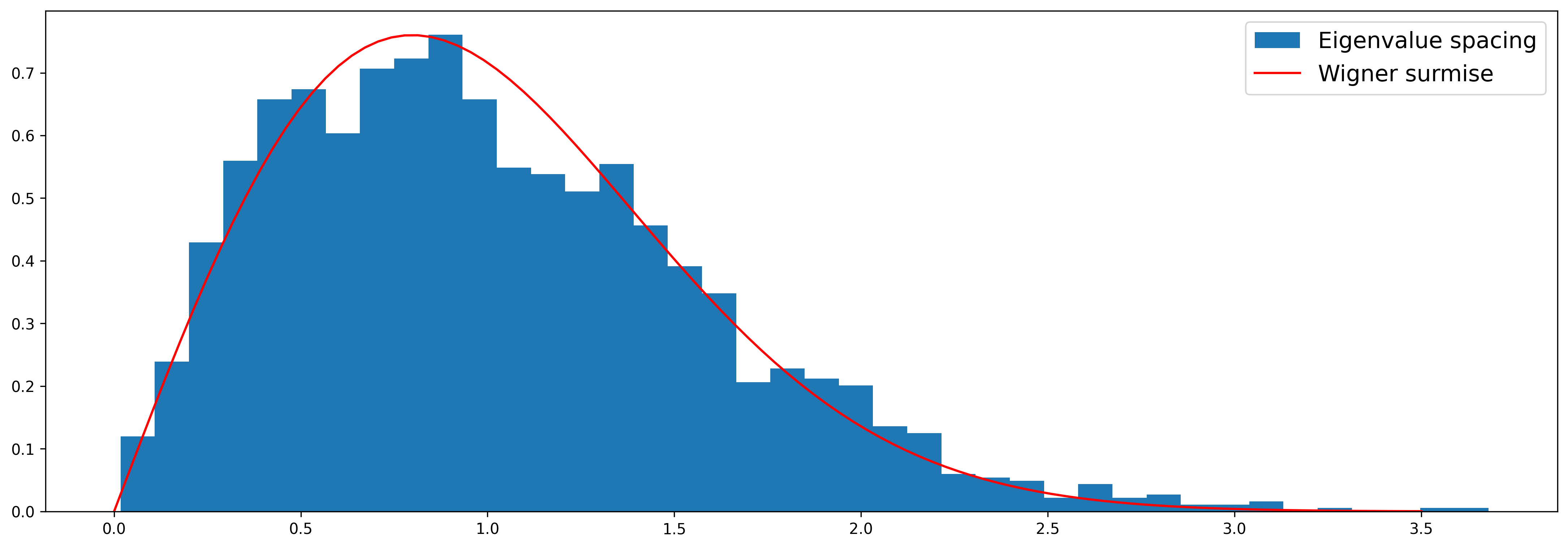}
\end{center}
\caption{Empirical spacing distribution based on the first 2000 LQG eigenvalues (blue) versus GOE statistics approximated by Wigner surmise (red).}
\label{fig:ev-spacing}
\end{figure}

\paragraph{Boundary conditions.} All the conjectures (and results in this paper) above have natural analogues for the eigenvalues of Liouville Brownian motion with Neumann, i.e. reflecting, boundary conditions, both when the underlying GFF itself has Dirichlet or Neumann boundary conditions. However we do not discuss these variants here in order to keep the paper at a reasonable length. \\

\paragraph{Random planar maps.} Likewise, we believe that these conjectures have natural analogues on {random planar maps}, for which Liouville Brownian motion is conjectured to describe the scaling limit of random walk (this has now been proved for instance for mated-CRT planar maps, see \cite{BerestyckiGwynne}). For instance, we conjecture that on a uniformly chosen triangulation with $n$ vertices, the eigenvalue sequence $(\bl_1, \ldots, \bl_n)$ associated to the discrete Laplacian (i.e., of $I - P$, where $P$ is the transition matrix of simple random walk) grows linearly with the eigenvalue level $1\le k \le n$. The linear coefficient should itself be proportional to $n^{-1/4}$ (which should correspond to the order of magnitude of the spectral gap, and to the correct scaling in order to obtain Liouville Brownian motion; see e.g., \cite{GwynneHutchcroft}, \cite{GwynneMiller_rw}) and to the constant $c_{\sqrt{8/3}} = 3/(2\pi)$ if the eigenvalues are scaled so that random walk converges to Liouville Brownian motion (note that the result of \cite{BerestyckiGwynne} involves an additional constant in the scaling, hence the chosen formulation above). Whether this linear growth should be uniform in $k$ as $n\to \infty$, or only hold as $k\ge 1$ is fixed but large and $n\to \infty$, is unclear to us at this stage. 

We also conjecture that the associated eigenfunctions $\mathbf{f}_k$ are delocalised for large $k$, and in fact approximately uniformly distributed over the planar map in an $L^2$ sense. Finally, we conjecture that the eigenvalue spacing is also given by the GOE ensemble in the limit $n\to \infty$, in agreement with 
\eqref{C:Sar}. \\

\paragraph{Critical LQG.} We end this series of conjectures on the spectral geometry of LQG by asking what (if any) of these results and conjectures become in the critical case $\gamma=2$. Note that $c_\gamma= 1/ [\pi (2 - \gamma^2/2)] \to \infty$ so it is likely that the Weyl law would require a different way of scaling the eigenvalue counting function compared to Theorem \ref{T:Weyl_intro}. 
\subsection{Short-time heat trace and heat kernel asymptotics}\label{sec:introheat}
Theorem \ref{T:Weyl_intro} may be understood from the perspective of the short-time asymptotics of the heat kernel of Liouville Brownian motion, for which we establish various results that could be of independent interest.

For points $x,y \in D$, let $\mathbf{p}_t^{\gamma, D}(x, y)$ denote the heat kernel (\cite{GRV_hk, RhodesVargas_spectral}). Recall from \cite{MRVZ} and \cite{AK2016} that there exists a jointly continuous version of the heat kernel in all three arguments $(t>0, x \in D, y \in D)$ which therefore identifies the function $\mathbf{p}_t^{\gamma, D}(x, y)$ uniquely. The heat kernel and spectrum of Liouville Brownian motion are related by the following fundamental trace formula: almost surely, for all $t>0$ and all $x, y \in D$, 
$$
\mathbf{p}_t^{\gamma, D}(x, y) = \sum_{n=1}^\infty e^{ - \boldsymbol{\lambda}_n t} \mathbf{f}_n(x) \mathbf{f}_n(y);
$$
see \cite[equation (5.10)]{AK2016}. In particular, setting $y = x$ (which is allowed since this formula holds a.s. simultaneously for \emph{all} $x,y\in D$ and $t>0$), and integrating, we obtain:
\begin{equation}\label{eq:trace}
\int_D \mathbf{p}_t^{\gamma, D}(x, x) \mu_\gamma (dx) = \sum_{n=1}^\infty e^{ - \boldsymbol{\lambda}_n t}.
\end{equation}
The integral on the left hand side is known as the \textbf{heat trace} and will be denoted in the following by $\mathbf{S}_\gamma(t;D)$. 

Note that the identity \eqref{eq:trace} implies that the heat trace  $\mathbf{S}_{\gamma}(t;D)$ is equal to the Laplace transform of the eigenvalue counting function: in other words,
\begin{align}\label{eq:lhk_counting}
\mathbf{S}_{\gamma}(t;D) := \int_0^\infty e^{-t \lambda} d\mathbf{N}_{\gamma}(\lambda)
\qquad \text{where} \quad
\mathbf{N}_{\gamma}(\lambda) := \sum_{k} 1_{\{\boldsymbol{\lambda}_k \le \lambda\}}.
\end{align}

As a consequence, using a probabilistic extension of the Hardy--Littlewood Tauberian theorem (see \Cref{theo:tauberian}),
 the behaviour of the eigenvalue counting function at high energy values is closely related to short time heat-trace asymptotics. Indeed we will obtain Theorem \ref{T:Weyl_intro} from the following result:

\begin{theo}\label{theo:lhk_laplace}
Let $\gamma \in (0, 2)$ and $A \subset D$ be any fixed open set. Denoting $\mathbf{S}_{\gamma}(t) = \mathbf{S}_{\gamma}(t; A) := \int_A \lhk_t(x, x) \mu_{\gamma}(dx)$, we have
\begin{align}\label{eq:lhk_averaged}
t \mathbf{S}_{\gamma}(t;A) \to c_{\gamma} \mu_{\gamma}(A)
\end{align}
in probability as $t \to 0^+$.
\end{theo}

\medskip \noindent \textbf{Pointwise asymptotics.} Since $A$ was an arbitrary open subset of $D$, it is natural to wonder if the asymptotics in Theorem \ref{theo:lhk_laplace} holds pointwise. In other words, if we sample $x$ from the Liouville measure $\mu_\gamma$ and fix it, does $\mathbf{p}^{\gamma, D}_t (x,x)$ behave asymptotically (in probability) like $c_\gamma /t$ as $t \to 0^+$? 

It turns out that the small-time behaviour of the heat kernel is much more subtle. We can in fact \emph{prove} that the answer to the above question is negative by establishing the following result:

\begin{theo} \label{T:HK}
Let $\gamma \in (0, 2)$. Sampling from  $\mu_\gamma$, 
$$
J_\gamma^\lambda(x) = \lambda \int_0^\infty e^{-\lambda t} t \lhk_t (x,x) dt  \xrightarrow{\lambda \to \infty} J_\gamma^\infty
$$
in distribution (where the average is over the law of the Gaussian free field $h$). Here $J_\gamma^\infty \in (0, \infty)$ is a non-constant random variable with expectation $c_\gamma$. More precisely, for any $f\in C_b(\overline{D} \times \mathbb{R}_+)$, we have
\begin{equation}\label{E:HK_Taub}
\mathbb{E}\left[ \int_D \mu_\gamma (dx) f (x, J_\gamma^\lambda(x) )  \right] \xrightarrow{\lambda \to \infty} 
\mathbb{E}\left[ \int_D \mu_{\gamma}(dx)\mathbb{E} [ f(x, J_\gamma^\infty) ] \right]
= \int_D dx R(x; D)^{\frac{\gamma^2}{2}} \mathbb{E} [ f(x, J_\gamma^\infty) ].
\end{equation}

\end{theo}

By adapting the proof of \Cref{T:HK}, one could generalise the above to a multiple-point setting and show e.g. for any $f \in C_b(\overline{D} \times \overline{D} \times \mathbb{R}_+ \times \mathbb{R}_+)$, 
\begin{align*}
& \mathbb{E}\left[\int_{D\times D} \frac{\mu_\gamma (dx)}{\mu_\gamma (D)}\frac{\mu_\gamma (dy)}{\mu_\gamma (D)} f (x, y, J_\gamma^\lambda(x), J_\gamma^\lambda(y) )  \right] \\
& \qquad \xrightarrow{\lambda \to \infty}  \mathbb{E}\left[\int_{D\times D} \frac{\mu_\gamma (dx)}{\mu_\gamma (D)}\frac{\mu_\gamma (dy)}{\mu_\gamma (D)} f (x, y, J_\gamma^\infty(x), J_\gamma^\infty(y) )  \right]
\end{align*}
where $J_\gamma^\infty(\cdot)$ are i.i.d. random variables independent of the Gaussian free field $h(\cdot)$.

We now explain why this result rules out that $ t \mathbf{p}^{\gamma, D}_t (x,x)$ converges to any constant in probability. Suppose by contradiction that 
$$
 t \mathbf{p}^{\gamma, D}_t (x,x) \to c
$$
in probability. Then by applying our probabilistic extension of the Hardy--Littlewood Tauberian theorem (see \Cref{theo:tauberian}) we would then have $J^\lambda_\gamma(x)$ converges (in probability) as $\lambda \to \infty$ to $c$. This would imply that $J^\infty_\gamma$ is the constant random variable equal to $c$, which is a contradiction. 

Note that if the Tauberian theorem (\Cref{theo:tauberian}) could be extended to cover convergence in distribution, Theorem \ref{T:HK} would imply that if we sample $x$ from the Liouville measure $\mu_\gamma(dx)$, then the distribution of $t \mathbf{p}^{\gamma, D}_t (x,x) $ converges (when averaged with respect to the law of the Gaussian free field $h$) to a nontrivial random variable. We formulate this as a conjecture:

\begin{conjecture} Let $\gamma \in (0, 2)$. Sample $x$ from Liouville measure. Then as $t\to 0^+$ and we average of the law of the Gaussian free field $h$,
$$
t \mathbf{p}^{\gamma, D}_t (x,x) \xrightarrow[t \to 0^+]{d} \xi_{\gamma}
$$
for some (non-constant) random variable $\xi_\gamma > 0$. In other words, 
$$
\E \left[ \int_D \mu_\gamma(dx) f(x, t\mathbf{p}^{\gamma, D}_t(x,x) ) \right] \to  \int_D \E [ f(x,\xi_\gamma) ] R(x;D)^{\frac{\gamma^2}{2}} dx 
$$
for any test function $f \in C_b(\overline{D} \times \mathbb{R}_+)$.
\end{conjecture}

We also believe that if we sample multiple points $x_1, \ldots, x_n$ from the Liouville measure $\mu_\gamma$ then the same convergence holds jointly with the limiting random variables $\xi_{\gamma, 1}, \ldots, \xi_{\gamma, n}$ being independent of each other as well as the Gaussian free field, similar to what we discussed after Theorem \ref{T:HK}.

\medskip Coming back to quenched heat kernel fluctuations (i.e., when we do not average over the law of the environment) Theorem \ref{T:HK} suggests that $t\mathbf{p}^{\gamma,D}_{t} (x,x)$ has considerable fluctuations. In fact we believe there are nontrivial logarithmic fluctuations in both directions (see below for further discussions).

\subsection{Previous work and our approach} 
\label{SS:mainidea}
\paragraph{Existing results for self-similar fractals}
Weyl laws in a random geometric context were derived for finitely ramified, random recursive fractals, starting in particular with the work of Hambly \cite{Hambly}.  Croydon and Hambly \cite{CH2008, CH2010} obtained similar results for the random fractals given respectively by Aldous' continuous random tree and more generally stable trees. The paper by Charmoy, Croydon and Hambly \cite{CharmoyCroydonHambly} obtained considerable refinements including Gaussian fluctuations. (We thank Takashi Kumagai for drawing our attention to these works and Ben Hambly for subsequent highly illuminating discussions). See also earlier works e.g. by Kigami and Lapidus on self-similar (non-random) fractals such as the Sierpinski gasket \cite{kigami1993weyl} where however periodicity phenomena preclude a strict Weyl asymptotics for the eigenvalues. 

Unlike the case of smooth geometries, the analysis of short-time behaviour of heat kernel is extremely challenging on fractals. The best one might hope with current technology is a two-sided sub-Gaussian bound on $\mathbf{p}^{\gamma, D}_t (x,y)$, but obviously such estimates will not identify the leading order coefficient in Weyl laws. This is further complicated by the fact that the heat kernel is expected to exhibit non-trivial fluctuations on the diagonal (see e.g. \cite{Kaj2013} for results concerning p.c.f. fractals), and thus any approaches that require short-time asymptotic expansions of the heat kernel are bound to be infeasible even for the weaker problem of identifying the correct order of magnitude (say, up to constant) of the heat trace.

Therefore, instead of using the trace formula, the aforementioned works investigated the spectral problems via the classical Dirichlet-Neumann bracketing method. Essentially, one performs a multi-scale decomposition of the domain and derives the asymptotics for the associated eigenvalue counting function or heat trace using techniques from renewal theory (where one renewal corresponds to changing scale). Despite its power and elegance, the renewal framework does not seem applicable to LQG as it relies heavily on a strong form of independence across scales which is not present in the context of LQG. Moreover, quantitative control of the difference between Dirichlet and Neumann eigenvalue counting functions is crucial for the application of the bracketing method. Unfortunately these estimates are not available beyond the class of finitely ramified fractals (with the only exception of Sierpinski carpets), and a very different approach is needed for the spectral analysis of LQG surfaces.

\paragraph{Existing results for Liouville Brownian motion}
Not much is known about the spectral geometry of LQG surfaces. Prior to our work, the only available result in this direction is that the spectral dimension is equal to two: with probability $1$, we have
\begin{align*}
    \lim_{t \to \infty} \frac{2\log\mathbf{p}^{\gamma, D}_t (x,x)}{-\log t} = 2 \qquad \text{for $\mu_\gamma$-a.e. $x\in D$}.
\end{align*}

\noindent This behaviour was predicted by Ambj\o rn et al. \cite{ABNRW1998} in the physics literature, and first established by Rhodes and Vargas in \cite{RhodesVargas_spectral}. It is interesting to note that \cite[Remark 3.7]{RhodesVargas_spectral} suggested that one might investigate the convergence of $t\mathbf{p}^{\gamma, D}_t (x,x)$ to some random variable as $t \to 0^+$, which we have now shown is impossible (in the sense of convergence in probability) as a consequence of our \Cref{T:HK}.

The challenging problem of obtaining pointwise estimates for different variants of Liouville heat kernel was also explored in the work of \cite{AK2016} and \cite{MRVZ}.
For our setting, \cite[Theorems 1.2 and 1.3]{AK2016} led to another proof of the spectral dimension. The same paper also provided some estimates for the logarithmic corrections, and further explained why one could not hope for the complete removal of such corrections. 

\medskip As such, even the weaker goal of strengthening $\mathbf{S}_{\gamma}(t) = t^{-1+o(1)}$ to the tightness of $t\mathbf{S}_{\gamma}(t)$ (as $t \to 0^+$) presents very serious difficulties requiring new insights. Our main theorems are thus a significant improvement over existing results in both the LQG and fractal literature in that:
\begin{itemize}
    \item we are the first to establish not only the tightness of the rescaled LQG heat trace, but also convergence of the leading order coefficient, all achieved without the bracketing method; and
    
    \item we identify the leading order coefficient explicitly, including the formula for the special constant $c_\gamma$ and its relation to the on-diagonal behaviour of the heat kernel, all of which would not have been possible even if the renewal techniques had been applicable.
\end{itemize}

To the best of our knowledge, our paper is the \textbf{first successful application of the trace formula} in a random geometric context where self-similarity and independence are absent, and we now explain at a high level the novelty of our analysis.

\paragraph{Main idea}
Let us focus on the proof of Theorem \ref{theo:lhk_laplace}, and recall our goal of establishing (for fixed open subset $A \subset D$)
\begin{equation}\label{eq:goalintro}
    \mathbf{S}_{\gamma}(t) 
    = \mathbf{S}_{\gamma}(t; A) 
    :=  \int_A \lhk_t(x, x) \mu_{\gamma}(dx)
    \overset{t \to 0}{\sim} \frac{c_{\gamma} \mu_{\gamma}(A) }{t}
\end{equation}
in probability (where $a_n \sim b_n$ in probability means $a_n /b_n \to 1$ in probability as $n \to \infty$). 

\medskip Given that fine estimates for Liouville heat kernel are out of reach with standard machinery as we discussed just now, it is very difficult to have a direct handle on $\mathbf{S}_{\gamma}(t)$ at a fixed time $t>0$. Instead, we take advantage of the fact that Liouville Brownian motion is a time-change of ordinary Brownian motion. This leads us naturally to try to establish a suitable `integrated asymptotics': that is, we seek to establish a form of \eqref{eq:goalintro} where we integrate with respect to time.

\medskip It turns out that this integrated asymptotics is equivalent to \eqref{eq:goalintro}. It should be noted, however, that the equivalence between the pointwise and integrated probabilistic asymptotics should not be seen as immediate consequence of deterministic counterparts (i.e., Tauberian theorems). Indeed extra considerations are needed because our probabilistic asymptotics only hold in the sense of convergence in probability and not almost surely (see \Cref{app:probasy}). 

\medskip At the heart of our proof of the integrated asymptotics is the bridge decomposition (\cite{RhodesVargas_spectral}, \cite{HKPZ}; see below for more details) which relates time integrals involving the Liouville heat kernel to their Euclidean counterparts. This exploits the fact that Liouville Brownian motion is a time-change of ordinary Brownian motion (a feature of conformal invariance) and lies behind the ``solvability'' (including computation of leading constants) of our results.

\medskip Let us give a few more details about the integrated version of  \eqref{eq:goalintro} we consider. As a first guess, one may naively consider a quantity such as $\int_t^1 \mathbf{S}_\gamma(s) ds$. In that case it would appear that the first step would be to prove that this blows up logarithmically as $t\to 0$ with a proportionality constant dictated by \eqref{eq:goalintro} and then try to apply Tauberian theory. Unfortunately, this logarithmic behaviour falls precisely outside the scope of the most classical results in Tauberian theory even in the deterministic case: one would instead need to appeal to so-called de Haan theory (see e.g. \cite[Chapter IV.6]{Kor2004}), which is only applicable if one has good control over the subleading order terms, and this is out of question in our setting.

\medskip Luckily, there is a simple solution around this. 
Since  $t \mapsto \mathbf{S}_{\gamma}(t)$ is monotone in $t$, it suffices (see Lemma \ref{lem:asympdiff} for a proof) 
to establish the integrated asymptotics in probability
\begin{align}\label{eq:2nd_trans}
\int_0^t u \mathbf{S}_{\gamma}(u) du &\sim c_{\gamma} \mu_{\gamma}(A) t
\qquad \text{as} \qquad t \to 0^+,
\end{align}
(note that the multiplication by $u$ in the integral in the left hand side effectively changes the index of regular variation). Equivalently, by the probabilistic extension of the Tauberian theorem (see Theorem \ref{theo:tauberian}), it suffices to prove
\begin{align}
\label{eq:2nd_tau}
\int_0^\infty e^{-\lambda u} u\mathbf{S}_{\gamma}(u) du &\sim \frac{c_\gamma \mu_{\gamma}(A)}{\lambda}
\qquad \text{as} \qquad \lambda \to \infty
\end{align}
in probability. As already alluded to above, a key tool for obtaining \eqref{eq:2nd_tau} is the following bridge decomposition (\cite{RhodesVargas_spectral}, \cite{HKPZ}):

\begin{lem}\label{L:bbDec}
For any measurable $f: [0, \infty) \to [0, \infty)$, we have
\begin{align}\label{eq:bridge_dec}
\int_0^\infty f(t) \lhk_t(x, y) dt
= \int_0^\infty \bdec{x}{y}{t}[f(F_{\gamma}(\mathbf{b}))1_{\{t < \tau_D(\mathbf{b})\}}] p_t(x, y) dt
\end{align}

\noindent where

\begin{itemize}[leftmargin=*]
\item $\bdec{x}{y}{t} = $ law of Brownian bridge $(\mathbf{b}_s)_{s \le t}$ of duration $t$ from $x$ to $y$ (without killing); and 
\item for any process $\mathbf{b}$ defined on $\mathbb{R}^2$ with starting position $\mathbf{b}_0 \in D$ and duration $\ell = \ell(\mathbf{b})$:
\begin{itemize}
\item $\tau_D(\mathbf{b}) := \inf\{t > 0: \mathbf{b}_t \in \partial D\}$,
\item $F_{\gamma}(\mathbf{b})$ is the Liouville clock associated to $\mathbf{b}$, i.e. $F_\gamma(\mathbf{b}) := \int_0^\ell F_\gamma(ds; \mathbf{b})$ with
\begin{align}\label{eq:liouville_clock}
F_{\gamma}(ds; \mathbf{b}) := e^{\gamma h(\mathbf{b}_s) - \frac{\gamma^2}{2} \mathbb{E}[h(\mathbf{b}_s)^2]} R(\mathbf{b}_s; D)^{\frac{\gamma^2}{2}} 1_{\{\mathbf{b}_s \in D\}}ds,
\end{align}
\end{itemize}
\item $p_t(x, y)$ is the transition density of standard 2-dimensional Brownian motion (in particular $p_u(x,x) = 1/ (2\pi u)$ for any $u > 0$ and $x \in D$).
\end{itemize}
\end{lem}

The bridge decomposition as stated in \cite{RhodesVargas_spectral, HKPZ} has slightly different assumptions (e.g., we need here to restrict to trajectories remaining inside of $D$, which was not the case in \cite{RhodesVargas_spectral, HKPZ}, but the proof is straightforward to adapt. 
Using Lemma \ref{L:bbDec}, the left-hand side of \eqref{eq:2nd_tau} can be rewritten as
\begin{align}
\label{eq:bridge_tau}
\int_0^\infty e^{-\lambda u} u\mathbf{S}_{\gamma}(u) du 
& = \int_A \mu_{\gamma}(dx) \int_0^\infty \bdec{x}{x}{u}\left[F_{\gamma}(\mathbf{b}) e^{-\lambda F_{\gamma}(\mathbf{b})}1_{\{u < \tau_D(\mathbf{b})\}}\right] p_u(x, x) du.
\end{align}

\noindent Thus \Cref{theo:lhk_laplace} will follow from the following result:
\begin{theo}\label{theo:Weyl}
Let $A \subset D$ be a fixed open subset of $D$ and $\gamma \in (0, 2)$. Then
\begin{align*}
\lim_{\lambda \to \infty}\int_A \mu_{\gamma}(dx) \int_0^\infty \frac{du}{2\pi u} \bdec{x}{x}{u}[ \mathcal{I} \left(\lambda  F_{\gamma}(\mathbf{b})\right)1_{\{u< \tau_D(\mathbf{b}) \}}] = c_{\gamma} \mu_{\gamma}(A)
\end{align*}

\noindent in the sense of $L^1$-convergence (and hence convergence in probability).
\end{theo}

Unlike in \cite{RhodesVargas_spectral} where the bridge decomposition was used to deduce $\mathbf{S}_{\gamma}(t) = t^{-1 + o(1)}$ from crude moment estimates for the mass of $\mu_\gamma$ (which would not have been sufficient for the removal of $o(1)$ error in the exponent), the proof of Theorem \ref{theo:Weyl} requires a very refined analysis of the short-time (i.e. small $u$) behaviour of the random variables $\bdec{x}{x}{u}[ \mathcal{I} \left(\lambda  F_{\gamma}(\mathbf{b})\right)1_{\{u< \tau_D(\mathbf{b}) \}}]$ simultaneously for all $x \in D$. At a high level, we shall draw inspiration from the thick points approach described in \cite{Ber2017}, however the details are of course much more technical. In particular we develop a general method for handling correlations of possibly different functionals applied to small neighbourhoods in the vicinity of Liouville typical points; see Lemma \ref{lem:main} for a statement and Section \ref{sec:proofWeyl} for a proof of Theorem \ref{theo:Weyl}.

\subsection{Outline of the paper.}

We start in Section \ref{S:prelim} with some preliminaries on Gaussian comparison, estimates for the Green's function and decomposition results both for the GFF and Brownian motion, which leads us to the main lemma (Lemma \ref{lem:main}). 

Section \ref{sec:proofWeyl} contains the proof of the main results of this paper, namely Theorem \ref{theo:Weyl}. We start in Proposition \ref{prop:toy} with a quick and simple illustration for how the main lemma is used throughout the paper by computing ``one-point estimates'' with it. We end with a very brief description of how Theorem \ref{theo:Weyl} implies both Theorem \ref{theo:lhk_laplace} and Theorem \ref{T:Weyl_intro}.

Section \ref{sec:proof_ptwise} gives the proof of Theorem \ref{T:HK} that pertain to the pointwise asymptotics of the heat kernel (as opposed to the heat trace asymptotics which are at the heart of Theorem \ref{theo:Weyl}, and which involve by definition a spatially averaged heat kernel asymptotics). The identification of the limiting constant in Theorem \ref{T:Weyl_intro}, i.e. Theorem \ref{theo:constant}, is also proved in that section. 

Finally, Appendix \ref{app:probasy} contains probabilistic extensions of results from asymptotic analysis, namely ``asymptotic differentiation under the integral sign'' (Lemma \ref{lem:asympdiff})  and Tauberian theorem (Theorem \ref{theo:tauberian}).

\paragraph{Notations.} For the readers' convenience we list a few crucial notations below which are used repeatedly in the main proofs in \Cref{sec:proofWeyl} and \Cref{sec:proof_ptwise}, and provide pointers to their defining equations.
\begin{itemize}
\item $F_{\gamma}(\mathbf{p})$ and $F_\gamma(ds; \mathbf{p})$: Liouville clock associated to the path $\mathbf{p}$, see \eqref{eq:liouville_clock}.
\item $F_\gamma^{\mathcal{S}}(\mathbf{p})$: Liouville clock with insertions in $\mathcal{S}$, see \eqref{eq:Girsanov1}; when $\mathcal{S} = \emptyset$ this coincides with the previous definition.
\item $\mathcal{G}_I^{\mathcal{S}}(p)$: `good event' concerning the thickness of Gaussian free field at $p \in D$ at dyadic levels in $I$, see \eqref{eq:eventGir}; when $\mathcal{S} = \emptyset$ we suppress its dependence in the notation.
\item $\overline{F}_\gamma^{\mathcal{S}}(\mathbf{p}; Y)$: random clock associated to the path $\mathbf{p}$ with respect to background field $Y$ and insertions in $\mathcal{S}$, see \eqref{eq:overlineF}.
\end{itemize}

\paragraph{Acknowledgement.} Part of the work was carried out when both authors were in residence at the Mathematical Sciences Research Institute in Berkeley for the semester programme `The Analysis and Geometry of Random Spaces' in 2022, and we wish to thank the Institute for its hospitality and support (NSF grant 1440140). NB's research is supported by FWF grant P33083 on ``Scaling limits in random conformal geometry''.

\section{Preliminaries}\label{S:prelim}
\subsection{Gaussian comparison}

\begin{lem}\label{lem:Gcompare}
Let $X(\cdot)$ and $Y(\cdot)$ be two continuous centred Gaussian field on $D$, $\rho$ be a Radon measure on $D$, and $P: \mathbb{R}_+ \to \mathbb{R}$ be a smooth function with at most polynomial growth at infinity. For $t \in [0, 1]$, define $Z_t(x) := \sqrt{t} X(x) + \sqrt{1-t} Y(x)$ and
\begin{align*}
\varphi(t):= \mathbb{E}[P(M_t)], \qquad
M_t := \int_D e^{Z_t(x) - \frac{1}{2}\mathbb{E}[Z_t(x)^2]}\rho(dx).
\end{align*}

\noindent Then
\begin{align*}
\varphi'(t) & =
\frac{1}{2}\int_D \int_D \left(\mathbb{E}[X(x)X(y)] - \mathbb{E}[Y(x) Y(y)]\right)\\
& \qquad \times \mathbb{E}\left[e^{Z_t(x) + Z_t(y) - \frac{1}{2}\mathbb{E}[Z_t(x)^2] - \frac{1}{2}\mathbb{E}[Z_t(y)^2]}P''(M_t)\right] \rho(dx) \rho(dy).
\end{align*}

\noindent In particular, if there exists some constant $C>0$ such that
\begin{align*}
\left| \mathbb{E}[X(x)X(y)] - \mathbb{E}[Y(x) Y(y)]\right| \le C \qquad \forall x, y \in D,
\end{align*}

\noindent then
\begin{align*}
\left|\varphi(1) - \varphi(0)\right|
\le \frac{C}{2} \int_0^1 \mathbb{E}\left[M_t^2 |P''(M_t)|\right] dt.
\end{align*}

\end{lem}

\begin{cor} \label{cor:Gcompare}
Using the same notations as in \Cref{lem:Gcompare}, suppose
$\mathbb{E} [ X(x) X(y)] \le \mathbb{E}[Y(x) Y(y)]$ and $P$ is convex, then $\varphi(1) \le \varphi(0)$, i.e., $\mathbb{E} [ P(M_0)] \le \mathbb{E}[ P(M_1)]$.
\end{cor}

\subsection{Estimates for Brownian bridge}
Let $\mathbf{b}_\cdot = (\mathbf{b}_{\cdot, 1}, \mathbf{b}_{\cdot, 2})$ be a $2$-dimensional Brownian bridge with starting position $\iota(\mathbf{b}):= \mathbf{b}_0$ and duration $\ell(\mathbf{b})$ (e.g. $\iota(\mathbf{b}) = x$ and $\ell(\mathbf{b}) = u$ if $\mathbf{b} \sim \pdec{x}{x}{u}$). We recall the following formula for the distribution of running maximum of one-dimensional Brownian bridge:

\begin{lem}\label{lem:bb_exact}
For $i \in \{1, 2\}$ and any $k \ge 0$,
\begin{align}\label{eq:bb_exact}
\pdec{0}{0}{\ell}
\left(\max_{s \le \ell} \mathbf{b}_{s, i} \ge k \right)
= e^{-\frac{2}{\ell}k^2} \qquad \forall k \ge 0.
\end{align}
\end{lem}

The exact formula \eqref{eq:bb_exact} leads to the following inequalities which we shall use repeatedly throughout this article:
\begin{cor}\label{cor:bb_bound}
For any $u > 0$,
\begin{align*}
\pdec{0}{0}{\ell} \left( \max_{s \le \ell} |\mathbf{b}_s| \le u \right) 
& \le 1 \wedge \frac{2u^2}{\ell}\\
\text{and}
\qquad 
\pdec{0}{0}{\ell} \left( \max_{s \le \ell} |\mathbf{b}_s| \ge u \right) 
& \le 4 e^{-\frac{u^2}{2l}}.
\end{align*}

\end{cor}
\begin{proof}
The two inequalities follow from
\begin{align*}
\pdec{0}{0}{\ell} \left( \max_{s \le \ell} |\mathbf{b}_s| \le u \right) 
\le \pdec{0}{0}{\ell} \left( \max_{s \le \ell} \mathbf{b}_{s, 1} \le u \right)
= 1 - e^{-2u^2 / l},
\end{align*}

\noindent and 
\begin{align*}
\pdec{0}{0}{\ell} \left( \max_{s \le \ell} |\mathbf{b}_s| \ge u \right) 
\le 2\pdec{0}{0}{\ell} \left( \max_{s \le \ell} |\mathbf{b}_{s, 1}| \ge \frac{u}{2} \right)
\le 4\pdec{0}{0}{\ell} \left( \max_{s \le \ell} \mathbf{b}_{s, 1} \ge \frac{u}{2} \right)
= 2e^{-\frac{u^2}{2l}}
\end{align*}

\noindent by \Cref{lem:bb_exact}.
\end{proof}

\subsection{Estimates for Green's function}
\begin{lem}\label{lem:Green_estimate}
Suppose $D$ is a bounded domain with at least one regular point on $\partial D$. Then the following estimates hold for our Green's function $G_0^D(\cdot, \cdot)$.
\begin{itemize}
\item For any $x, z \in D$ satisfying $|x-z| \le \frac{1}{3} d(x, \partial D)$, we have
\begin{align}\label{eq:Green_uniform_local}
\bigg| G_0^D(x, z) - \left[ -\log|x-z| + \log R(x; D)\right]\bigg| \le 6\frac{|x-z|}{d(x, \partial D)} \log \frac{R(x; D)}{d(x, \partial D)}.
\end{align}

\item For any $x, y, z \in D$ satisfying $d(x, z) \le \min(|x-y|, d(x, \partial D))$,
\begin{align}\label{eq:Green_uniform_global}
\bigg| G_0^D(z, y) - G_0^D(x, y)\bigg| 
\le 2\left[ \frac{|x-z|}{d(x, \partial D)} + \frac{|x-z|}{|x-y|}\right].
\end{align}
\end{itemize}
\end{lem}

\begin{proof}
For the first estimate, it suffices to consider the case where $x = 0$ and $d(x, \partial D) = 1$ by translation and rescaling. But then
\begin{align*}
\left|G_0^D(0, z) - [-\log|z| + \log R(0; D)]\right|
\le \frac{1}{\pi} \int_0^{2\pi} G_0^D(0, e^{i\theta}) \left|H_{\mathbb{D}}(z, e^{i\theta}) - H_{\mathbb{D}}(0, e^{i\theta})\right|d\theta.
\end{align*}

\noindent Using the fact that
\begin{align*}
\frac{1}{2\pi} \int_0^{2\pi} G_0^D(0, e^{i\theta})  d\theta =  \log R(0; D)
\end{align*}

\noindent and the explicit formula for the Poisson kernel on the unit disc $\mathbb{D}$
\begin{align*}
H_{\mathbb{D}}(z, e^{i\theta})
= \frac{1}{2} \frac{1-|z|^2}{|e^{i\theta} - z|^2}, \qquad |z| < 1,
\end{align*}

\noindent we obtain the upper bound \eqref{eq:Green_uniform_local} by a direct computation.\\

For the second estimate, we recall the probabilistic representation of the Green's function
\begin{align*}
G_0^D(\cdot, y) = \mathbb{E}^y \left[\log|W_{\tau_D} - \cdot|\right] - \log|\cdot - y|
\end{align*}

\noindent where $(W_t)_{t \ge 0}$ is a (planar) Brownian motion starting from $y \in D$ (with respect to the probability measure $\mathbb{P}^y)$ and $\tau_D$ is its hitting time of $\partial D$. Then \eqref{eq:Green_uniform_global} can be verified directly using the elementary inequality $\log |1+x| \le 2|x|$ for any $|x| \le \frac{1}{2}$.

\end{proof}

We state a useful consequence of the above estimate.
\begin{cor}\label{cor:gff_exact}
Let $a, b, x \in D$ be such that $\max (|x-a|, |x-b|) \le \frac{1}{4} d(x, \partial D)$. Then
\begin{align*}
\left|G_0^D(a, b) -  \left[ -\log|a-b| + \log R(x; D)\right]\right| \le 4.
\end{align*}

\noindent In particular, for any $z \in B(x, \frac{1}{4} d(x, \partial D))$, we have
\begin{align*}
| \log R(z; D) - \log R(x; D)| \le 4.
\end{align*}

\end{cor}

\begin{proof}
Let $\mathbb{E}^a$ be the expectation with respect to a planar Brownian motion $(W_t)_{t \ge 0}$ starting from $a \in D$, and $\tau_D :=\{ t > 0: W_t \in \partial D\}$. Then
\begin{align*}
&\left|G_0^D(a, b) -  \left[ -\log|a-b| + \log R(x; D)\right]\right|\\
& = \left|\mathbb{E}^a \left[ \log |W_{\tau_D} - b| \right] - \log R(x; D) \right|\\
& \le \left|\mathbb{E}^a \left[ \log |W_{\tau_D} - x| \right] - \log R(x; D)\right| +  \left|\mathbb{E}^a \left[ \log |W_{\tau_D} - b| \right] -  \mathbb{E}^a \left[ \log |W_{\tau_D} - x| \right]\right|\\
& = \left|G_0^D(a, x) -  \left[ -\log|a-x| + \log R(x; D)\right]\right| + \left|\mathbb{E}^a \left[ \log \left| \frac{(W_{\tau_D} - x) + (x - b)}{ |W_{\tau_D} - x| }\right| \right]\right|.
\end{align*}

\noindent Using \eqref{eq:Green_uniform_local} and Koebe quarter theorem, we have
\begin{align*}
 \left|G_0^D(a, x) -  \left[ -\log|a-x| + \log R(x; D)\right]\right|
&\le 6 \frac{|x-a|}{d(x, \partial D)}\log \frac{R(x; D)}{d(x, \partial D)}\\
& \le 6 \frac{1}{4} \log 4 \le 3,
\end{align*}

\noindent whereas the elementary inequality  $|\log |1 + x| | \le 2|x|$ for any $|x| \le \tfrac{1}{2}$ implies
\begin{align*}
 \left|\mathbb{E}^a \left[ \log \left| \frac{(W_{\tau_D} - x) + (x - b)}{ |W_{\tau_D} - x| }\right| \right]\right|
\le 2 \frac{|x-b|}{d(x, \partial D)} \le 1
\end{align*}

\noindent which gives the desired claim.
\end{proof}

\begin{lem}[{cf. \cite[Lemma 3.5]{Ber2017}}]
\label{lem:mollified_cov}
For each $r > 0$, let $h_r(\cdot)$ be the circle average of the Gaussian free field over $\partial B(\cdot, r)$. Then for any $\epsilon, \delta > 0$, 
\begin{align*}
\mathbb{E} \left[ h_\epsilon(x) h_\delta(y)\right]
= -\log \left(|x-y| \vee \epsilon \vee \delta\right) + \mathcal{O}(1)
\end{align*}

\noindent where the $\mathcal{O}(1)$ error is uniform for all $x, y \in D$ bounded away from $\partial D$.
\end{lem}

\subsection{Decomposition of Gaussian free field}
Let us mention the following decomposition of Gaussian free field, which will play a crucial role in the proof of \Cref{T:HK}.

\begin{lem}\label{lem:GFF_dec}
Let $\kappa \in (0, 1]$. Then on some suitable probability space we can construct simultaneously three Gaussian fields $h^{\kappa \mathbb{D}}$, $X^{\kappa \mathbb{D}}$ and $\mathcal{G}^{\kappa \mathbb{D}}$ such that
\begin{align}\label{eq:GFF_dec}
h^{\kappa \mathbb{D}} (\cdot) = X^{\kappa \mathbb{D}}(\cdot) - Y^{\kappa \mathbb{D}}(\cdot) \qquad \text{on $B(0, \kappa)$}
\end{align}

\noindent where
\begin{itemize}
\item $h^{\kappa \mathbb{D}}$ is a Gaussian free field on $B(0, \kappa)$ with Dirichlet boundary condition;
\item $X^{\kappa \mathbb{D}}$ is the exactly scale invariant field with covariance given by $\mathbb{E}[X^{\kappa \mathbb{D}}(x) X^{\kappa \mathbb{D}}(y)] = -\log|x-y| + \log \kappa$ on $B(0, \kappa)$.
\item $Y^{\kappa \mathbb{D}}(\cdot)$ is a Gaussian field on $B(0, \kappa)$ independent of $h$, and is uniformly continuous when restricted to compact subset of $B(0, \kappa)$; moreover $Y^{\kappa \mathbb{D}}(0) = 0$.
\end{itemize}
\end{lem}

\begin{proof}
Since
\begin{align*}
h^{\kappa \mathbb{D}}(\cdot) \overset{d}{=} h^{\mathbb{D}}(\cdot / \kappa)
\qquad \text{and} \qquad 
X^{\kappa \mathbb{D}}(\cdot) \overset{d}{=} X^{\mathbb{D}}(\cdot / \kappa)
\end{align*}

\noindent on $B(0, \kappa)$, the general result follows from the special case $\kappa = 1$ using a scaling argument.

Let us now focus on $\kappa = 1$, and view $\mathbb{D} \subset \mathbb{C}$. Recall that
\begin{align*}
\mathbb{E}[X^{\mathbb{D}}(x) X^{\mathbb{D}}(y)]
&= - \log|x-y|\\
& = -\log \left| \frac{x-y}{1 - x\bar{y}}\right| - \log|1-x\bar{y}|
= G_0^{\mathbb{D}}(x, y) - \log|1-x\bar{y}| \qquad \forall x, y \in \mathbb{D}.
\end{align*}

We claim that the kernel $-\log|1-x\bar{y}|$ is positive definite on $\mathbb{D} \times \mathbb{D}$ and therefore could be realised as the covariance kernel of some Gaussian field $Y^{\mathbb{D}}$: indeed the field can be explicitly constructed by
\begin{align}\label{eq:GaussAnalytic}
Y^{\mathbb{D}}(z) :=  \Re \left[ \sum_{k=1}^\infty \sqrt{\frac{2}{k}} \mathcal{N}^{\mathbb{C}}_kz^k\right], \qquad z \in \mathbb{D}
\end{align}

\noindent where $\mathcal{N}^{\mathbb{C}}_k$ are i.i.d. standard complex Gaussian random variables. We can then construct a Gaussian free field $h^{\mathbb{D}}$ independent of $Y^{\mathbb{D}}$ and set $X^{\mathbb{D}} := h^{\mathbb{D}} + Y^{\mathbb{D}}$ so that \eqref{eq:GFF_dec} holds by definition.

Last but not least, since $Y^{\mathbb{D}}(z)$ is the real part of a random analytic function with radius of convergence equal to $1$, it follows immediately that $Y^{\mathbb{D}}(z)$ is uniformly continuous when restricted to any compact subset of $\mathbb{D}$, and substituting $z = 0$ into \eqref{eq:GaussAnalytic} we have $Y^{\mathbb{D}}(0) = 0$ almost surely, as claimed.
\end{proof}

\subsection{Williams' path decomposition of Brownian motion}\label{subsec:path_dec}
The following result is due to Williams \cite{Wil1974}; see also \cite{RP1981}.
\begin{lem}\label{lem:time_reversal}
Let $(B_t)_{t \ge 0}$ be a Brownian motion, and for $m> 0$ write $B_t^m := B_t + m t$. Fix $x > 0$ and define
\begin{align*}
\tau_x := \inf \{t > 0: B_t^m = x\}.
\end{align*}

\noindent Then we have the following equality of path distributions
\begin{align*}
(x - B_{\tau_x - t}^m)_{t \in [0, \tau_x]} \overset{d}{=} (\mathcal{B}_t^m)_{t \in [0, L_x]}
\end{align*}

\noindent where $(\mathcal{B}_t^m)_{t \ge 0}$ is a Brownian motion with drift $m$ conditioned to stay non-negative, and
\begin{align*}
L_x := \sup \{t > 0: \mathcal{B}_t^\mu = x\}.
\end{align*}
\end{lem}

The following definition will be used in \Cref{sec:proof_ptwise} of the article: for each $m > 0$ we define the two-sided process $(\beta_t^m)_{t \in \mathbb{R}}$ by
\begin{align}\label{eq:beta_process}
\beta_t^m = \begin{cases}
B_t - mt & \text{if $t \ge 0$} \\
\mathcal{B}_{-t}^m & \text{if $t \le 0$}
\end{cases}
\end{align}

\noindent where $(B_t)_{t \ge 0}$ and $(\mathcal{B}_{t}^m)_{t \ge 0}$ are independent of each other.  In particular we can re-express the constant $c_\gamma(m)$ defined in \eqref{eq:constant} as
\begin{align}\label{eq:constant2}
c_{\gamma}(m) = \frac{1}{\pi} \mathbb{E} \left[ \int_{-\infty}^\infty \mathcal{I}\left(e^{\gamma \beta_t^m}\right)dt\right].
\end{align}

\noindent Before we proceed, let us explain why the constant $c_\gamma(m)$ is finite for positive $\gamma$ and $m$.
\begin{lem}\label{lem:const_finite}
The constant $c_{\gamma}(m)$ defined in \eqref{eq:constant} is finite for any $\gamma, m > 0$.
\end{lem}

\begin{proof}
We start with the first expectation in \eqref{eq:constant}, and consider
\begin{align*}
\mathbb{E}\left[\mathcal{I}\left(e^{\gamma (B_t - mt)}\right)\right] 
& = 
\mathbb{E}\left[e^{\gamma (B_t - mt)} \exp\left(-e^{\gamma (B_t - mt)}\right)1_{\{B_t - mt \le -\frac{1}{2}mt\}}\right] \\
& \qquad  +
\mathbb{E}\left[e^{\gamma (B_t - mt)} \exp\left(-e^{\gamma (B_t - mt)}\right)1_{\{B_t -mt > -\frac{1}{2}mt\}}\right]\\
& \le e^{-\frac{\gamma m}{2}t} + \mathbb{P}\left(B_t -mt > -\frac{1}{2}mt\right) \\
&\le e^{-\frac{\gamma m}{2}t} + e^{-\frac{1}{8} m^2 t}.
\end{align*}

\noindent This shows that
\begin{align*}
\mathbb{E}\left[\int_{0}^\infty 
\mathcal{I}\left(e^{\gamma (B_t - mt)} \right)dt \right]  
\le \int_0^\infty \left[ e^{-\frac{\gamma m}{2}t} + e^{-\frac{1}{8} m^2 t}\right] dt < \infty.
\end{align*}

\noindent As for the second expectation in \eqref{eq:constant}, we consider
\begin{align*}
\mathbb{E}\left[\mathcal{I}\left(e^{\gamma \mathcal{B}_t^{m}} \right)\right]
& = \mathbb{E}\left[e^{\gamma \mathcal{B}_t^{m}} \exp\left(-e^{\gamma \mathcal{B}_t^{m}}\right)1_{\{\mathcal{B}_t^m \le \frac{1}{2}mt\}}\right]
+  \mathbb{E}\left[e^{\gamma \mathcal{B}_t^{m}} \exp\left(-e^{\gamma \mathcal{B}_t^{m}} \right)1_{\{\mathcal{B}_t^m > \frac{1}{2}mt\}}\right].
\end{align*}

\noindent The fact that $B_t + mt$ is stochastically dominated by $\mathcal{B}_t^m$ implies that
\begin{align*}
\mathbb{E}\left[e^{\gamma \mathcal{B}_t^{m}} \exp\left(-e^{\gamma \mathcal{B}_t^{m}}\right)1_{\{\mathcal{B}_t^m \le \frac{1}{2}mt\}}\right]
\le \mathbb{P}\left(\mathcal{B}_t^m \le \frac{1}{2} mt \right)
\le \mathbb{P}\left(B_t + mt \le \frac{1}{2} mt \right)
\le e^{-\frac{1}{8} m^2 t}.
\end{align*}

\noindent Meanwhile, using the elementary inequality $x e^{-x} \le 2 e^{-x/2}$ for $x \ge 0$ we also obtain
\begin{align*}
\mathbb{E}\left[e^{\gamma \mathcal{B}_t^{m}} \exp\left(-e^{\gamma \mathcal{B}_t^{m}} \right)1_{\{\mathcal{B}_t^m > \frac{1}{2}mt\}}\right] \le 2 e^{-\frac{\gamma m}{4}t}.
\end{align*}

\noindent Hence,
\begin{align*}
\mathbb{E}\left[\int_{0}^{\infty} \mathcal{I}\left(
e^{\gamma \mathcal{B}_t^{m}} \right)dt\right]
\le \int_0^\infty \left[ e^{-\frac{1}{8} m^2 t} +  2e^{-\frac{\gamma m}{4}t}\right]dt < \infty
\end{align*}

\noindent and we conclude that $c_{\gamma}(m) < \infty$.
\end{proof}

\subsection{Main lemma}
The following lemma will be used to help us obtain uniform estimates and pointwise limits that are needed for the application of dominated convergence in the main proof. We will be using the following notation: for each $\gamma, m > 0$ and function $f: [0, \infty) \to [0, \infty)$, define
\begin{align}\label{eq:constantg}
\begin{split}
c_{\gamma}(m; f) 
:=& \frac{1}{\pi} \mathbb{E} \left[ \int_0^\infty f\left(e^{\gamma \beta_t^m}\right) dt \right]\\
=&\frac{1}{\pi}\Bigg\{
\mathbb{E}\left[\int_{0}^{\infty} 
f\left(e^{\gamma \mathcal{B}_t^{m}} \right)dt\right]
+ \mathbb{E}\left[\int_{0}^\infty
f\left(e^{\gamma (B_t - mt)}\right)dt\right] 
\Bigg\}
\end{split}
\end{align}

\noindent In particular, if $\mathcal{I}(x) = xe^{-x}$, then $c_{\gamma}(m; \mathcal{I}) = c_{\gamma}(m)$ as defined in \eqref{eq:constant}.

\begin{lem}\label{lem:main}
Consider the following random objects:
\begin{itemize}
\item $(B_{1,t})_{t \ge 0}$ and $(B_{2,t})_{t \ge 0}$ are two independent Brownian motions;
\item $\mathcal{E}_0, \mathcal{E}_1, \mathcal{E}_2$ are non-negative random variables that are independent of $(B_{1,t})_{t \ge 0}$ and $(B_{2, t})_{t\ge 0}$, and $\mathbb{E}[\mathcal{E}_0] < \infty$.
\end{itemize}

\noindent In addition, for each $i \in \{1, 2\}$ let $m_i, \gamma_i > 0$ and $\mathcal{I}_i: [0, \infty) \to [0, \infty)$ be such that $c_{\gamma_i}(m_i; \mathcal{I}_i) < \infty$  and that $\mathcal{I}_i(0) = 0$. Then the following statements hold.
\begin{itemize}[leftmargin=*]
\item For all $\lambda_1, \lambda_2 > 0$,
\begin{align}
\label{eq:uniform_main1}
\mathbb{E} \left[\mathcal{E}_0 \int_0^\infty \mathcal{I}_1\left(\lambda_1 \mathcal{E}_1 e^{\gamma_1 (B_{1,t} - m_1t)}\right)dt\right]
& \le \pi c_{\gamma_1}(m_1; \mathcal{I}_1)\mathbb{E} \left[\mathcal{E}_0 \right] \\
\label{eq:uniform_main2}
\text{and} \quad
\mathbb{E} \left[\mathcal{E}_0 \prod_{i=1}^2\left( \int_0^\infty \mathcal{I}_i\left(\lambda_i \mathcal{E}_i e^{\gamma_i (B_{i,t} - m_it)}\right)dt\right)\right]
&\le  \left[\prod_{i=1}^2 \pi c_{\gamma_i}(m_i; \mathcal{I}_i)\right]\mathbb{E} \left[\mathcal{E}_0 \right].
\end{align}

\item We have
\end{itemize}
\begin{align}
\label{eq:limit_main1}
\lim_{\lambda_1 \to \infty} \mathbb{E} \left[\mathcal{E}_0 \int_0^\infty \mathcal{I}_1\left(\lambda_1 \mathcal{E}_1 e^{\gamma_1 (B_{1,t} - m_1t)}\right)dt\right]
& = \pi c_{\gamma_1}(m_1; \mathcal{I}_1)\mathbb{E} \left[\mathcal{E}_0 \right] \\
\label{eq:limit_main2}
\text{and} \quad
\lim_{\lambda_1, \lambda_2 \to \infty} \mathbb{E} \left[\mathcal{E}_0 \prod_{i=1}^2\left( \int_0^\infty \mathcal{I}_i\left(\lambda_i \mathcal{E}_i e^{\gamma_i (B_{i,t} - m_it)}\right)dt\right)\right]
&=  \left[\prod_{i=1}^2 \pi c_{\gamma_i}(m_i; \mathcal{I}_i)\right]\mathbb{E} \left[\mathcal{E}_0 \right].
\end{align}
\end{lem}

\begin{rem}
The random variables $\mathcal{E}_0, \mathcal{E}_1, \mathcal{E}_2$ need not be independent of each other, and the limit as $\lambda_1, \lambda_2$ go to infinity on the LHS of \eqref{eq:limit_main2} can be taken in any order/along any subsequence. See also Proposition \ref{prop:toy} for a simple application of Lemma \ref{lem:main} which gives an idea of how it is applied to the problem of interest.
\end{rem}

\begin{proof}
Let us treat \eqref{eq:uniform_main1} and \eqref{eq:limit_main1}. The assumption on $\mathcal{I}_1$ means that
\begin{align*}
 \int_0^\infty \mathcal{I}_1\left(\lambda_1 \mathcal{E}_1 e^{\gamma_1 (B_{1,t} - m_1t)}\right)dt
=  1_{\{\lambda_1 \mathcal{E}_1 > 0\}} \int_0^\infty \mathcal{I}_1\left(\lambda_1 \mathcal{E}_1 e^{\gamma_1 (B_{1,t} - m_1t)}\right)dt\qquad a.s.
\end{align*}

\noindent and so we will analyse the expectation by splitting it into two contributions depending on whether $\lambda_1 \mathcal{E}_1 \in (0, 1]$ or $\lambda_1 \mathcal{E}_1 > 1$. We start with
\begin{align*}
& \mathbb{E} \left[\mathcal{E}_0 1_{\{\lambda_1 \mathcal{E}_1 \in (0, 1]\}}\int_0^\infty \mathcal{I}\left(\lambda_1 \mathcal{E}_1 e^{\gamma_1 (B_{1,t} - m_1t)}\right)dt \right] \\
& \qquad = \sum_{n \ge 0}  \mathbb{E} \left[\mathcal{E}_0 1_{\{\lambda_1 \mathcal{E}_1 \in (2^{-(n+1)}, 2^{-n}]\}}\int_0^\infty \mathcal{I}_1\left(\lambda_1 \mathcal{E}_1 e^{\gamma_1 (B_{1,t} - m_1t)}\right)dt \right]\\
& \qquad = \sum_{n \ge 0}  \mathbb{E} \left[\mathcal{E}_0 1_{\{\lambda_1 \mathcal{E}_1 \in (2^{-(n+1)}, 2^{-n}]\}}\int_{\widehat{\tau}_{\lambda_1 \mathcal{E}_1}^{(1)}}^\infty \mathcal{I}_1\left(e^{\gamma_1 (B_{1,t} - m_1t)}\right)dt \right]
\end{align*}

\noindent where 
\begin{align*}
\widehat{\tau}_{\lambda_1 \mathcal{E}_1}^{(1)} := \inf \{t \ge 0: e^{\gamma_1 (B_{1,t} - m t)} = \lambda_1 \mathcal{E}_1\}
\end{align*}

\noindent by strong Markov property. We may control the last expression with the rough upper bound
\begin{align*}
& \mathbb{E} \left[\mathcal{E}_0\left( \sum_{n \ge 0}1_{\{\lambda_1 \mathcal{E}_1 \in (2^{-n}, 2^{-(n-1)}]\}}\right)\int_{0}^\infty \mathcal{I}_1\left(   e^{\gamma_1 (B_{1,t} - m_1 t)} \right) dt\right]\\
& \qquad = \mathbb{E} \left[\mathcal{E}_0 1_{\{\lambda_1 \mathcal{E}_1 \in (0,1]\}}\right]\mathbb{E}\left[\int_{0}^\infty \mathcal{I}_1\left(   e^{\gamma_1 (B_{1,t} - m_1 t)} \right) dt\right]
\le \pi c_{\gamma_1}(m_1; \mathcal{I}_1) \mathbb{E}\left[\mathcal{E}_0 1_{\{ 0 < \lambda_1 \mathcal{E}_1 \le 2\}}\right]
\end{align*}

\noindent which is 
\begin{itemize}
\item uniformly bounded by $\pi c_{\gamma_1}(m_1; \mathcal{I}_1) \mathbb{E}\left[\mathcal{E}_0\right]$, and
\item converging to $0$ as $\lambda_1 \to \infty$ by monotone convergence.
\end{itemize}

Next, we look at the main term
\begin{align}
\label{eq:mainlem_main1}
& \mathbb{E} \left[\mathcal{E}_0 1_{\{\lambda_1 \mathcal{E}_1 > 1\}}\int_0^\infty \mathcal{I}_1\left(\lambda_1 \mathcal{E}_1 e^{\gamma_1 (B_{1,t} - m_1t)}\right)dt \right].
\end{align}

\noindent Let us introduce a different stopping time
\begin{align*}
\widetilde{\tau}_{\lambda_1 \mathcal{E}_1}^{(1)} := \inf \{t > 0: e^{\gamma_1 (B_{1,t} - m_1 t)} = (\lambda_1 \mathcal{E}_1)^{-1}\}
\end{align*}

\noindent which is strictly positive (and finite) on the event that $\lambda_1 \mathcal{E}_1 > 1$, where we have
\begin{align*}
&\int_0^\infty  \mathcal{I}_1\left(\lambda_1  \mathcal{E}_1 e^{\gamma_1 (B_{1,t} - m_1 t)}\right)dt\\
&\qquad \overset{d}{=}
\int_0^{\infty} \mathcal{I}_1\left(\exp\left(\gamma_1\left[(B_{1,t} - m_1t)  - (B_{1, \widetilde{\tau}_{\lambda_1 \mathcal{E}_1}^{(1)}}  - m_1\widetilde{\tau}_{\lambda_1 \mathcal{E}_1}^{(1)} ) \right]\right)\right) dt
\end{align*}

\noindent and the integral on the RHS can be split into two parts:
\begin{itemize}
\item $t \ge \widetilde{\tau}_{\lambda_1 \mathcal{E}_1}^{(1)}$. By strong Markov property, the process
\begin{align*}
\left[B_{1, \widetilde{\tau}_{\lambda_1 \mathcal{E}_1}^{(1)} + t} -m_1 ( \widetilde{\tau}_{\lambda_1 \mathcal{E}_1}^{(1)} + t)\right]
-\left[B_{1, \widetilde{\tau}_{\lambda_1 \mathcal{E}_1}^{(1)}}-m_1 \widetilde{\tau}_{\lambda_1 \mathcal{E}_1}^{(1)}\right], \qquad t \ge 0
\end{align*}

\noindent is a Brownian motion with negative drift $-m_1$ independent of $(B_{1, t} - m_1t)_{t \le \widetilde{\tau}_{\lambda_1 \mathcal{E}_1}^{(1)}}$.

\item $t \le \widetilde{\tau}_{\lambda_1 \mathcal{E}_1}^{(1)}$: we apply \Cref{lem:time_reversal} and write
\begin{align*}
&\left( \left[B_{1, \widetilde{\tau}_{\lambda_1 \mathcal{E}_1}^{(1)} - t} - m_1(\widetilde{\tau}_{\lambda_1 \mathcal{E}_1}^{(1)} - t)\right]- \left[B_{1, \widetilde{\tau}_{\lambda_1 \mathcal{E}_1}^{(1)}} - m_1\widetilde{\tau}_{\lambda_1 \mathcal{E}_1}^{(1)}\right] \right)_{t \in [0, \widetilde{\tau}_{\lambda_1 \mathcal{E}_1}^{(1)}]}
= (\mathcal{B}_{1,t}^{m_1})_{t \in [0, \widetilde{L}_{\lambda_1 \mathcal{E}_1}^{(1)}]}
\end{align*}

\noindent where $(\mathcal{B}_{1,t}^{m_1})_{t \ge 0}$ is a Brownian motion with drift $m_1$ conditioned to be non-negative (and independent of $\mathcal{E}_1$), and
\begin{align*}
\widetilde{L}_{\lambda_1 \mathcal{E}_1}^{(1)} := \sup \{t > 0: e^{\gamma_1 \mathcal{B}_{1,t}^{m_1}} = \lambda_1 \mathcal{E}_1\}.
\end{align*}
\end{itemize}

\noindent Substituting everything back to the expectation \eqref{eq:mainlem_main1}, we get
\begin{align*}
\mathbb{E} \left[\mathcal{E}_0 1_{\{\lambda_1 \mathcal{E}_1 > 1\}} 
\bigg\{ \int_0^{\widetilde{L}_{\lambda_1 \mathcal{E}_1}^{(1)}} \mathcal{I}_1\left(e^{\gamma_1 \mathcal{B}_{1,t}^{m_1}}\right)
dt+  \int_0^\infty \mathcal{I}_1\left(e^{\gamma_1 (B_{1,t} - m_1 t)} \right)dt\bigg\}
 \right]
\end{align*}

\noindent which is
\begin{itemize}
\item uniformly bounded by $\pi c_{\gamma_1}(m_1; \mathcal{I}_1)\mathbb{E} \left[\mathcal{E}_0 1_{\{\lambda_1 \mathcal{E}_1 > 1\}} \right]$, and
\item converging to $\pi c_{\gamma_1}(m_1; \mathcal{I}_1)\mathbb{E} \left[\mathcal{E}_0\right]$ as $\lambda_1 \to \infty$ by monotone convergence.
\end{itemize}

\noindent This gives \eqref{eq:uniform_main1} and \eqref{eq:limit_main1}. The proof of \eqref{eq:uniform_main2} and \eqref{eq:limit_main2} is similar and omitted.
\end{proof}

\section{Weyl's law and heat trace asymptotics}  \label{sec:proofWeyl}
This section is devoted to the proof of \Cref{theo:Weyl}. Before we begin, let us mention that we can assume without loss of generality that $\mathrm{diam}(D) := \sup_{x, y \in D} |x-y| < \frac{1}{2}$.  This is not a problem because of the scale-invariant nature of the asymptotics in \Cref{theo:Weyl} (and hence the other results). To simplify notation, we shall also write $c_{\gamma} = c_\gamma(Q-\gamma; \mathcal{I})$ where $\mathcal{I}(x) = xe^{-x}$ throughout this section.

The following is an outline of our proof of \Cref{theo:Weyl}, which  follows a modified second moment method:
\begin{itemize}
    \item To avoid any complication arising from the boundary, we perform several pre-processing steps in \Cref{sec:pf_preprocess} to show that boundary contributions are irrelevant in the limit $\lambda \to \infty$. To certain extent such analysis is a manifestation of Kac's principle of `not feeling the boundary'.
    \item For $\gamma \in [1, 2)$ it is well-known that $\mu_{\gamma}$ (and related random variables) are not $L^2$-integrable. Inspired by \cite{Ber2017}, we introduce a good event on which second moment method can be performed in the entire subcritical phase. We first establish in \Cref{sec:pf_badevent} that contribution from the complementary event vanishes as $\lambda \to \infty$, and then provide a roadmap for the remaining analysis.
    \item Finally, we will evaluate all the second moments by means of dominated convergence and show that they all coincide in the limit as $\lambda \to \infty$.
\end{itemize}

Note that the last part of the analysis makes heavy use of our Main lemma. To get a flavour of how \Cref{lem:main} may be applied, it may be instructive to look at the following toy computation.

\begin{prop}
\label{prop:toy}For $\gamma \in (0, 2)$, let $\widetilde{\mu}_{\gamma}(dx) := e^{\gamma X^{2\mathbb{D}}(x) - \frac{\gamma^2}{2}\mathbb{E}[X^{2\mathbb{D}}(x)^2]} dx$ be the GMC measure associated to the log-correlated Gaussian field $X^{2\mathbb{D}}$ with covariance 
\begin{align*}
    \mathbb{E}[X^{2\mathbb{D}}(x)X^{2\mathbb{D}}(y)] = -\log|x-y| + \log 2 \qquad \forall x, y \in B(0, 2).
\end{align*}

\noindent Then for any $A \subset B(0, 1)$, we have
\begin{align*}
\lim_{\lambda \to \infty} \mathbb{E}\left[\int_A \widetilde{\mu}_{\gamma}(dx) \int_0^1 \frac{du}{2\pi u} \mathcal{I}(\lambda \widetilde{\mu}_\gamma(B(x, \sqrt{u}))) \right]
=  c_\gamma \mathbb{E}[\widetilde{\mu}_\gamma(A)].
\end{align*}
\end{prop}
\begin{proof}
By Fubini and Cameron-Martin theorem, we start by rewriting
\begin{align*}
& \mathbb{E}\left[\int_A \widetilde{\mu}_{\gamma}(dx) \int_0^1 \frac{du}{2\pi u} \mathcal{I}(\lambda \widetilde{\mu}_\gamma(B(x, \sqrt{u}))) \right]
= \int_A dx \mathbb{E}\left[ \int_0^1 \frac{du}{2\pi u} \mathcal{I}(\lambda \widetilde{\mu}_\gamma(x, \sqrt{u})) \right]
\end{align*}

\noindent where
\begin{align*}
\widetilde{\mu}_{\gamma}(x, \sqrt{u})
& := \int_{B(x, \sqrt{u})} \frac{e^{\gamma X^{2\mathbb{D}}(z) - \frac{\gamma^2}{2} \mathbb{E}[X^{2\mathbb{D}}(z)^2]}dz}{(|x-z|/2)^{\gamma^2}}.
\end{align*}

\noindent From exact scale invariance
\begin{align*}
 \mathbb{E}[X^{2\mathbb{D}}(x + a\sqrt{u})X^{2\mathbb{D}}(x + b\sqrt{u})]
 = \mathbb{E}[X^{2\mathbb{D}}(a)X^{2\mathbb{D}}(b)] - \log \sqrt{u} \qquad \forall a, b \in B(0, 1),
\end{align*}

\noindent it follows (with a substitution of variable $z \leftrightarrow x + \sqrt{u} z$) that
\begin{align*}
\widetilde{\mu}_{\gamma}(x, \sqrt{u}) & \overset{d}{=} \sqrt{u}^{2 - \gamma^2} e^{\gamma B_{t(u)} - \frac{\gamma^2}{2} \mathbb{E}[B_{t(u)}^2]} \underbrace{\int_{B(0, 1)} \frac{e^{\gamma X^{2\mathbb{D}}(z) - \frac{\gamma^2}{2} \mathbb{E}[X^{2\mathbb{D}}(z)^2]}dz}{(|z|/2)^{\gamma^2}}}_{=: \mathcal{E}_1}
\end{align*}

\noindent where $B_{t(u)} \sim \mathcal{N}(0, t(u))$ is independent of $\mathcal{E}_1$ with $t(u) := -\log \sqrt{u}$. Thus
\begin{align*}
\widetilde{\mu}_{\gamma}(x, \sqrt{u}) \overset{d}{=} \mathcal{E}_1 e^{\gamma (B_{t(u)} - m t(u))}
\qquad \text{where $m = Q-\gamma$ with $Q = \frac{\gamma}{2} + \frac{2}{\gamma}$}.
\end{align*}

Using the substitution $u = e^{-2t}$ we have
\begin{align*}
\int_A dx \mathbb{E}\left[ \int_0^1 \frac{du}{2\pi u} \mathcal{I}(\lambda \widetilde{\mu}_\gamma(x, \sqrt{u})) \right]
& = \int_A dx \mathbb{E}\left[ \int_0^\infty \frac{dt}{\pi}  \mathcal{I}(\lambda \mathcal{E}_1 e^{\gamma (B_{t} - m t)} ) \right].
\end{align*}

\noindent If we now apply \Cref{lem:main} with $\mathcal{E}_0 := \frac{1}{\pi}$, then:
\begin{itemize}
\item our integrand is uniformly bounded in $x \in A$ and $\lambda > 0$, and so we can apply dominated convergence when evaluating the limit $\lambda \to \infty$;
\item the pointwise limit of our integrand as $\lambda \to \infty$ is given by $c_\gamma = c_\gamma(m)$,
\end{itemize}

\noindent i.e. we conclude that
\begin{align*}
\lim_{\lambda \to \infty}\int_A dx \mathbb{E}\left[ \int_0^1 \frac{du}{2\pi u} \mathcal{I}(\lambda \widetilde{\mu}_\gamma(x, \sqrt{u})) \right]
= c_{\gamma} \int_A dx  = c_\gamma \mathbb{E}[\widetilde{\mu}_\gamma(A)].
\end{align*}
\end{proof}

\subsection{Pre-processing: removal of irrelevant contributions}\label{sec:pf_preprocess}
To avoid any complication when we derive uniform estimates in later steps, we show that contributions from Brownian bridges with high probability of hitting the boundary $\partial D$ are irrelevant in the following sense.
\begin{lem}\label{lem:u_cutoff}
We have
\begin{align*}
&\limsup_{\lambda \to \infty} \mathbb{E}\left[\int_D \mu_{\gamma}(dx) \int_{1}^\infty \frac{du}{2\pi u} \bdec{x}{x}{u}[ \mathcal{I} \left(\lambda  F_{\gamma}(\mathbf{b})\right)1_{\{u< \tau_D(\mathbf{b})\}}]\right] = 0.
\end{align*}

\end{lem}

\begin{proof}
As $\mathcal{I}(x) \le 1$ for all $x \ge 0$,
\begin{align*}
\bdec{x}{x}{u}[ \mathcal{I} \left(\lambda  F_{\gamma}(\mathbf{b})\right)1_{\{u< \tau_D(\mathbf{b}) \}}]
& \le \pdec{x}{x}{u}\left( \mathbf{b}_s \in  D ~ \forall s \le u\right)\\
& \le \pdec{x}{x}{u} \left(\max_{s \le u} |\mathbf{b}_s - x| \le 1\right)
 \le 1 \wedge \frac{2}{u}
\end{align*}

\noindent by \Cref{cor:bb_bound}, and hence
\begin{align*}
\int_D \mu_{\gamma}(dx) \int_{1}^\infty \frac{du}{2\pi u} \bdec{x}{x}{u}[ \mathcal{I} \left(\lambda  F_{\gamma}(\mathbf{b})\right)1_{\{u< \tau_D(\mathbf{b}) \}}]
\le \mu_{\gamma}(D) \int_{1}^\infty \frac{du}{2\pi u} \frac{2}{u} \le \mu_{\gamma}(D)
\end{align*}

\noindent which has finite expectation. On the other hand, since $\mathcal{I}(x) \to 0$ as $x \to \infty$, we see that $\bdec{x}{x}{u}[ \mathcal{I} \left(\lambda  F_{\gamma}(\mathbf{b})\right)1_{\{u< \tau_D(\mathbf{b}) \}}] \to 0$ almost surely for almost every $x \in D$ and $u \ge 1$. The claim now follows from dominated convergence.
\end{proof}

Let us also highlight that boundary contributions are irrelevant in the following sense.
\begin{lem}\label{lem:L1-boundary}
We have
\begin{align}\label{eq:L1-boundary}
\limsup_{\kappa \to 0^+}
\limsup_{\lambda \to \infty} \mathbb{E}\left[\int_D 1_{\{d(x, \partial D) \le \kappa\}}\mu_{\gamma}(dx) \int_{0}^{1} \frac{du}{2\pi u} \bdec{x}{x}{u}[ \mathcal{I} \left(\lambda F_{\gamma}(\mathbf{b})\right)1_{\{u< \tau_D(\mathbf{b}) \}}]\right] = 0.
\end{align}
\end{lem}

In order to prove \Cref{lem:L1-boundary}, we first apply Fubini and Cameron-Martin theorem and rewrite \eqref{eq:L1-boundary} as
\begin{align}
\notag
&\mathbb{E}\left[\int_D 1_{\{d(x, \partial D) \le \kappa\}}\mu_{\gamma}(dx) \int_{0}^{1} \frac{du}{2\pi u} \bdec{x}{x}{u}[ \mathcal{I} \left(\lambda  F_{\gamma}(\mathbf{b}) \right)1_{\{u< \tau_D(\mathbf{b}) \}}]\right]\\
\label{eq:pf_boundary1}
& \qquad =
\int_D 1_{\{d(x, \partial D) \le \kappa\}} R(x; D)^{\frac{\gamma^2}{2}} dx 
\mathbb{E}\left[  \int_{0}^{1} \frac{du}{2\pi u} \bdec{x}{x}{u}[ \mathcal{I} \left(\lambda F_{\gamma}^{\{x\}}(\mathbf{b})\right)1_{\{u< \tau_D \}}]\right]
\end{align}

\noindent where, for any finite set $\mathcal{S} \subset D$ and process $\mathbf{p}$,
\begin{align}\label{eq:Girsanov1}
F_\gamma^{\mathcal{S}}(\mathbf{p})& :=
 \int_0^{\ell(\mathbf{p})}
e^{\gamma^2 \sum_{z \in \mathcal{S}} G_0^D(z, \mathbf{p}_s) }
F_\gamma(ds; \mathbf{p}).
\end{align}

To proceed further, we need to control the expectation on the RHS of \eqref{eq:pf_boundary1} uniformly in $\lambda > 0$. We now demonstrate how this can be done by partitioning the probability space according to the range of the Brownian bridge $\mathbf{b}$, a trick that will be used repeatedly throughout the rest of this article.
\begin{lem}\label{lem:G1_uniform}
For each $k \in \mathbb{N}$, let
\begin{align}\label{eq:eventH}
\mathcal{H}_k
= \mathcal{H}_k(\mathbf{b})
=  \left\{ \max_{s\le \ell(\mathbf{b})} \frac{|\mathbf{b}_s -\iota(\mathbf{b})|}{\sqrt{\ell(\mathbf{b})}} \in [k-1, k) \right\}
\end{align}

\noindent where $\ell(\mathbf{b})$ and $\iota(\mathbf{b})$ are the duration and starting point of the Brownian bridge $\mathbf{b}$ respectively. There exists some $C \in (0, \infty)$, possibly dependent on $\gamma$ but uniformly in $x \in D$, $\lambda > 0$ and $k \in \mathbb{N}$ such that
\begin{align}\label{eq:G1_uniform}
\mathbb{E}\left[  \int_{0}^{1} 1_{\{ d(x, \partial D) \ge 4k\sqrt{u}\}} \frac{du}{2\pi u} \bdec{x}{x}{u} [ \mathcal{I} \left(\lambda  F_\gamma^{\{x\}}(\mathbf{b})\right)1_{\mathcal{H}_k}]\right]
\le C \pdec{0}{0}{1}\left(\mathcal{H}_k\right).
\end{align}
\end{lem}

\begin{proof}
Let us start by interchanging the order of expectations:
\begin{align}
\notag
&\mathbb{E}\left[  \int_{0}^{1} \frac{du}{2\pi u} 1_{\{d(x, \partial D) \ge  4k\sqrt{u}\}}\bdec{x}{x}{u}[ \mathcal{I} \left(\lambda  F_{\gamma}^{\{x\}}(\mathbf{b})\right)1_{\mathcal{H}_k}]\right]\\
\label{eq:G1_uniform_step0}
& \qquad \qquad = \int_{0}^{1} \frac{du}{2\pi u} 1_{\{d(x, \partial D) \ge  4k\sqrt{u}\}}\bdec{x}{x}{u}\left[ \mathbb{E}\left[ \mathcal{I} \left(\lambda  F_{\gamma}^{\{x\}}(\mathbf{b})\right)\right]1_{\mathcal{H}_k}\right].
\end{align}

\noindent Applying Cameron--Martin to the inner expectation, we have
\begin{align}
\notag
 \mathbb{E}\left[\mathcal{I}(\lambda F_{\gamma}^{\{x\}}(\mathbf{b}))\right]
& =  \mathbb{E}\left[\lambda F_{\gamma}^{\{x\}}(\mathbf{b}) e^{-\lambda F_{\gamma}^{\{x\}}(\mathbf{b})}\right]\\
\notag
& = \int_0^u \lambda e^{\gamma^2 G_0^D(x, \mathbf{b}_{s_1})} R(\mathbf{b}_{s_1}; D)^{\frac{\gamma^2}{2}} 1_{\{\mathbf{b}_{s_1} \in D\}} ds_1\\
\label{eq:bound_boundary1}
& \qquad \qquad \times \mathbb{E}\left[ \exp\left(-\lambda \int_{0}^u e^{\gamma^2 [G_0^D(x, \mathbf{b}_{s_1}) + G_0^D(\mathbf{b}_{s_1}, \mathbf{b}_{s_2})]}F_\gamma(ds_2; \mathbf{b})\right)\right].
\end{align}

\noindent The rest of the proof may be divided into three steps which we now explain.

\paragraph{Step (i): Gaussian comparison.}

\noindent On the event $\mathcal{H}_k$, we know that the Brownian bridge $(\mathbf{b}_{s})_{s \le u}$ stays in the ball $B(x, k\sqrt{u})$. Furthermore, since $d(x, \partial D) \ge 4k\sqrt{u}$, it follows from  \Cref{cor:gff_exact} that
\begin{align*}
 \left| G_0^D(\mathbf{b}_{s_1}, \mathbf{b}_{s_2}) - \left[ - \log |\mathbf{b}_{s_1} - \mathbf{b}_{s_2}| + \log R(x; D) \right] \right| &\le 4 .
\end{align*}

\noindent In particular this implies
\begin{align*}
 \left| G_0^D(x, \mathbf{b}_{s_1}) - \left[ - \log |x - \mathbf{b}_{s_1}| + \log R(x; D) \right] \right| &\le 4 \qquad \text{(by setting $s_2 = 0$)}\\
\text{and} \qquad \qquad  \qquad \qquad |\log R(x; D) - \log R(\mathbf{b}_{s_1}; D)| &\le 4 \qquad \text{(by letting $s_2 \to s_1$)}
\end{align*}

\noindent so that \eqref{eq:bound_boundary1} may be upper-bounded by
\begin{align}
\notag
& \lambda e^{6 \gamma^2} R(x; D)^{\frac{3\gamma^2}{2}} \int_0^u \frac{1_{\{\mathbf{b}_{s_1} \in D \}}ds_1}{|\mathbf{b}_{s_1}-x|^{\gamma^2}}\\
\label{eq:Girsanov_boundary2}
&\quad \times \mathbb{E}\left[ \exp\left(-\lambda  e^{-10\gamma^2} R(x; D)^{\frac{5\gamma^2}{2}} \int_{0}^u 1_{\{\mathbf{b}_{s_2} \in B(x, k\sqrt{u}) \}} \frac{e^{\gamma h(\mathbf{b}_{s_2}) - \frac{\gamma^2}{2} \mathbb{E}[h(\mathbf{b}_{s_2})^2]}ds_2}{|\mathbf{b}_{s_2}-x|^{\gamma^2} |\mathbf{b}_{s_1}-\mathbf{b}_{s_2}|^{\gamma^2}}\right)\right].
\end{align}

\noindent We would like to perform a Gaussian comparison using \Cref{cor:Gcompare} with the convex function $P(x) = \exp(-x)$, replacing the Gaussian free field with an exactly scale invariant field $X(\cdot)$ with covariance
\begin{align*}
\mathbb{E}[X(a) X(b)] = -\log |a-b| + \log R(x; D) + 4 \qquad \forall a, b \in B(x, k\sqrt{u}).
\end{align*}

\noindent This field is well-defined because the above kernel is positive definite in a ball of radius at least $R(x; D)$, whereas $k\sqrt{u} \le d(x, \partial D)/4 \le R(x; D)$ where the last inequality follows from Koebe quarter theorem. By construction, we have
\begin{align*}
\mathbb{E}[h(a)h(b)] = G_0^D(a, b)\le \mathbb{E}[X(a)X(b)] \qquad \forall a, b \in B(x, k\sqrt{u}),
\end{align*}

\noindent and thus \eqref{eq:Girsanov_boundary2} may be further upper-bounded by
\begin{align}
\notag
& \lambda e^{6 \gamma^2} R(x; D)^{\frac{3\gamma^2}{2}} \int_0^u \frac{1_{\{\mathbf{b}_{s_1} \in B(x, k\sqrt{u}) \}}ds_1}{|\mathbf{b}_{s_1}-x|^{\gamma^2}}\\
\notag
&\qquad \times \mathbb{E}\left[ \exp\left(-\lambda  e^{-10\gamma^2} R(x; D)^{\frac{5\gamma^2}{2}} \int_{0}^u 1_{\{\mathbf{b}_{s_2} \in B(x, k\sqrt{u}) \}}\frac{e^{\gamma X(\mathbf{b}_{s_2}) - \frac{\gamma^2}{2} \mathbb{E}[X(\mathbf{b}_{s_2})^2]} ds_2}{|\mathbf{b}_{s_2}-x|^{\gamma^2} |\mathbf{b}_{s_1}-\mathbf{b}_{s_2}|^{\gamma^2}}\right)\right]\\
\notag
& = e^{6 \gamma^2}  \mathbb{E}\left[\lambda R(x; D)^{\frac{3\gamma^2}{2}}
\overline{F}_{\gamma}^{\{x\}}(\mathbf{b}; X)
\exp\left(-\lambda  e^{-14\gamma^2} R(x; D)^{\frac{3\gamma^2}{2}}
\overline{F}_{\gamma}^{\{x\}}(\mathbf{b}; X)
\right)\right]\\
\label{eq:G1_uniform_step1}
& = e^{20\gamma^2} \mathbb{E}\left[\mathcal{I}\left(\widetilde{\lambda}\overline{F}_{\gamma}^{\{x\}}(\mathbf{b}; X)\right)\right] \qquad \qquad \text{with} \qquad \widetilde{\lambda}:=\lambda e^{-14\gamma^2}R(x; D)^{\frac{3\gamma^2}{2}} 
\end{align}

\noindent where, for any finite set $\mathcal{S} \subset D$,
\begin{align} \label{eq:overlineF}
\overline{F}_{\gamma}^{\mathcal{S}}(\mathbf{p};  Y)
& :=
\int_{0}^{\ell(\mathbf{p})} e^{\gamma Y(\mathbf{p}_s) - \frac{\gamma^2}{2}\mathbb{E}[Y(\mathbf{p}_s)^2]} \frac{ds}{\prod_{z \in \mathcal{S}}| \mathbf{p}_s - z |^{\gamma^2}}.
\end{align}

\paragraph{Step (ii): scale invariance.} Under $\bdec{x}{x}{u}$, the rescaled process
\begin{align}\label{eq:BB_rescale}
\left(\frac{1}{\sqrt{u}} \left(\mathbf{b}_{us} - x\right), \quad s \le 1\right)
\end{align}

\noindent has the same distribution as a Brownian loop of duration $1$ starting from the origin. It follows from \eqref{eq:bound_boundary1} and \eqref{eq:G1_uniform_step1} that
\begin{align}
\notag
\bdec{x}{x}{u}\left[ \mathbb{E}\left[ \mathcal{I} \left(\lambda  F_\gamma^{\{x\}}(\mathbf{b})\right)\right]1_{\mathcal{H}_k}\right]
& \le e^{20\gamma^2} \bdec{x}{x}{u}\left[ \mathbb{E}\left[ \mathcal{I} \left(\widetilde{\lambda}\overline{F}_{\gamma}^{\{x\}}(\mathbf{b}; X)
\label{eq:G1_uniform_step1b}
\right)\right]1_{\mathcal{H}_k}\right]\\
& = e^{20\gamma^2} \bdec{0}{0}{1}\left[ \mathbb{E}\left[ \mathcal{I} \left(\widetilde{\lambda}
\overline{F}_{\gamma}^{\{x\}}(x + \sqrt{u}\mathbf{b}_{\cdot / u}; X)
\right)\right]1_{\mathcal{H}_k}\right]
\end{align}

\noindent where
\begin{align}
\notag
\overline{F}_{\gamma}^{\{x\}}(x + \sqrt{u}\mathbf{b}_{\cdot / u}; X)
& =  \int_0^u 1_{\{x+\sqrt{u}\mathbf{b}_{s/u} \in B(x, k\sqrt{u}) \}}\frac{e^{\gamma X(x+\sqrt{u}\mathbf{b}_{s/u}) - \frac{\gamma^2}{2} \mathbb{E}[X(x + \sqrt{u}\mathbf{b}_{s/u})^2]}ds}{|x + \sqrt{u}\mathbf{b}_{s/u}-x|^{\gamma^2}}\\
\label{eq:Ftilde_std}
& = u ^{1 - \frac{\gamma^2}{2}}\int_0^1 1_{\{\mathbf{b}_{s} \in B(0, k) \}}\frac{e^{\gamma X(x+\sqrt{u}\mathbf{b}_{s}) - \frac{\gamma^2}{2} \mathbb{E}[X(x + \sqrt{u}\mathbf{b}_{s})^2]}ds}{|\mathbf{b}_{s}|^{\gamma^2}}.
\end{align}

\noindent Let us quickly mention that the presence of the indicator inside the integrands in \eqref{eq:Ftilde_std} is not exactly consistent with our definition in \eqref{eq:overlineF} but it does not change anything. We are adopting this abuse of notation (here and elsewhere in the article) as a reminder for the reader that the corresponding random variable is analysed on the event $\mathcal{H}_k$.

We now want to proceed by invoking the scale invariance of $X(\cdot)$. For this purpose, let $\overline{X}(\cdot)$ be a log-correlated Gaussian field on $B(0,1)$ with covariance $ \mathbb{E}[\overline{X}(x_1) \overline{X}(x_2)] = -\log|x_1 - x_2| + 4$, and $B_{T_x(u; k)}$ an independent Gaussian random variable with zero mean and variance $T_x(u; k) := -\log \left(k\sqrt{u}/R(x; D)\right)$. (Note that $T_x(u;k) \ge 0$ since $k\sqrt{u}/R(x; D) \le k\sqrt{u}/d(x, \partial D)$ by Koebe quarter theorem and we are working under the condition $d(x, \partial D) \ge 4k\sqrt{u}$, and thus $B_{T_x(u;k)}$ is well-defined.) Then
\begin{align*}
\mathbb{E}[\overline{X}(x_1) \overline{X}(x_2)] + \mathbb{E} \left[B_{T_x(u; k)}^2\right]
& = - \log |x_1 - x_2| + 4 - \log\left(k\sqrt{u}/R(x; D)\right)\\
& = \mathbb{E}\left[X(x + k\sqrt{u} x_1) X(x + k\sqrt{u} x_2)\right]  \qquad \forall x_1, x_2 \in B(0, 1),
\end{align*}

\noindent i.e. we have
\begin{align*}
X(x+ k\sqrt{u}~ \cdot) \overset{d}{=} \overline{X}(\cdot) + B_{T_x(u;k)} \qquad \text{on $B(0, 1)$}.
\end{align*}

\noindent Substituting this into $\overline{F}_{\gamma}^{\{x\}}(x + \sqrt{u}\mathbf{b}_{\cdot / u}; X)$, \eqref{eq:Ftilde_std} becomes
\begin{align*}
& u^{1-\frac{\gamma^2}{2}}e^{\gamma B_{T_x(u; k)} - \frac{\gamma^2}{2} T_x(u; k)}\underbrace{\int_0^{1} 1_{\{\mathbf{b}_{s} \in B(0, k)\}}\frac{e^{\gamma \overline{X}(k^{-1}\mathbf{b}_s) - \frac{\gamma^2}{2} \mathbb{E}[\overline{X}( k^{-1}\mathbf{b}_s)^2]}ds}{|\mathbf{b}_s|^{\gamma^2}}}_{=:\overline{F}_{\gamma}(k^{-1}\mathbf{b}; \overline{X})}\\
& = e^{\gamma \left( B_{T_x(u; k)} - (Q-\gamma) T_{x}(u; k)\right)} \left(k/R(x; D)\right)^{-(2-\gamma^2)}\overline{F}_{\gamma}(k^{-1}\mathbf{b}; \overline{X})\\
&=: e^{\gamma \left( B_{T_x(u; k)} - (Q-\gamma) T_{x}(u; k)\right)} \mathcal{E}.
\end{align*}

\noindent where the law of $\mathcal{E} = \left[k/R(x; D)\right]^{-(2-\gamma^2)}\overline{F}_\gamma(k^{-1}\mathbf{b}; \overline{X})$ does not depend on $u$. Summarising all the work we have done from \eqref{eq:G1_uniform_step0} and \eqref{eq:G1_uniform_step1b}, we have
\begin{align*}
&  \int_{0}^{1} \frac{du}{2\pi u} 1_{\{d(x, \partial D) \ge  4k\sqrt{u}\}}\bdec{x}{x}{u}\left[ \mathbb{E}\left[ \mathcal{I} \left(\lambda  F_\gamma(\mathbf{b})\right)\right]1_{\mathcal{H}_k}\right]\\
& \qquad \le  e^{20\gamma^2} \mathbb{E} \otimes \bdec{0}{0}{1}\left[ \int_0^1 \frac{du}{2\pi u} 1_{\{d(x, \partial D) \ge  4k\sqrt{u}\}} 
 \mathcal{I} \left(\widetilde{\lambda}
\overline{F}_{\gamma}^{\{x\}}(x + \sqrt{u}\mathbf{b}_{\cdot / u}; X)
\right)1_{\mathcal{H}_k}\right]\\
& \qquad = e^{20\gamma^2} \mathbb{E} \otimes \bdec{0}{0}{1} \left[\int_{0}^{1} \frac{du}{2\pi u} 1_{\{d(x, \partial D) \ge  4k\sqrt{u}\}} \mathcal{I} \left(\widetilde{\lambda} \mathcal{E} e^{\gamma \left( B_{T_x(u; k)} - (Q-\gamma) T_{x}(u; k)\right)}\right) 1_{\mathcal{H}_k}\right]\\
&\qquad   \le \frac{e^{20\gamma^2}}{\pi} \int_0^\infty dt \mathbb{E} \otimes \bdec{0}{0}{1}\left[\mathcal{I}\left(\widetilde{\lambda} \mathcal{E} e^{\gamma(B_t - (Q-\gamma)t)}\right) 1_{\mathcal{H}_k}\right]
\end{align*}

\noindent where $(B_t)_{t \ge 0}$ is a Brownian motion. By \Cref{lem:main}, the last expression is bounded by 
\begin{align*}
& e^{20 \gamma^2} c_{\gamma} \bdec{0}{0}{1}\left[1_ {\mathcal{H}_k}\right]
\end{align*}
\noindent uniformly in $x \in D$ and $\widetilde{\lambda} > 0$, which concludes the proof.
\end{proof}

\begin{proof}[Proof of \Cref{lem:L1-boundary}]
Observe that
\begin{align*}
& \mathbb{E}\left[  \int_{0}^{1} \frac{du}{2\pi u} \bdec{x}{x}{u}[ \mathcal{I} \left(\lambda  F_\gamma^{\{x\}}(\mathbf{b})\right)1_{\{u< \tau_D(\mathbf{b}) \}}]\right]\\
& \qquad \le \sum_{k \ge 1} \mathbb{E}\left[  \int_{0}^{1} \frac{du}{2\pi u} \bdec{x}{x}{u} 1_{\{ d(x, \partial D) \ge 4k\sqrt{u}\}}[ \mathcal{I} \left(\lambda  F_\gamma^{\{x\}}(\mathbf{b})\right)1_{\mathcal{H}_k}]\right]\\
& \qquad  \quad + \sum_{k \ge 1} \mathbb{E}\left[  \int_{0}^{1} \frac{du}{2\pi u} \bdec{x}{x}{u} 1_{\{ d(x, \partial D) \le 4k\sqrt{u}\}}[ \mathcal{I} \left(\lambda  F_\gamma^{\{x\}}(\mathbf{b})\right)1_{\mathcal{H}_k}]\right].
\end{align*}

\noindent We already saw from \Cref{lem:G1_uniform} that the first sum is uniformly bounded in $x \in D$ and $\lambda > 0$. As for the second sum,
\begin{align*}
& \sum_{k\ge 1} \mathbb{E}\left[  \int_{0}^{1} \frac{du}{2\pi u} 1_{\{d(x, \partial D) \le  4k\sqrt{u}\}}\bdec{x}{x}{u}[ \mathcal{I} \left(\lambda  F_\gamma^{\{x\}}(\mathbf{b})\right)1_{\mathcal{H}_k}]\right]\\
&\qquad \le   \sum_{k\ge 1} \int_{[d(x, \partial D)/4k]^2}^1 \frac{du}{2\pi u}\pdec{x}{x}{u}\left(\mathcal{H}_k\right)
\end{align*}

\noindent which may be further bounded, using 
\begin{align*}
\pdec{x}{x}{u}(\mathcal{H}_k) \le \pdec{x}{x}{u}\left(\max_{s \le u} |\mathbf{b}_s - x| \ge (k-1) \sqrt{u}\right)
\end{align*}

\noindent and \Cref{cor:bb_bound}, by
\begin{align*}
& \sum_{k \ge 1} \frac{4}{\pi}e^{-\frac{(k-1)^2}{2}} \log \frac{4k}{d(x, \partial D)}
\le C \left(1 + \log \frac{1}{d(x, \partial D)}\right)
\end{align*}

\noindent for some $C \in (0, \infty)$ uniformly in $\lambda > 0$. In other words, the integrand on the RHS of \eqref{eq:pf_boundary1} is bounded by some function independent of $\lambda$ (and $\kappa$) that is integrable with respect to $R(x; D)^{\frac{\gamma^2}{2}} dx$. The statement of \Cref{lem:L1-boundary} now follows from dominated convergence.
\end{proof}

Let us also show that
\begin{lem}
For any fixed $\kappa > 0$, we have
\begin{align*}
&\limsup_{\lambda \to \infty} \mathbb{E}\left[\int_D1_{\{d(x, \partial D) \ge  \kappa\}} \mu_{\gamma}(dx) \int_{0}^1 \frac{du}{2\pi u} \bdec{x}{x}{u}[ \mathcal{I} \left(\lambda  F_{\gamma}(\mathbf{b})\right)1_{\{u \ge  \tau_D(\mathbf{b}) \}}]\right] = 0.
\end{align*}
\end{lem}

\begin{proof}
Note that for $x \in D$ satisfying $d(x, \partial D) \ge \kappa$,
\begin{align*}
\bdec{x}{x}{u}[ \mathcal{I} \left(\lambda  F_{\gamma}(\mathbf{b})\right)1_{\{u > \tau_D(\mathbf{b}) \}}]
& \le  \pdec{x}{x}{u}\left( \exists s \le u:  ~ \mathbf{b}_s  \in  \partial D \right)\\
& \le \pdec{x}{x}{u}\left( \max_{s \le u}  |\mathbf{b}_s-x|  \ge \kappa \right)
\le 4 e^{-\frac{\kappa^2}{2u}}
\end{align*}

\noindent by \Cref{cor:bb_bound}. Therefore,
\begin{align*}
\int_D1_{\{d(x, \partial D) \ge  \kappa\}} \mu_{\gamma}(dx) \int_{0}^1 \frac{du}{2\pi u} \bdec{x}{x}{u}[ \mathcal{I} \left(\lambda  F_{\gamma} (\mathbf{b})\right)1_{\{u > \tau_D(\mathbf{b}) \}}]
\le \mu_{\gamma}(D) \underbrace{\int_0^1 \frac{du}{u} e^{-\frac{\kappa^2}{2u}}}_{<\infty}
\end{align*}

\noindent which has finite first moment, and the claim follows from dominated convergence again.
\end{proof}

\subsection{Part I: $L^1$-estimates for bad event}\label{sec:pf_badevent}
We shall denote by $h_r(x)$ the circle average of the field over $\partial B(x, r)$. Let us introduce the notation
\begin{align}
\label{eq:event}
\mathcal{G}_I(x)
:= \bigg\{ h_{2^{-n}}(x) \le \alpha \log (2^n) \quad \forall n \in I \cap \mathbb{N}\bigg\}.
\end{align}

As in \cite{Ber2017}, the key is to be able to work on this good event. The issue is that Gaussian comparison and scale invariance are key to computations of moments, but these do not mix well with good events (essentially, the indicator of the good event cannot be written as some convex function of the mass of the chaos). We will replace this indicator by exponentials in the $L^1$ computation showing that bad events do not contribute significantly to the expectation, and will need arguments in the subsequent $L^2$ computation.

\begin{lem}\label{lem:bad_event}
Let $\alpha > \gamma$. Then
\begin{align}
\lim_{n \to \infty} \mathbb{E}\left[\int_{D} 1_{\mathcal{G}_{[n, \infty)}(x)^c} \mu_{\gamma}(dx) \right] 
&= 0\\
\text{and} \quad \lim_{n \to \infty} \limsup_{\lambda \to \infty} \mathbb{E}\left[\int_{D} 1_{\mathcal{G}_{[n, \infty)}(x)^c} \mu_{\gamma}(dx) \int_0^1 \frac{du}{2\pi u}\bdec{x}{x}{u}[ \mathcal{I} \left(\lambda  F_{\gamma}(\mathbf{b})\right)]
\right] 
&= 0.
\end{align}
\end{lem}

\begin{proof}
We only treat the second claim since the first one is simpler (and a similar statement was proved in \cite{Ber2017}). By \Cref{lem:L1-boundary}, it suffices to establish the analogous result with the domain of integration in the $x$-integral replaced by $\{x: d(x, \partial D) \ge \kappa\}$ for any $\kappa > 0$.

\medskip

Let us apply Fubini and Cameron-Martin again and rewrite
\begin{align*}
&  \mathbb{E}\left[\int_{\{d(x, \partial D) \ge \kappa\}} 1_{\mathcal{G}_{[n, \infty)}(x)^c} \mu_{\gamma}(dx) \int_0^1 \frac{du}{2\pi u} 
\bdec{x}{x}{u}[ \mathcal{I} \left(\lambda  F_{\gamma}(\mathbf{b}) \right)1_{\mathcal{H}_k}]
\right] \\
& \qquad =  \int_{\{d(x, \partial D) \ge \kappa\}} R(x; D)^{\frac{\gamma^2}{2}}dx \int_0^1 \frac{du}{2\pi u}\mathbb{E}\left[ 1_{\mathcal{G}^{\{x\}}_{[n, \infty)}(x)^c} 
\bdec{x}{x}{u}[ \mathcal{I} \left(\lambda  F_\gamma^{\{x\}}(\mathbf{b})\right) 1_{\mathcal{H}_k}]
\right]
\end{align*}

\noindent where $F_{\gamma}^{\{x\}}(\mathbf{b})$ was already defined in \eqref{eq:Girsanov1}, and for any finite set $\mathcal{S} \subset D$
\begin{align}\label{eq:eventGir}
\mathcal{G}_{I}^{\mathcal{S}}(x) := \bigg\{ h_{2^{-j}}(x)  + \gamma \sum_{z \in \mathcal{S}}\mathbb{E}[h_{2^{-j}}(x) h(z) ] \le \alpha \log (2^j) \quad \forall j \in I \cap \mathbb{N}\bigg\}.
\end{align}

Since $x$ is bounded away from $\partial D$, it follows from \Cref{lem:mollified_cov} that there exists some constant $C_\kappa > 0$ such that
\begin{align*}
\left|\mathbb{E}\left[h_{2^{-j}}(x) h_{\delta}(x)\right] + \log (2^{-j})\right| \le C_\kappa
\qquad \forall \delta \in [0, 2^{-j}], \qquad \forall j \ge n.
\end{align*}

\noindent In particular, for any $\beta > 0$ we have
\begin{align}
\notag
1_{\mathcal{G}^{\{x\}}_{[n, \infty)}(x)^c}
&\le \sum_{j \ge n} \exp\left( \beta [ h_{2^{-j}}(x) +  \gamma \mathbb{E}[h_{2^{-j}}(x) h(x) ] - \alpha \log (2^j)]\right) \\
\label{eq:ind_exp_trick0}
& \le e^{(\frac{\beta^2}{2} + \beta \gamma )C_{\kappa}}\sum_{j \ge n} 2^{-\frac{\beta}{2}\left[2(\alpha - \gamma ) - \beta\right]j} e^{\beta h_{2^{-j}}(x) - \frac{\beta^2}{2} \mathbb{E}[h_{2^{-j}}(x)^2]}
\end{align}

\noindent and thus
\begin{align*}
&\mathbb{E}\left[ 1_{\mathcal{G}^{\{x\}}_{[n, \infty)}(x)^c} \mathcal{I}\left(\lambda  F_{\gamma}^{\{x\}}(\mathbf{b})\right)\right]\\
& \qquad \le  e^{(\frac{\beta^2}{2} + \beta \gamma )C_{\kappa}}\sum_{j \ge n} 2^{-\frac{\beta}{2}\left[2(\alpha - \gamma ) - \beta\right]j} \mathbb{E} \left[e^{\beta h_{2^{-j}}(x) - \frac{\beta^2}{2} \mathbb{E}[h_{2^{-j}}(x)^2]}
\mathcal{I}\left(\lambda  F_{\gamma}^{\{x\}}(\mathbf{b})\right)
\right]\\
& \qquad =  e^{(\frac{\beta^2}{2} + \beta \gamma )C_{\kappa}}\sum_{j \ge n} 2^{-\frac{\beta}{2}\left[2(\alpha - \gamma ) - \beta\right]j} \mathbb{E} \left[ \mathcal{I}\left(\lambda  F_{\gamma, (j, \beta)}^{\{x\}}(\mathbf{b})\right)\right]
\end{align*}

\noindent where
\begin{align*}
F_{\gamma, (j,\beta)}^{\{x\}}(\mathbf{b})
:= \int_0^u e^{\gamma^2 G_0^D(x, \mathbf{b}_s) +\gamma \beta \mathbb{E}[h_{2^{-j}}(x) h(\mathbf{b}_s)]} F_{\gamma}(ds; \mathbf{b}).
\end{align*}

Next, let $\delta \in (0, \kappa / 100)$ and consider 
\begin{align*}
& \int_0^1 \frac{du}{2\pi u}\mathbb{E}\left[
\bdec{x}{x}{u}[ \mathcal{I} \left(\lambda  F_{\gamma, (j,\beta)}^{\{x\}}(\mathbf{b})\right) 1_{\mathcal{H}_k}]
\right]\\
&\qquad = \int_0^{\delta^2 k^{-2}2^{-2j}} \frac{du}{2\pi u}\mathbb{E}\left[ 
\bdec{x}{x}{u}[ \mathcal{I} \left(\lambda  F_{\gamma, (j,\beta)}^{\{x\}}(\mathbf{b})\right) 1_{\mathcal{H}_k}]
\right]\\
& \qquad \qquad +\int_{\delta^2 k^{-2}2^{-2j}}^1 \frac{du}{2 \pi u}\mathbb{E}\left[
\bdec{x}{x}{u}[ \mathcal{I} \left(\lambda F_{\gamma, (j,\beta)}^{\{x\}}(\mathbf{b})\right) 1_{\mathcal{H}_k}]
\right].
\end{align*}

\noindent The second term can be easily bounded by 
\begin{align*}
\int_{\delta^2 k^{-2}2^{-2j}}^1 \frac{du}{2\pi u} \pdec{x}{x}{u}\left(\mathcal{H}_k\right) \le  \pdec{0}{0}{1}\left(\mathcal{H}_k\right) \log( k 2^j / \delta).
\end{align*}

\noindent As for the first term, since (by \Cref{lem:mollified_cov} again, up to a redefinition of $C_\kappa$)
\begin{align*}
\left|\mathbb{E}[h_{2^{-j}}(x) h(z) ] + \log(2^{-j})\right| \le C_{\kappa} \qquad \forall z \in B(x, 2^{-j})
\end{align*}

\noindent and $\mathbf{b}_{\cdot} \in B(x, 2^{-j})$ on the event $\mathcal{H}_k$ (under the probability measure $\bdec{x}{x}{u}$ with $k\sqrt{u} \le 2^{-j}$), one obtains
\begin{align*}
e^{-\gamma \beta C_\kappa } F_{\gamma}^{\{x\}}(\mathbf{b})\le 2^{\gamma \beta j}  F_{\gamma, (j,\beta)}^{\{x\}}(\mathbf{b}) \le e^{\gamma \beta C_\kappa }F_{\gamma}^{\{x\}}(\mathbf{b})
\end{align*}

\noindent and hence
\begin{align*}
&  \int_0^{\delta^2 k^{-2}2^{-2j}} \frac{du}{2\pi u}\mathbb{E}\left[
\bdec{x}{x}{u}[ \mathcal{I} \left(\lambda  F_{\gamma, (j,\beta)}^{\{x\}}(\mathbf{b})\right) 1_{\mathcal{H}_k}]
\right]\\
& \qquad \le \int_0^{\delta^2 k^{-2}2^{-2j}}  \frac{du}{2\pi u}e^{2\gamma \beta C_\kappa}  \mathbb{E}\left[
\bdec{x}{x}{u}[ \mathcal{I} \left(\lambda e^{-\gamma \beta (C_\kappa + \log 2^{-j})}F_{\gamma}^{\{x\}}(\mathbf{b})\right) 1_{\mathcal{H}_k}]
\right]\\
& \qquad \le  e^{2\gamma \beta C_\kappa } \int_0^{\delta^2 k^{-2}2^{-2j}} \frac{du}{2\pi u}  \mathbb{E} \left[\bdec{x}{x}{u}[ \mathcal{I}\left(\widetilde{\lambda}F_{\gamma}^{\{x\}}(\mathbf{b})\right)1_{\mathcal{H}_k}]\right],
\qquad \widetilde{\lambda} := \lambda e^{-\gamma \beta (C_\kappa + \log 2^{-j})}.
\end{align*}

\noindent Since $4k\sqrt{u} \le 4\delta 2^{-j} \le \kappa$, the last expression can be bounded by $C \pdec{0}{0}{1}\left(\mathcal{H}_k\right)$ for some $C \in (0, \infty)$ uniformly in $\widetilde{\lambda}> 0$ and for all $x \in D$ satisfying $d(x, \partial D) \ge \kappa$ by \Cref{lem:G1_uniform}. 
Combining everything together, we have
\begin{align*}
& \int_0^1 \frac{du}{2\pi u}\mathbb{E}\left[ 
1_{\mathcal{G}^{\{x\}}_{[n, \infty)}(x)^c} 
\bdec{x}{x}{u}[ \mathcal{I} \left(\lambda  F_{\gamma}^{\{x\}}(\mathbf{b})\right)\right]\\
& \qquad \le \sum_{k \ge 1}
e^{(\frac{\beta^2}{2} + \beta \gamma )C_{\kappa}}\sum_{j \ge n} 2^{-\frac{\beta}{2}\left[2(\alpha - \gamma ) - \beta\right]j}  \Bigg\{
\int_0^{\delta^2 k^{-2}2^{-2j}} \frac{du}{2\pi u}\mathbb{E}\left[ 
\bdec{x}{x}{u}[ \mathcal{I} \left(\lambda  F_{\gamma, (j,\beta)}^{\{x\}}(\mathbf{b})\right) 1_{\mathcal{H}_k}]
\right]\\
& \qquad \qquad \qquad \qquad +\int_{\delta^2 k^{-2}2^{-2j}}^1 \frac{du}{2 \pi u}\mathbb{E}\left[
\bdec{x}{x}{u}[ \mathcal{I} \left(\lambda F_{\gamma, (j,\beta)}^{\{x\}}(\mathbf{b})\right) 1_{\mathcal{H}_k}]
\right]
\Bigg\}\\
& \qquad \le (C + \log \delta^{-1})  e^{(\frac{\beta^2}{2} + \beta \gamma )C_{\kappa}}\left[\sum_{k \ge 1}  k\pdec{0}{0}{1}\left(\mathcal{H}_k\right)\right]  
\left[\sum_{j \ge n} j 2^{-\frac{\beta}{2}\left[2(\alpha - \gamma ) - \beta\right]j}\right] \\
& \qquad =: \widetilde{C} \sum_{j \ge n} j 2^{-\frac{\beta}{2}\left[2(\alpha - \gamma ) - \beta\right]j} 
\end{align*}

\noindent where $\widetilde{C} \in (0, \infty)$ is independent of $n \in \mathbb{N}$ or $\lambda > 0$, uniformly for $d(x, \partial D) \ge \kappa$. Choosing $\beta = \alpha - \gamma > 0$, the above bound is summable and vanishes as $n \to \infty$ uniformly. Hence,
\begin{align*}
\limsup_{n \to \infty} \limsup_{\lambda \to \infty} \mathbb{E}\left[\int_{\{d(x, \partial D) \ge \kappa\}} 1_{\mathcal{G}_{[n, \infty)}(x)^c} \mu_{\gamma}(dx) \int_0^1 \frac{du}{2\pi u} 
\bdec{x}{x}{u}[ \mathcal{I} \left(\lambda  F_{\gamma}(\mathbf{b}) \right)]
\right] =0
\end{align*}

\noindent for any $\kappa > 0$, which concludes the proof.
\end{proof}

\paragraph{Roadmap for the remaining analysis in \Cref{sec:proofWeyl}.} 
Based on all the estimates that have appeared in the current section, \Cref{theo:Weyl} can be established if we can show, for any $\kappa > 0$ and $n_0 = n_0(\kappa) \in \mathbb{N}$ sufficiently large that
\begin{align}\label{eq:goalWeyl}
\lim_{\lambda \to \infty}
\mathbb{E}\Bigg[
\left|\int_A \mu_{\gamma}^{\kappa, n_0}(dx) \int_0^\infty \frac{du}{2\pi u} \bdec{x}{x}{u}[ \mathcal{I} \left(\lambda  F_{\gamma}(\mathbf{b})\right)] -c_{\gamma}\mu_{\gamma}^{\kappa, n_0}(A)\right|^2
\Bigg] = 0
\end{align}

\noindent where $\mu_\gamma^{\kappa, n_0}(A) := \int_A \mu_\gamma^{\kappa, n_0}(dx)$ with
\begin{align}
\mu_\gamma^{\kappa, n_0}(dx) &:=  1_{\{d(x, \partial D) \ge \kappa\}} 1_{\mathcal{G}_{[n_0, \infty)}(x)}\mu(dx).
\end{align}

\noindent Expanding the second moment on the LHS of \eqref{eq:goalWeyl}, it suffices to verify the following claim.
\begin{lem}\label{lem:DCT_limits}
For any $\kappa > 0$ and $n_0 \in \mathbb{N}$ such that $2^{1-n_0} < \kappa$, we have
\begin{align}
\label{eq:cross_limit}
\lim_{\lambda \to \infty} \mathbb{E}\left[
\mu_\gamma^{\kappa, n_0}(A) \int_A \mu_\gamma^{\kappa, n_0}(dx) \int_0^1 \frac{du}{2\pi u} 
\bdec{x}{x}{u}[ \mathcal{I} \left(\lambda  F_{\gamma}(\mathbf{b}) \right)]
\right]
&= c_\gamma \mathbb{E}\left[\mu_\gamma^{\kappa, n_0}(A)^2\right]\\
\label{eq:diagonal_limit}
\text{and} \qquad 
\lim_{\lambda \to \infty} \mathbb{E}\left[
\left( \int_A \mu_\gamma^{\kappa, n_0}(dx)  \int_0^1 \frac{du}{2\pi u} 
\bdec{x}{x}{u}[ \mathcal{I} \left(\lambda  F_{\gamma}(\mathbf{b}) \right)]
\right)^2\right]
&= c_\gamma^2 \mathbb{E}\left[\mu_\gamma^{\kappa, n_0}(A)^2\right].
\end{align}
\end{lem}

It is standard to check that the right hand sides of \eqref{eq:cross_limit} and \eqref{eq:diagonal_limit} are finite. Our approach to \Cref{lem:DCT_limits} will be based on a dominated convergence argument. More specifically, we shall apply Fubini/Cameron-Martin to rewrite the LHS's of \eqref{eq:cross_limit} and \eqref{eq:diagonal_limit} as some integrals over $A \times A$, and then provide uniform estimates and evaluate pointwise limits for the integrands in order to conclude the desired results. The analysis of the cross term \eqref{eq:cross_limit} will be performed in \Cref{sec:pf_cross}, and that of the diagonal term \eqref{eq:diagonal_limit} in the subsequent \Cref{sec:pf_diagonal}.

\subsection{Part II: analysis of cross term \eqref{eq:cross_limit}} \label{sec:pf_cross}
As explained just now, our proof of \eqref{eq:cross_limit} starts with an application of Fubini and Cameron-Martin theorem: we have
\begin{align}
\notag
& \mathbb{E}\left[
\int_{A \times A} \mu_\gamma^{\kappa, n_0}(dy)  \mu_\gamma^{\kappa, n_0}(dx) \int_0^1 \frac{du}{2\pi u} 
\bdec{x}{x}{u}[ \mathcal{I} \left(\lambda  F_\gamma(\mathbf{b}) \right)]
\right]\\
\notag 
& \quad = \int_{A \times A} 1_{\{d(x, \partial D) \ge \kappa \}}1_{\{d(y, \partial D) \ge \kappa \}}R(x; D)^{\frac{\gamma^2}{2}}R(y; D)^{\frac{\gamma^2}{2}}e^{\gamma^2 G_0^D(x, y)} dxdy\\
\label{eq:cross_preDCT}
&\qquad \qquad \times \mathbb{E} \left[ 1_{\mathcal{G}^{\{x, y\}}_{[n_0, \infty)}(x) \cap \mathcal{G}^{\{x, y\}}_{[n_0, \infty)}(y)} \int_0^1 \frac{du}{2\pi u} 
\bdec{x}{x}{u}\left[ \mathcal{I} \left(\lambda  F_{\gamma}^{\{x, y\}}(\mathbf{b}) \right)\right]
\right] 
\end{align}

\noindent where (recalling \eqref{eq:Girsanov1} and \eqref{eq:eventGir}) 
\begin{align}
\begin{split}
\label{eq:Girsanov2}
F_\gamma^{\{x, y\}}(\mathbf{p}) 
&= \int_0^{\ell(\mathbf{p})}  e^{\gamma^2 [G_0^D(x, \mathbf{p}_s) + G_0^D(y, \mathbf{p}_s)]} F_\gamma(ds; \mathbf{p})\\
\text{and}\qquad 
\mathcal{G}_{I}^{\{x, y\}}(\cdot) 
& = \bigg\{ h_{2^{-k}}(\cdot)  + \gamma  \mathbb{E} \left[ h_{2^{-k}}(\cdot) \left(h(x)+ h(y)\right)\right]\le \alpha \log (2^k) \quad \forall k \in I \cap \mathbb{N}\bigg\}.
\end{split}
\end{align}

In order to apply dominated convergence to \eqref{eq:cross_preDCT} and \eqref{eq:diagonal_preDCT}, we have to establish integrable upper bounds (with respect to $e^{\gamma^2 G_0^D(x, y)}  \asymp |x-y|^{-\gamma^2}$) as well as pointwise limits (as $\lambda \to \infty$) of the expectation on the RHS of \eqref{eq:cross_preDCT}. 

\subsubsection{Uniform estimate for the cross term}
Recall the assumption that $\mathrm{diam}(D) < \frac{1}{2}$, which in particular implies that $-\log|x-y| > 0$ for any distinct $x, y \in D$.
\begin{lem}\label{lem:cross_uniform}
Let $\beta > 0$ and $n_0 \in \mathbb{N}$ satisfying $2^{1-n_0} < \kappa$. Then there exists some constant $C = C(\kappa, n_0, \gamma, \alpha, \beta) \in (0, \infty)$ such that 
\begin{align}
\notag
& \mathbb{E} \left[ 
1_{\mathcal{G}^{\{x, y\}}_{[n_0, \infty)}(x) \cap \mathcal{G}^{\{x, y\}}_{[n_0, \infty)}(y)}
 \int_0^1 \frac{du}{2\pi u} 
\bdec{x}{x}{u}[ \mathcal{I} \left(\lambda  F_{\gamma}^{\{x, y\}}(\mathbf{b}) \right)]
\right] \\
 \label{eq:cross_uniform}
&\qquad \qquad  \le C \left(1 - \log|x-y|\right) |x-y|^{(2\gamma - \alpha)\beta- \frac{\beta^2}{2}}
\end{align}

\noindent uniformly in $\lambda > 0$ and $x, y \in D$ satisfying $d(x, \partial D) \wedge d(y, \partial D) \ge \kappa$.
\end{lem}

Observe that the bound \eqref{eq:cross_uniform} is integrable if one chooses $\alpha$ sufficiently close to $\gamma \in (0, \sqrt{2d})$ and $\beta = 2\gamma - \alpha$ such that $(2\gamma - \alpha)^2 / 2 < d$.

\begin{proof}
Similar to the proof of \Cref{lem:bad_event}, we will consider
\begin{align*}
\bdec{x}{x}{u}\left[ \mathcal{I} \left(\lambda  F_{\gamma}^{\{x, y\}}(\mathbf{b})\right)\right]
= \sum_{k \ge 1} 
\bdec{x}{x}{u}\left[ \mathcal{I} \left(\lambda  F_{\gamma}^{\{x, y\}}(\mathbf{b})\right) 1_{\mathcal{H}_k}\right]
\end{align*}

\noindent and split our analysis into two cases, depending on the distance between $x$ and $y$.

\paragraph{Case 1:} $|x-y| \ge 2^{-n_0}$. Using the observation that
\begin{align*}
& \mathbb{E} \left[ 
1_{\mathcal{G}^{\{x, y\}}_{[n_0, \infty)}(x) \cap \mathcal{G}^{\{x, y\}}_{[n_0, \infty)}(y)}
\int_{(k2^{n_0+1})^{-2}}^{1} \frac{du}{2\pi u}
\bdec{x}{x}{u}\left[ \mathcal{I} \left(\lambda  F_{\gamma}^{\{x, y\}}(\mathbf{b})\right) 1_{\mathcal{H}_k}\right]
 \right] \\
& \qquad \qquad \le \int_{(k2^{n_0+1})^{-2}}^{1} \frac{du}{2\pi u} \pdec{x}{x}{u}\left(\mathcal{H}_k\right)
\le \pdec{0}{0}{1}\left(\mathcal{H}_k\right) \log\left(k2^{n_0+1}\right)
\end{align*}

\noindent which is summable in $k$,  it suffices to show that the sum
\begin{align}\label{eq:cross_uniform_case1}
& \sum_{k \ge 1} \mathbb{E} \left[ 
1_{\mathcal{G}^{\{x, y\}}_{[n_0, \infty)}(x) \cap \mathcal{G}^{\{x, y\}}_{[n_0, \infty)}(y)}
\int_0^{(k2^{n_0+1})^{-2}} \frac{du}{2\pi u}
\bdec{x}{x}{u}\left[ \mathcal{I} \left(\lambda  F_{\gamma}^{\{x, y\}}(\mathbf{b})\right) 1_{\mathcal{H}_k}\right]
 \right] 
\end{align}

\noindent is bounded with the desired uniformity in the statement of \Cref{lem:cross_uniform}.

Recall on the event $\mathcal{H}_k$ (and under the probability measure $\pdec{x}{x}{u}$) that $\mathbf{b}_\cdot \in B(x, k\sqrt{u}) \subset B(x, 2^{-(n_0+1)})$. By the continuity of the Green's function away from the diagonal, there exists some $C_{D}(n_0) < \infty$ such that
\begin{align*}
|G_0^D(y, \mathbf{b}_s)| \le C_{D}(n_0) \qquad \forall s \le u \le (k2^{n_0+1})^{-2}
\end{align*}

\noindent since $|y-\mathbf{b}_s| \ge |x-y| - |x - \mathbf{b}_s| \ge 2^{-(n_0+1)}$. In particular, for any $u \in [0, (k2^{n_0+1})^{-2}]$ we have
\begin{align*}
e^{-\gamma^2 C_D(n_0)} F_{\gamma}^{\{x\}}(\mathbf{b}) 
\le F_{\gamma}^{\{x, y\}}(\mathbf{b})
\le e^{\gamma^2 C_D(n_0)}F_{\gamma}^{\{x\}}(\mathbf{b})
\end{align*}

\noindent and hence
\begin{align*}
\mathcal{I}\left(\lambda F_{\gamma}^{\{x, y\}}(\mathbf{b})\right) \le e^{2\gamma^2 C_D(n_0)}\mathcal{I}\left( \widetilde{\lambda} F_{\gamma}^{\{x\}}(\mathbf{b})\right)
\end{align*}

\noindent with $\widetilde{\lambda} := \lambda e^{-\gamma^2 C_D(n_0)}$. Therefore, the sum \eqref{eq:cross_uniform_case1} can be upper bounded by
\begin{align*}
&e^{2\gamma^2 C_D(n_0)} \sum_{k \ge 1} \mathbb{E} \left[\int_0^{(k2^{n_0+1})^{-2}} \frac{du}{2\pi u}
\bdec{x}{x}{u}[ \mathcal{I} \left(\widetilde{\lambda} F_{\gamma}^{\{x\}}(\mathbf{b}) \right) 1_{\mathcal{H}_k}]
 \right].
\end{align*}

\noindent This may be further bounded uniformly in $\widetilde{\lambda} > 0$ with \Cref{lem:G1_uniform}, which is applicable since
\begin{align*}
u \le (k2^{n_0+1})^{-2} 
\quad \Rightarrow \quad 
4k\sqrt{u} \le 2^{1-n_0} < \kappa \le d(x, \partial D).
\end{align*}

\paragraph{Case 2:} $|x-y| < 2^{-n_0}$. Using \Cref{lem:mollified_cov}, there exists some constant $C_\kappa \in (0, \infty)$ such that for any $\epsilon, \delta > 0$,
\begin{align}\label{eq:mollified_cov_precise}
\left|\mathbb{E}[h_\epsilon(a) h_{\delta}(b)] + \log \left( |a-b| \vee \epsilon \vee \delta \right) \right| &\le C_\kappa
\end{align}

\noindent uniformly for all $a, b \in D$ bounded away from $\partial D$ by at least a distance of $\kappa / 2$. If we let $n_0 \le n \in \mathbb{N}$ satisfy $2^{-(n+1)} \le |x-y| < 2^{-n}$, then
\begin{align*}
\mathcal{G}^{\{x, y\}}_{[n_0, \infty)}(x) \cap \mathcal{G}^{\{x, y\}}_{[n_0, \infty)}(y)
&\subset \bigg\{ h_{2^{-n}}(x) + \gamma \mathbb{E} \left[ h_{2^{-n}}(x) \left(h(x) + h(y)\right)\right]\le \alpha \log (2^{n_0})\bigg\}\\
& \subset \bigg\{ h_{2^{-n}}(x) \le (\alpha - 2\gamma) \log (2^{n}) + 2 C_\kappa \bigg\}.
\end{align*}

\noindent In particular, for any $\beta > 0$ we have
\begin{align}
\notag
1_{\mathcal{G}^{\{x, y\}}_{[n_0, \infty)}(x) \cap \mathcal{G}^{\{x, y\}}_{[n_0, \infty)}(y)}
& \le \exp \left\{-\beta \bigg[h_{2^{-n}}(x) - (\alpha - 2\gamma) \log (2^{n})  -2 C_\kappa\bigg]\right\}\\
\notag
& = e^{2\beta C_\kappa} e^{\beta (\alpha - 2\gamma) \log(2^{n}) + \frac{\beta^2}{2} \mathbb{E}[h_{2^{-n}}(x)^2]}  e^{-\beta h_{2^{-n}}(x) - \frac{\beta^2}{2} \mathbb{E}[h_{2^{-n}}(x)^2]}\\
\label{eq:ind_exp_trick}
& \le \widetilde{C} |x-y|^{(2\gamma - \alpha)\beta- \frac{\beta^2}{2}} e^{-\beta h_{2^{-n}}(x) - \frac{\beta^2}{2} \mathbb{E}[h_{2^{-n}}(x)^2]}
\end{align}

\noindent for some constant $\widetilde{C} = \widetilde{C}(\kappa, \gamma, \alpha, \beta) \in (0, \infty)$. Substituting this into the LHS of \eqref{eq:cross_uniform} and applying Cameron-Martin theorem, we see that
\begin{align*}
& \mathbb{E} \left[ 
1_{\mathcal{G}^{\{x, y\}}_{[n_0, \infty)}(x) \cap \mathcal{G}^{\{x, y\}}_{[n_0, \infty)}(y)}
 \int_0^1 \frac{du}{2\pi u} 
\bdec{x}{x}{u}[ \mathcal{I} \left(\lambda   F_{\gamma}^{\{x, y\}}(\mathbf{b}) \right)]
 \right] \\
& \qquad \qquad\le  \widetilde{C} |x-y|^{(2\gamma - \alpha)\beta- \frac{\beta^2}{2}}\mathbb{E} \left[\int_0^1 \frac{du}{2\pi u}
\bdec{x}{x}{u}\left[ \mathcal{I} \left(\lambda  F_{\gamma, (n, -\beta)}^{\{x, y\}}(\mathbf{b}) \right)\right]
\right]
\end{align*}

\noindent where 
\begin{align}
\label{eq:beta_mu_G1}
F_{\gamma, (n, -\beta)}^{\{x, y\}}(\mathbf{b}) 
& = \int_0^{\ell(\mathbf{b})}  e^{\gamma^2 [G_0^D(x, \mathbf{b}_s) + G_0^D(y, \mathbf{b}_s)] - \beta \gamma \mathbb{E}[h(\mathbf{b}_s) h_{2^{-n}}(x)]} F_{\gamma}(ds; \mathbf{b}).
\end{align}

Let us consider
\begin{align*}
&\mathbb{E} \left[\int_0^1 \frac{du}{2\pi u}
\bdec{x}{x}{u}\left[ \mathcal{I} \left(\lambda  F_{\gamma, (n, -\beta)}^{\{x, y\}}(\mathbf{b})  \right)\right]
\right]\\
& \qquad \le \sum_{k \ge 1} \mathbb{E} \left[\int_{( |x-y|/4k)^2}^1 \frac{du}{2\pi u}
\bdec{x}{x}{u}\left[ \mathcal{I} \left(\lambda  F_{\gamma, (n, -\beta)}^{\{x, y\}}(\mathbf{b})  \right)1_{\mathcal{H}_k}\right]
\right]\\
& \qquad \qquad +  \sum_{k \ge 1} \mathbb{E} \left[\int_0^{( |x-y|/4k)^2} \frac{du}{2\pi u}
\bdec{x}{x}{u}\left[ \mathcal{I} \left(\lambda  F_{\gamma, (n, -\beta)}^{\{x, y\}}(\mathbf{b})  \right)1_{\mathcal{H}_k}\right]
\right]
\end{align*}

\noindent and show that they are bounded with the desired uniformity, from which we can conclude the proof. The first sum on the RHS is easily bounded by
\begin{align*}
\sum_{k \ge 1} \mathbb{E} \left[\int_{( |x-y|/4k)^2}^1 \frac{du}{2\pi u}
\bdec{x}{x}{u}[1_{\mathcal{H}_k}]
\right]
\le
 \sum_{k \ge 1} \left[ -\log |x-y| + \log(4k) \right]\pdec{0}{0}{1}(\mathcal{H}_k)
\end{align*}

\noindent and when multiplied by $|x-y|^{(2\gamma - \alpha)\beta- \frac{\beta^2}{2}}$ satisfies a bound of the form \eqref{eq:cross_uniform}. As for the second sum, note that
\begin{align*}
u \le \left(\frac{|x-y|}{4k}\right)^2 \quad \Rightarrow \quad 4k\sqrt{u} \le |x-y| < 2^{-n_0} < \frac{1}{2}\kappa \le d(x, \partial D),
\end{align*}

\noindent and we would like to follow arguments similar to those in Case 1 and apply \Cref{lem:G1_uniform}. To do so, first observe on the event $\mathcal{H}_k$ that 
\begin{align*}
\mathbf{b}_s \in B(x, k\sqrt{u}) \subset B(x, |x-y|/ 4)
\end{align*}

\noindent and in particular $d(\mathbf{b}_s, \partial D) \ge \kappa /2$ for all $s \ge 0$. The estimate \eqref{eq:mollified_cov_precise} then implies
\begin{align*}
\left|G_0^D(y, \mathbf{b}_s) + \log|y-\mathbf{b}_s| \right| &\le C_\kappa\\
\text{and} \qquad \left|\mathbb{E}\left[h(\mathbf{b}_s) h_{2^{-n}}(x) \right] + \log (2^{-n})\right| & \le C_\kappa
\end{align*}

\noindent for the entire duration of the Brownian bridge $\mathbf{b}$. Since there exists some absolute constant $C > 0$ such that
\begin{align*}
\max \left\{ \left| \log |y-\mathbf{b}_s| - \log |x-y| \right|, \left| \log(2^{-n}) - \log |x-y| \right| \right\} \le C,
\end{align*}

\noindent we see (from \eqref{eq:beta_mu_G1}) that there exists some constant $\widehat{C} = \widehat{C}(\kappa, \beta, \gamma) \in (0, \infty)$ such that
\begin{align}\label{eq:unbeta_mu1}
\widehat{C}^{-1}
F_{\gamma}^{\{x\}}(\mathbf{b})
\le |x-y|^{-\gamma(\beta-\gamma)} F_{\gamma, (n, -\beta)}^{\{x, y\}}(\mathbf{b})
\le \widehat{C}  F_{\gamma}^{\{x\}}(\mathbf{b}).
\end{align}

\noindent Gathering all the work so far, we arrive at
\begin{align*}
& \sum_{k \ge 1} \mathbb{E} \left[\int_0^{( |x-y|/4k)^2} \frac{du}{2\pi u}
\bdec{x}{x}{u}\left[ \mathcal{I} \left(\lambda F_{\gamma, (n, -\beta)}^{\{x, y\}}(\mathbf{b}) \right)1_{\mathcal{H}_k}\right]
\right]\\
& \qquad \le 
\widehat{C}^2\sum_{k \ge 1} \mathbb{E} \left[\int_0^{( |x-y|/4k)^2} \frac{du}{2\pi u}
\bdec{x}{x}{u}\left[ \mathcal{I} \left(\widehat{\lambda}  F_{\gamma}^{\{x\}}(\mathbf{b}) \right)1_{\mathcal{H}_k}\right]
\right]
\end{align*}

\noindent where $\widehat{\lambda} := \lambda \widehat{C}^{-1} |x-y|^{\gamma (\beta-\gamma)}$. This expression is uniformly bounded in $\widehat{\lambda} > 0$ by \Cref{lem:G1_uniform} and we are done.
\end{proof}

\subsubsection{Pointwise limit of the cross term} \label{sec:cross_pointwise}
We now argue that
\begin{lem}\label{lem:cross_pointwise}
For any fixed $n_0 \in \mathbb{N}$ satisfying $2^{1-n_0} < \kappa$,
\begin{align}\label{eq:L1_cross_pointwise}
\begin{split}
& \lim_{\lambda \to \infty} 
\mathbb{E} \left[ 
1_{\mathcal{G}^{\{x, y\}}_{[n_0, \infty)}(x) \cap \mathcal{G}^{\{x, y\}}_{[n_0, \infty)}(y)}
\int_0^1 \frac{du}{2\pi u} 
\bdec{x}{x}{u} \left[\mathcal{I}\left(\lambda F_{\gamma}^{\{x, y\}}(\mathbf{b}) \right) \right]
\right]\\
& \qquad = c_\gamma
\mathbb{P} \left(\mathcal{G}^{\{x, y\}}_{[n_0, \infty)}(x) \cap \mathcal{G}^{\{x, y\}}_{[n_0, \infty)}(y)\right)
\end{split}
\end{align}

\noindent for any distinct points $x, y \in D$ satisfying $d(x, \partial D) \wedge d(y, \partial D) \ge \kappa$ and $-\log_2|x-y| \not \in \mathbb{N}$.
\end{lem}

The proof of the above lemma relies on a similar claim with an extra cutoff: 
\begin{lem}\label{lem:cross_pointwise_cutoff}
Under the same setting as \Cref{lem:cross_pointwise}, for any integer $m > 3 + \max(n_0, -\log_2 |x-y|)$ sufficiently large,
\begin{align}
\label{eq:L1_cross_pointwise_cutoff}
\begin{split}
&\lim_{\lambda \to \infty} 
\mathbb{E} \left[ 
1_{\mathcal{G}^{\{x, y\}}_{[n_0, m)}(x) \cap \mathcal{G}^{\{x, y\}}_{[n_0, m)}(y)}
\int_0^1 \frac{du}{2\pi u} 
\bdec{x}{x}{u} \left[\mathcal{I}\left(\lambda F_{\gamma}^{\{x, y\}}(\mathbf{b}) \right) \right]
\right]\\
&\qquad =  c_\gamma
\mathbb{P} \left(\mathcal{G}^{\{x, y\}}_{[n_0, m)}(x) \cap \mathcal{G}^{\{x, y\}}_{[n_0, m)}(y)\right).
\end{split}
\end{align}
\end{lem}

\begin{proof}
Let us fix some $\delta \in (0, 2^{-m})$ sufficiently small, and for each $k \in \mathbb{N}$ define
\begin{align}
\label{eq:cross_ptwise_main}
I_k& := \mathbb{E} \left[  
1_{\mathcal{G}^{\{x, y\}}_{[n_0, m)}(x) \cap \mathcal{G}^{\{x, y\}}_{[n_0, m)}(y)}
\int_0^{(\delta/k)^2}  \frac{du}{2\pi u} 
\bdec{x}{x}{u}\left[ \mathcal{I}\left(\lambda F_{\gamma}^{\{x, y\}}(\mathbf{b})\right) 1_{\mathcal{H}_k}\right]
\right], \\
\label{eq:cross_ptwise_truncate}
\text{and} \quad 
I_k^c & := \mathbb{E} \left[  
1_{\mathcal{G}^{\{x, y\}}_{[n_0, m)}(x) \cap \mathcal{G}^{\{x, y\}}_{[n_0, m)}(y)}
\int_{(\delta/k)^2}^1  \frac{du}{2\pi u} 
\bdec{x}{x}{u}\left[ \mathcal{I}\left(\lambda F_{\gamma}^{\{x, y\}}(\mathbf{b})\right) 1_{\mathcal{H}_k}\right]
\right].
\end{align}

\noindent  Our goal is to show that
\begin{align*}
\lim_{\lambda \to \infty} \sum_{k \ge 1} I_k^c = 0
\qquad \text{and}\qquad 
\lim_{\lambda \to \infty} \sum_{k \ge 1} I_k =  c_\gamma \mathbb{P} \left(\mathcal{G}_{[n_0, m)}^{\{x, y\}}(x) \cap  \mathcal{G}_{[n_0, m)}^{\{x, y\}}(y)\right).
\end{align*}

\paragraph{Bounding the residual terms $I_k^c$.} Using \Cref{cor:bb_bound},
\begin{align*}
I_k^c \le 
\int_{(\delta/k)^2}^1 \frac{du}{2\pi u} \pdec{x}{x}{u}\left(\mathcal{H}_k\right)
\le -2e^{-\frac{1}{2}(k-1)^2} \log(\delta/k)
\end{align*}

\noindent which is summable in $k \in \mathbb{N}$ uniformly in $\lambda > 0$. Arguing as before using the fact that $1 \ge \mathcal{I}(\lambda F_{\gamma}^{\{x, y\}}(\mathbf{b})) \to 0$ as $\lambda \to \infty$, we obtain $\lim_{\lambda \to \infty} I_k^c = 0$ and $\lim_{\lambda \to \infty} \sum_{k \ge 1} I_k^c = 0$ by two applications of dominated convergence.

\paragraph{Gaussian comparison.}
We now treat the main term $I_k$. By a change of variable, recall
\begin{align*}
F_{\gamma}^{\{x, y\}}(\mathbf{b})
&= \int_0^u  e^{\gamma^2 [G_0^D(x, \mathbf{b}_s) + G_0^D(y, \mathbf{b}_s)]} F_\gamma(ds; \mathbf{b})\\
&= u \int_0^1  e^{\gamma^2 [G_0^D(x, \mathbf{b}_{s/u}) + G_0^D(y, \mathbf{b}_{s/u})]} F_\gamma(ds; \mathbf{b}_{\cdot / u})
= uF_\gamma^{\{x, y\}}(\mathbf{b}_{\cdot / u}).
\end{align*}

\noindent Writing everything in terms of standardised Brownian bridge, we have
\begin{align*}
\bdec{x}{x}{u}[ \mathcal{I}\left(\lambda F_{\gamma}^{\{x, y\}}(\mathbf{b}) \right) 1_{\mathcal{H}_k}]
= \bdec{0}{0}{1} \left[\mathcal{I}\left(\lambda u F_{\gamma}^{\{x, y\}}(x+\sqrt{u}\mathbf{b})\right) 1_{\mathcal{H}_k} \right]
\end{align*}

\noindent and hence
\begin{align}\label{eq:Ik_rescaled1}
I_k := \mathbb{E} \otimes \bdec{0}{0}{1} \left[
1_{\mathcal{G}^{\{x, y\}}_{[n_0, m)}(x) \cap \mathcal{G}^{\{x, y\}}_{[n_0, m)}(y)}
1_{\mathcal{H}_k}
\int_0^{(\delta/k)^2}  \frac{du}{2\pi u} \mathcal{I}\left(\lambda u F_{\gamma}^{\{x, y\}}(x+\sqrt{u}\mathbf{b})\right) 
\right].
\end{align}

Set $\eta = 4\cdot2^{-m}< \frac{|x-y|}{2} \wedge \frac{\kappa}{2}$ so that the balls $B(x, \eta), B(y, \eta)$ are disjoint and contained in our domain $D$.  Since $ 0 < -\log_2|x-y| \not\in\mathbb{N}$, there exists some $d_{x, y} \in \mathbb{N}$ such that $2^{-d_{x, y}} < |x-y| < 2^{-d_{x, y}+1}$, and it is possible to pick $m$ sufficiently large so that
\begin{align}\label{eq:m_condition}
|x-y| - \eta > 2^{-d_{x, y}} \qquad \text{and} \qquad |x-y| + \eta < 2^{-d_{x, y} + 1}.
\end{align}

\begin{figure}[h!] 
\centering
\begin{tikzpicture}[scale=0.6, every node/.style={scale=0.8}]
\draw (-6+0.5,0) arc (0:360:0.5);
\draw (-6+2.5,0) arc (0:360:2.5);
\draw [fill] (-6, 0) circle [radius=0.05] node[right]{$y$};
\draw[<->] (-6.5, -0.5) -- (-6.5, 0);
\node at (-7.5, -0.2) {$\delta < 2^{-m}$};
\draw[<->] (-6, 0) -- (-6, 2.5);
\node at (-5.1, 1) {$\eta = 4 \cdot 2^{-m}$};

\draw (6+0.5,0) arc (0:360:0.5);
\draw (6+2.5,0) arc (0:360:2.5);
\draw [fill] (6, 0) circle [radius=0.05] node[right]{$x$};

\draw[dotted, thick] (-1.5,0) arc (180:150:7);
\draw[dotted, thick]  (-1.5,0) arc (180:210:7);

\draw[thick] (-9,0) arc (180:160:15);
\draw[thick] (-9,0) arc (180:200:15);

\draw[dashed] (-1.5, -5) -- (-1.5, 5);
\draw[<->] (-1.5, 1) -- (6, 1);
\node at (1.5, 1.5) {$2^{-d_{x, y}}$};

\draw[dashed] (-9, -5) -- (-9, 5);
\draw[<->] (-9, 3) -- (6, 3);
\node at (1.5, 3.5) {$2^{-d_{x, y} + 1}$};

\draw[dashed] (-3.5, -5) -- (-3.5, 5);
\draw[<->] (-3.5, -1) -- (6, -1);
\node at (1.5, -1.5) {$|x-y| - \eta > 2^{-d_{x, y}}$ };

\draw[dashed] (-8.5, -5) -- (-8.5, 5);
\draw[<->] (-8.5, -3) -- (6, -3);
\node at (2, -3.5) {$|x-y| + \eta < 2^{-d_{x, y} + 1}$};

\end{tikzpicture}

\caption{\label{fig:scales}comparison of different scales.}

\end{figure}
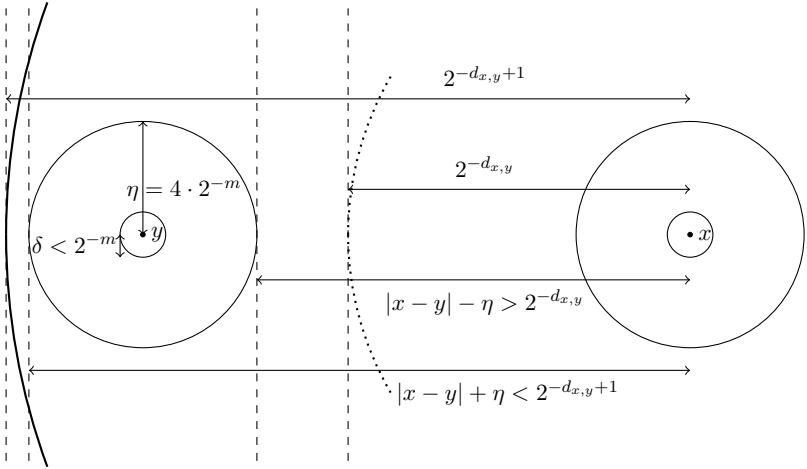

We apply the domain Markov property of Gaussian free field on $B(x, \eta) \cup B(y,\eta)$ and perform the decomposition
\begin{align}\label{eq:GFF_markov}
h(\cdot) = \overline{h}(\cdot) + h^{x, \eta}(\cdot) + h^{y, \eta}(\cdot)
\end{align}

\noindent where 
\begin{itemize}
\item $h^{x, \eta}$ and $h^{y, \eta}$ are Gaussian free fields on $B(x, \eta)$ and $B(y, \eta)$ respectively with Dirichlet boundary conditions,
\item $\overline{h}(\cdot)$ is the harmonic extension of $h$ to $B(x, \eta) \cup B(y, \eta)$,
\end{itemize}

\noindent and all these three objects are independent of each other. Let us further perform a radial-lateral decomposition of the Gaussian free field
\begin{align*}
h^{x, \eta}(\cdot) = h^{x, \mathrm{rad}}(\cdot)+ h^{x, \mathrm{lat}}(\cdot)
\end{align*}

\noindent where
\begin{align*}
\mathbb{E}\left[h^{x, \mathrm{rad}}(a) h^{x, \mathrm{rad}}(b)\right]
&= -\log \frac{|a-x| \vee |b-x|}{\eta}, \\
\mathbb{E}\left[h^{x, \mathrm{lat}}(a) h^{x, \mathrm{lat}}(b)\right]
& = G_0^{\mathbb{D}}\left(\frac{a-x}{\eta}, \frac{b-x}{\eta}\right) - \mathbb{E}\left[h^{x, \mathrm{rad}}(a) h^{x, \mathrm{rad}}(b)\right].
\end{align*}

We now clarify the choice of $\delta \in (0, 2^{-m})$, assuming that it is sufficiently small such that
\begin{align*}
\left|G_0^{\mathbb{D}}\left(\frac{a-x}{\eta}, \frac{b-x}{\eta}\right) +\log \bigg|\frac{a-b}{\eta}\bigg| \right| 
\le \delta \qquad \forall a, b \in B(x, \delta)
\end{align*}

\noindent as well as
\begin{align*}
\bigg| G_0^D(x, y) - G_0^D(z, y) \bigg| \le \delta \qquad
\text{and} \qquad \bigg| \log R(x; D) - \log R(z; D) \bigg| \le  \delta
\end{align*}

\noindent for all $z \in B(x, \delta)$ (this is possible by \Cref{lem:Green_estimate}). If we write
\begin{align*}
\mathcal{E}_x(\delta):= \sup_{z \in B(x, \delta)} | \overline{h}(z) - \overline{h}(x)|,
\qquad {e}_x(\delta):= \sup_{z \in B(x, \delta)} | \mathbb{E}[\overline{h}(z)^2 - \overline{h}(x)^2]|,
\end{align*}

\noindent then for any $\sqrt{u} \le \delta/ k$ we have 
\begin{align*}
F_{\gamma}^{\{x, y\}}(x+\sqrt{u}\mathbf{b})
\begin{dcases}
\le e^{\frac{5 \gamma^2}{2}\delta + \gamma \mathcal{E}_x(\delta) + \frac{\gamma^2}{2} e_x(\delta)}  \ R(x; D)^{\frac{3\gamma^2}{2}} e^{\gamma^2 G_0^D(x, y)} e^{\gamma \overline{h}(x) - \frac{\gamma^2}{2} \mathbb{E}[\overline{h}(x)^2] }\\
\qquad \qquad\times \int_{0}^1  e^{\gamma h^{x, \eta}(x + \sqrt{u} \mathbf{b}_s) - \frac{\gamma^2}{2}\mathbb{E}[h^{x, \eta}(x + \sqrt{u} \mathbf{b}_s)^2]} \frac{ds}{|\sqrt{u} \mathbf{b}_s|^{\gamma^2}},\\
\ge \left[e^{\frac{5 \gamma^2}{2}\delta + \gamma \mathcal{E}_x(\delta) + \frac{\gamma^2}{2} e_x(\delta)}\right]^{-1} R(x; D)^{\frac{3\gamma^2}{2}} e^{\gamma^2 G_0^D(x, y)} e^{\gamma \overline{h}(x) - \frac{\gamma^2}{2} \mathbb{E}[\overline{h}(x)^2] }\\
\qquad \qquad\times \int_{0}^1  e^{\gamma h^{x, \eta}(x + \sqrt{u} \mathbf{b}_s) - \frac{\gamma^2}{2}\mathbb{E}[h^{x, \eta}(x + \sqrt{u} \mathbf{b}_s)^2]} \frac{ds}{|\sqrt{u} \mathbf{b}_s|^{\gamma^2}}
\end{dcases}
\end{align*}

\noindent and thus 
\begin{align}
\label{eq:cross_term_ub1}
\mathcal{I}\left(\lambda u F_{\gamma}^{\{x, y\}}(x+\sqrt{u}\mathbf{b})\right)
& \le E_x(\delta)^{-2} 
\mathcal{I}\left(\widetilde{\lambda} E_x(\delta) 
u \overline{F}_{\gamma}^{\{x\}}(x+\sqrt{u}\mathbf{b}; h^{x, \eta}(\cdot) + \overline{h}(x)) \right)
\end{align}

\noindent where
\begin{align*}
\widetilde{\lambda}
& := \lambda R(x; D)^{\frac{3\gamma^2}{2}}
e^{\gamma^2 G_0^D(x, y)}, 
\qquad 
E_x(\delta)
:= \left[e^{\frac{5 \gamma^2}{2}\delta + \gamma \mathcal{E}_x(\delta) + \frac{\gamma^2}{2} e_x(\delta)}\right]^{-1},
\end{align*}

\noindent and $\overline{F}_{\gamma}^{\{x\}}(\cdot; \cdot)$ was defined in \eqref{eq:overlineF}. Substituting everything back into \eqref{eq:Ik_rescaled1}, we obtain
\begin{align}\notag
I_k \le \int_0^{(\delta / k)^2}\frac{du}{2\pi u} 
\mathbb{E}\otimes \bdec{0}{0}{1}\Bigg[
& 1_{\mathcal{G}^{\{x, y\}}_{[n_0, m)}(x) \cap \mathcal{G}^{\{x, y\}}_{[n_0, m)}(y)}
1_{\mathcal{H}_k}
E_x(\delta)^{-2} \\
\label{eq:Ik_ub}
& \qquad \times
\mathcal{I}\left(\widetilde{\lambda} E_x(\delta) 
u \overline{F}_{\gamma}^{\{x\}}(x+\sqrt{u}\mathbf{b}; h^{x, \eta}(\cdot) + \overline{h}(x)) \right)
\Bigg].
\end{align}

\noindent We shall perform a (conditional) Gaussian comparison, replacing the lateral field $h^{x, \mathrm{lat}}$ associated with $h^{x, \eta}$ 
by the field
\begin{align*}
\mathbb{E}\left[\widehat{X}(z_1) \widehat{X}(z_2)\right] = \log \frac{|z_1-x| \vee |z_2-x|}{|z_1-z_2|} \qquad \forall z_1, z_2 \in B(x, \delta).
\end{align*}

\noindent Note that this replacement is possible because $h^{x, \mathrm{lat}}$ is independent of $\mathcal{G}^{\{x, y\}}_{[n_0, m)}(x) \cap \mathcal{G}^{\{x, y\}}_{[n_0, m)}(y)$. To see why this is the case, let us go back to the decomposition \eqref{eq:GFF_markov} and consider
\begin{align*}
h_r(x) = \overline{h}_r(x) + h^{x, \eta}_r(x) + h_r^{y, \eta}(x)
\end{align*}

\noindent where the subscript $r$ refers to averaging over the circle $\partial B(x, r)$:
\begin{itemize}
\item Given the condition \eqref{eq:m_condition} on our choice of $m$ and $\eta$, we have $\partial B(x, 2^{-j}) \cap B(y, \eta) = \emptyset$ for all $j \in \mathbb{N}$ (see \Cref{fig:scales}). This means $h_{2^{-j}}^{y, \eta}(x) = 0$ for all $j \in [n_0, \infty) \cap \mathbb{N}$. On the other hand, $h_r^{x, \eta}(x) = h^{x, \mathrm{rad}}(x + r)$ is independent of $h^{x, \mathrm{lat}}$ by the definition of radial-lateral decomposition. Hence $h_{2^{-j}}(x) = \overline{h}_{2^{-j}}(x) + h^{x, \mathrm{rad}}(x + 2^{-j})$ for any $j \in \mathbb{N}$, i.e.  $h^{x, \mathrm{lat}}$ is independent of $\mathcal{G}^{\{x, y\}}_{[n_0, m)}(x)$.
\item Similarly, $\partial B(y, 2^{-j}) \cap B(x, \eta) = \emptyset$  for all $j \in \mathbb{N}$ means that $h^{x, \eta}$ (and in particular $h^{x, \mathrm{lat}}$) is independent of the circle average of $h$ centred at $y$ at all dyadic scales, and is therefore independent of $\mathcal{G}^{\{x, y\}}_{[n_0, m)}(y)$.
\end{itemize}

We also have (by \Cref{lem:Green_estimate})
\begin{align*}
&\left|\mathbb{E}\left[h^{x, \mathrm{lat}}(z_1) h^{x, \mathrm{lat}}(z_2)\right] - \mathbb{E}\left[\widehat{X}(z_1) \widehat{X}(z_2)\right] \right|\\
& \qquad = \left|G_0^{\mathbb{D}}\left(\frac{z_1-x}{\eta}, \frac{z_2-x}{\eta}\right) +\log \bigg|\frac{z_1-z_2}{\eta}\bigg| \right|
\le 20\sqrt{u}
\end{align*}

\noindent for all $z_1, z_2 \in B(x, k\sqrt{u})$ with $u \in [0, (\delta/k)^2]$ for $\delta$ sufficiently small. As a result, if we consider

\begin{align*}
&\overline{F}_{\gamma}^{\{x\}}(x+\sqrt{u}\mathbf{b}; h^{x, \mathrm{rad}} + \widehat{X} + \overline{h}(x))
:= e^{\gamma \overline{h}(x) - \frac{\gamma^2}{2} \mathbb{E}[\overline{h}(x)^2] }\\
& \qquad \times \int_{0}^1  e^{\gamma h^{x, \mathrm{rad}}(x + \sqrt{u} \mathbf{b}_s) - \frac{\gamma^2}{2}\mathbb{E}[h^{x, \mathrm{rad}}(x + \sqrt{u} \mathbf{b}_s)^2]}e^{\gamma \widehat{X}(x+\sqrt{u}\mathbf{b}_s) - \frac{\gamma^2}{2}\mathbb{E}[\widehat{X}(x+\sqrt{u}\mathbf{b}_s)^2]} \frac{ds}{| \sqrt{u}\mathbf{b}_s|^{\gamma^2}},
\end{align*}

\noindent then \Cref{lem:Gcompare} combined with the fact that
\begin{align*}
\left|x^2 \frac{\partial^2}{\partial x^2} \mathcal{I}(\lambda x)\right|
\le e^{-\lambda x} \left[2(\lambda x)^2 + |\lambda x|^3\right] \le 40 \qquad \forall \lambda,  x \ge 0,
\end{align*}

\noindent implies
\begin{align}
\notag
&\Bigg| \int_0^{(\delta/k)^2}\frac{du}{2\pi u} \mathbb{E}\otimes \bdec{0}{0}{1} \Bigg[
1_{\mathcal{G}^{\{x, y\}}_{[n_0, m)}(x) \cap \mathcal{G}^{\{x, y\}}_{[n_0, m)}(y)}
1_{\mathcal{H}_k}
 E_x(\delta)^{-2} \\
\notag
& \qquad \qquad \qquad \qquad \qquad \qquad \times
\mathcal{I}\left(\widetilde{\lambda} E_x(\delta) 
u \overline{F}_{\gamma}^{\{x\}}(x+\sqrt{u}\mathbf{b}; h^{x, \eta}(\cdot) + \overline{h}(x)) \right)
\Bigg]\\
\notag
& \qquad - \int_0^{(\delta/k)^2}\frac{du}{2\pi u} \mathbb{E}\otimes \bdec{0}{0}{1} \Bigg[
1_{\mathcal{G}^{\{x, y\}}_{[n_0, m)}(x) \cap \mathcal{G}^{\{x, y\}}_{[n_0, m)}(y)}
1_{\mathcal{H}_k}
 E_x(\delta)^{-2} \\
\notag
& \qquad \qquad \qquad \qquad \qquad \qquad \qquad \times
\mathcal{I}\left(\widetilde{\lambda} E_x(\delta) 
u \overline{F}_{\gamma}^{\{x\}}(x+\sqrt{u}\mathbf{b}; h^{x, \mathrm{rad}} + \widehat{X} + \overline{h}(x)) \right)
\Bigg]\Bigg|\\
\label{eq:cross_ptwise_GPerror}
&\le \mathbb{E}[E_x(\delta)^{-2}]\pdec{0}{0}{1}(\mathcal{H}_k)\int_0^{(\delta/k)^2} \frac{du}{2\pi u} \frac{20 \sqrt{u}}{2} \cdot 40 \le 400  \mathbb{E}[E_x(\delta)^{-2}] \frac{\delta}{k} e^{-\frac{1}{2}(k-1)^2}
\end{align}

\noindent which is summable in $k$ uniformly in $\lambda > 0$. This gives rise to a negligible contribution as we send $\delta \to 0$ towards the end of the proof.

\paragraph{Uniform control and identifying the limit.}
Let us examine the Gaussian fields appearing in the definition of $\overline{F}_{\gamma}^{\{x\}}(x+\sqrt{u}\mathbf{b}; h^{x, \mathrm{rad}} + \widehat{X} + \overline{h}(x))$. Observe that 
\begin{align*}
h^{x, \mathrm{rad}}(x + \delta e^{-t}) - h^{x, \mathrm{rad}}(x + \delta), \qquad t \ge 0
\end{align*}
\noindent is a Brownian motion independent of  $h^{x, \mathrm{rad}}(x + \delta)$. In particular, the field
\begin{align*}
\widetilde{h}(z):= \left[h^{x, \mathrm{rad}}(z) - h^{x, \mathrm{rad}}(x + \delta)\right] + \widehat{X}(z), \qquad z \in B(x, \delta)
\end{align*}

\noindent is independent of $\mathcal{G}_{[n_0, m]}^{\{x, y\}}(x)$ and $\mathcal{G}_{[n_0, m]}^{\{x, y\}}(y)$, and is furthermore exactly scale invariant with covariance
\begin{align*}
\mathbb{E}\left[\widetilde{h}(z_1) \widetilde{h}(z_2)\right]
&= -\log|z_1-z_2| +\log (\delta) \\
&= -\log\left|\frac{z_1-z_2}{k\sqrt{u}}\right| - \log\left(k\sqrt{u} / \delta\right) 
\qquad \forall z_1, z_2 \in B(x, k\sqrt{u}).
\end{align*}

\noindent We can then apply spatial rescaling and obtain
\begin{align*}
& \int_{0}^1 e^{\gamma h^{x, \mathrm{rad}}(x + \sqrt{u} \mathbf{b}_s) - \frac{\gamma^2}{2}\mathbb{E}[h^{x, \mathrm{rad}}(x + \sqrt{u} \mathbf{b}_s)^2]}e^{\gamma \widehat{X}(x+\sqrt{u}\mathbf{b}_s) - \frac{\gamma^2}{2}\mathbb{E}[\widehat{X}(x+\sqrt{u}\mathbf{b}_s)^2]} \frac{ds}{| \mathbf{b}_s|^{\gamma^2}}\\
&\quad =  e^{\gamma h^{x, \mathrm{rad}}(x+\delta) - \frac{\gamma^2}{2}\mathbb{E}[h^{x, \mathrm{rad}}(x+\delta)^2]}
\int_0^1 e^{\gamma \widetilde{h}(x + \sqrt{u}\mathbf{b}_s) - \frac{\gamma^2}{2} \mathbb{E}[\widetilde{h}(x + \sqrt{u}\mathbf{b}_s)^2]} \frac{ds}{|\mathbf{b}_s|^{\gamma^2}}\\
& \quad \overset{d}{=}e^{\gamma h^{x, \mathrm{rad}}(x+\delta) - \frac{\gamma^2}{2}\mathbb{E}[h^{x, \mathrm{rad}}(x+\delta)^2]} e^{\gamma B_{T} - \frac{\gamma^2}{2} T}
\overline{F}_\gamma^{\{0\}}(k^{-1}\mathbf{b}; X^{\mathbb{D}})
\end{align*}

\noindent where 
\begin{itemize}
\item $\overline{F}_\gamma^{\{0\}}(k^{-1}\mathbf{b}; X^{\mathbb{D}}) = \int_0^1 |\mathbf{b}_s|^{-\gamma^2} e^{\gamma X^{\mathbb{D}}(k^{-1}\mathbf{b}_s) - \frac{\gamma^2}{2}\mathbb{E}[X^{\mathbb{D}}(k^{-1}\mathbf{b}_s)^2]}ds$, with $X^{\mathbb{D}}$ being the Gaussian field on the unit disc $\mathbb{D}$ satisfying $\mathbb{E}[X^{\mathbb{D}}(z_1) X^{\mathbb{D}}(z_2)] = -\log|z_1 - z_2|$;
\item $T = T(u; k, \delta) = -\log (k\sqrt{u}/\delta)$ and $B_{T}$ is an independent $\mathcal{N}(0, T)$ random variable.
\end{itemize}

\noindent Using the fact that $\overline{h}(x) + h^{x, \mathrm{rad}}(x+\delta) = \overline{h}_{\delta}(x) + h_\delta^{x, \eta}(x) = h_{\delta}(x)$, we have
\begin{align}\label{eq:cross_ptwise_scale}
\begin{split}
&u \overline{F}_{\gamma}^{\{x\}}(x+\sqrt{u}\mathbf{b}; h^{x, \mathrm{rad}} + \widehat{X} + \overline{h}(x))\\
& \qquad  \overset{d}{=} (\delta / k)^{2 - \gamma^2} (k\sqrt{u}/\delta)^{2-\gamma^2}e^{\gamma \overline{h}(x) - \frac{\gamma^2}{2} \mathbb{E}[\overline{h}(x)^2]}\\
& \qquad \qquad \times e^{\gamma h^{x, \mathrm{rad}}(x+\delta) - \frac{\gamma^2}{2}\mathbb{E}[h^{x, \mathrm{rad}}(x+\delta)^2]} 
e^{\gamma B_{T} - \frac{\gamma^2}{2} T}
\overline{F}_\gamma^{\{0\}}(k^{-1}\mathbf{b}; X^{\mathbb{D}})\\
& \qquad = e^{\gamma (B_{T} - (Q-\gamma)T)}
(\delta / k)^{2 - \gamma^2}  e^{\gamma h_\delta(x) - \frac{\gamma^2}{2}\mathbb{E}[h_\delta(x)^2]} 
\overline{F}_\gamma^{\{0\}}(k^{-1}\mathbf{b}; X^{\mathbb{D}})\\
&\qquad  =:  e^{\gamma (B_{T} - (Q-\gamma)T)} \mathcal{R}_x.
\end{split}
\end{align}

 Substituting everything back to our main expression, and doing the change of variable $k\sqrt{u} / \delta = e^{-t}$, we obtain
\begin{align}
\notag
&  \int_0^{(\delta/k)^2}\frac{du}{2\pi u} \mathbb{E}\otimes \bdec{0}{0}{1} \Bigg[
1_{\mathcal{G}^{\{x, y\}}_{[n_0, m)}(x) \cap \mathcal{G}^{\{x, y\}}_{[n_0, m)}(y)}
1_{\mathcal{H}_k}
 E_x(\delta)^{-2} \\
\notag
&\qquad \qquad  \qquad \qquad \qquad \qquad \times
\mathcal{I}\left(\widetilde{\lambda} E_x(\delta) 
u \overline{F}_{\gamma}^{\{x\}}(x+\sqrt{u}\mathbf{b}; h^{x, \mathrm{rad}} + \widehat{X} + \overline{h}(x)) \right)
\Bigg]\\
\label{eq:cross_ptwise_final}
& = \frac{1}{\pi} \int_{0}^\infty
 \mathbb{E}\otimes \bdec{0}{0}{1}\Bigg[ 
1_{\mathcal{G}^{\{x, y\}}_{[n_0, m)}(x) \cap \mathcal{G}^{\{x, y\}}_{[n_0, m)}(y)}
1_{\mathcal{H}_k}
E_x(\delta)^{-2} \mathcal{I}\left(\widetilde{\lambda} E_x(\delta)\mathcal{R}_xe^{\gamma (B_{t} - (Q-\gamma)t)}\right)
\Bigg] dt 
\end{align}

\noindent where $(B_t)_{t \ge 0}$ is a Brownian motion independent of everything else. 
Using \Cref{lem:main}, we see that \eqref{eq:cross_ptwise_final} is uniformly bounded by
\begin{align*}
& c_\gamma \mathbb{E}\otimes \bdec{0}{0}{1}\left[
1_{\mathcal{G}^{\{x, y\}}_{[n_0, m)}(x) \cap \mathcal{G}^{\{x, y\}}_{[n_0, m)}(y)}
1_{\mathcal{H}_k}
E_x(\delta)^{-2} \right]\\
&\quad = \pi c_\gamma
\mathbb{E} \left[
1_{\mathcal{G}^{\{x, y\}}_{[n_0, m)}(x) \cap \mathcal{G}^{\{x, y\}}_{[n_0, m)}(y)}
E_x(\delta)^{-2} \right]
\pdec{0}{0}{1}\left( \mathcal{H}_k\right)
\le C e^{-\frac{1}{2}(k-1)^2}
\end{align*}

\noindent for some $C \in (0, \infty)$ independent of $k \in \mathbb{N}$, and this is summable in $k$. Moreover, the same lemma suggests that \eqref{eq:cross_ptwise_final} converges, as $\lambda \to \infty$, to
\begin{align*}
c_{\gamma}  \mathbb{E}\left[
1_{\mathcal{G}^{\{x, y\}}_{[n_0, m)}(x) \cap \mathcal{G}^{\{x, y\}}_{[n_0, m)}(y)}
E_x(\delta)^{-2}\right]\pdec{0}{0}{1}\left( \mathcal{H}_k\right).
\end{align*}

\noindent Combining these with \eqref{eq:Ik_ub} and \eqref{eq:cross_ptwise_GPerror}, we have
\begin{align*}
& \liminf_{\lambda \to \infty} \sum_{k \ge 1} I_k\\
& \le \sum_{k \ge 1} c_{\gamma} \mathbb{E}\left[
1_{\mathcal{G}^{\{x, y\}}_{[n_0, m)}(x) \cap \mathcal{G}^{\{x, y\}}_{[n_0, m)}(y)}
E_x(\delta)^{-2}\right]\pdec{0}{0}{1}\left( \mathcal{H}_k\right)
+ \sum_{k \ge 1} 400  \mathbb{E}[E_x(\delta)^{-2}] \frac{\delta}{k} e^{-\frac{1}{2}(k-1)^2}
\\
& = c_{\gamma}\mathbb{E}\left[
1_{\mathcal{G}^{\{x, y\}}_{[n_0, m)}(x) \cap \mathcal{G}^{\{x, y\}}_{[n_0, m)}(y)}
E_x(\delta)^{-2}\right] + \mathcal{O}(\delta).
\end{align*}

\noindent Now, recall that $E_x(\delta)^{-2}$ is non-negative, non-increasing in $\delta$, has finite moments and $E_x(\delta) \xrightarrow{\delta \to 0^+} 1$ almost surely. Since $\delta > 0$ is arbitrary in our analysis, it follows from monotone convergence that
\begin{align*}
&\liminf_{\lambda \to \infty} 
\mathbb{E} \left[  
1_{\mathcal{G}^{\{x, y\}}_{[n_0, m)}(x) \cap \mathcal{G}^{\{x, y\}}_{[n_0, m)}(y)}
\int_0^1  \frac{du}{2\pi u} \bdec{x}{x}{u}\left[ \mathcal{I}\left(\lambda F_{\gamma}^{\{x, y\}}(\mathbf{b})\right)\right]\right]\\
&\qquad \le \lim_{\delta \to 0^+} c_{\gamma} \mathbb{E}\left[ 
1_{\mathcal{G}^{\{x, y\}}_{[n_0, m)}(x) \cap \mathcal{G}^{\{x, y\}}_{[n_0, m)}(y)}
E_x(\delta)^{-2}\right]
= c_\gamma \mathbb{P} \left(\mathcal{G}^{\{x, y\}}_{[n_0, m)}(x) \cap \mathcal{G}^{\{x, y\}}_{[n_0, m)}(y)\right).
\end{align*}

A matching lower bound can be obtained in a similar fashion, by noting that
\begin{align*}
\mathcal{I}\left(\lambda F_{\gamma}^{\{x, y\}}(\mathbf{b})\right)
& \ge E_x(\delta)^{2} 
\mathcal{I}\left(\widetilde{\lambda} E_x(\delta)^{-1}
u \overline{F}_{\gamma}^{\{x\}}(x+\sqrt{u}\mathbf{b}; h^{x, \eta}(\cdot) + \overline{h}(x)) \right)
\end{align*}

\noindent (cf. \eqref{eq:cross_term_ub1}) so that
\begin{align*}
\limsup_{\lambda \to \infty} \sum_{k \ge 1} I_k
& \ge c_{\gamma} \mathbb{E}\left[
1_{\mathcal{G}^{\{x, y\}}_{[n_0, m)}(x) \cap \mathcal{G}^{\{x, y\}}_{[n_0, m)}(y)}
E_x(\delta)^{2}\right] + \mathcal{O}(\delta).
\end{align*}

\noindent and therefore
\begin{align*}
&\limsup_{\lambda \to \infty} 
\mathbb{E} \left[  
1_{\mathcal{G}^{\{x, y\}}_{[n_0, m)}(x) \cap \mathcal{G}^{\{x, y\}}_{[n_0, m)}(y)}
\int_0^1  \frac{du}{2\pi u} 
\bdec{x}{x}{u}\left[ \mathcal{I}\left(\lambda F_{\gamma}^{\{x, y\}}(\mathbf{b})\right)\right]
\right]\\
&\ge \lim_{\delta \to 0^+} c_{\gamma} \mathbb{E}\left[
1_{\mathcal{G}^{\{x, y\}}_{[n_0, m)}(x) \cap \mathcal{G}^{\{x, y\}}_{[n_0, m)}(y)}
E_x(\delta)^{2}\right]
=c_\gamma \mathbb{P} \left(\mathcal{G}^{\{x, y\}}_{[n_0, m)}(x) \cap \mathcal{G}^{\{x, y\}}_{[n_0, m)}(y)\right).
\end{align*}
This completes the proof of \Cref{lem:cross_pointwise_cutoff}.
\end{proof}

\begin{proof}[Proof of \Cref{lem:cross_pointwise}]
In order to obtain the desired result, we need to send the cutoff parameter $m \to \infty $ in \eqref{eq:L1_cross_pointwise_cutoff}. In particular, it suffices to show that
\begin{align}\label{eq:cross_removecutoff}
\lim_{m \to \infty}\limsup_{\lambda \to \infty} 
\mathbb{E} \left[ 
1_{\mathcal{G}^{\{x, y\}}_{[m, \infty)}(p)^c}
\int_0^1 \frac{du}{2\pi u} \bdec{x}{x}{u}\left[ \mathcal{I}\left(\lambda F_{\gamma}^{\{x, y\}}(\mathbf{b})\right)\right]
 \right]
=  0
\end{align}

\noindent for $p \in \{x, y\}$ and any fixed and distinct $x, y \in D$ satisfying $d(x, \partial D) \wedge d(y, \partial D) \ge \kappa$.

Our first step is to establish a bound on the indicator function  $1_{\mathcal{G}^{\{x, y\}}_{[m, \infty)}(p)}$ by adapting the argument in \eqref{eq:ind_exp_trick0}. Let us assume without loss of generality that $m \in \mathbb{N}$ is sufficiently large so that $2^{-m} < |x-y|$. Since $x$ and $y$ are bounded away from $\partial D$, there exists some constant $C_\kappa > 0$ such that
\begin{align} \label{eq:Ckappa}
\left|\mathbb{E}[h_{2^{-n}}(p) h_\delta(p')] - \log \left(2^{-n} \vee |p-p'|\right)\right| \le C_\kappa \qquad \forall \delta \in [0, 2^{-n}], \quad \forall n \ge m
\end{align}

\noindent for any $p, p' \in \{x, y\}$ by \Cref{lem:mollified_cov}.  Recalling
\begin{align*}
\mathcal{G}^{\{x, y\}}_{[m, \infty)}(p) 
:= \left\{ h_{2^{-n}}(p) + \gamma \mathbb{E}\left[h_{2^{-n}}(p) \left(h(x) + h(y)\right)\right] \le \alpha \log(2^n) \quad \forall n \in [m, \infty) \cap \mathbb{N} \right\},
\end{align*}

\noindent it holds for any $\beta > 0$ that
\begin{align}
\notag
1_{\mathcal{G}^{\{x, y\}}_{[m, \infty)}(p)^c}
&\le \sum_{n \ge m} \exp\left( \beta \left[h_{2^{-n}}(p) + \gamma \mathbb{E}\left[h_{2^{-n}}(p) \left(h(x) + h(y)\right)\right] - \alpha \log (2^n)\right]\right) \\
\label{eq:ind_exp_trick1}
& \le \frac{e^{(\frac{\beta^2}{2} + 2\beta \gamma )C_{\kappa}}}{|x-y|^\gamma}\sum_{n \ge m} 2^{-\frac{\beta}{2}\left[2(\alpha - \gamma ) - \beta\right]n} e^{\beta h_{2^{-n}}(p) - \frac{\beta^2}{2} \mathbb{E}[h_{2^{-n}}(p)^2]}
\end{align}

Using this bound, we obtain by Cameron-Martin theorem that
\begin{align}
\notag
&\mathbb{E} \left[ 
1_{\mathcal{G}^{\{x, y\}}_{[m, \infty)}(p)^c}
\int_0^1 \frac{du}{2\pi u} \bdec{x}{x}{u}\left[ \mathcal{I}\left(\lambda F_{\gamma}^{\{x, y\}}(\mathbf{b})\right)\right]
 \right]\\
\notag
& \le \frac{e^{(\frac{\beta^2}{2} + 2\beta \gamma )C_{\kappa}}}{|x-y|^\gamma} 
\sum_{n \ge m} 2^{-\frac{\beta}{2}\left[2(\alpha - \gamma ) - \beta\right]n} 
\mathbb{E} \left[ 
e^{\beta h_{2^{-n}}(p) - \frac{\beta^2}{2} \mathbb{E}[h_{2^{-n}}(p)^2]}
\int_0^1 \frac{du}{2\pi u} \bdec{x}{x}{u}\left[ \mathcal{I}\left(\lambda F_{\gamma}^{\{x, y\}}(\mathbf{b})\right)\right]
 \right]\\
\label{eq:cross_cutoff_limit0}
& = \frac{e^{(\frac{\beta^2}{2} + 2\beta \gamma )C_{\kappa}}}{|x-y|^\gamma} 
\sum_{n \ge m} 2^{-\frac{\beta}{2}\left[2(\alpha - \gamma ) - \beta\right]n} 
\mathbb{E} \left[ 
\int_0^1 \frac{du}{2\pi u} \bdec{x}{x}{u}\left[ \mathcal{I}\left(\lambda F_{\gamma, (p, n, \beta)}^{\{x, y\}}(\mathbf{b})\right)\right]
 \right]
\end{align}

\noindent where
\begin{align} \label{eq:Frare}
F_{\gamma, (p, n, \beta)}^{\{x, y\}}(\mathbf{b})
:= \int_0^u e^{\gamma^2 [G_0^D(x, \mathbf{b}_s) + G_0^D(y, \mathbf{b}_s)] +\gamma \beta \mathbb{E}[h_{2^{-n}}(p) h(\mathbf{b}_s)]} F_{\gamma}(ds; \mathbf{b}).
\end{align}

Let us now fix $\delta \in (0, \frac{1}{8} \min(|x-y|, \kappa))$, and split the sum in \eqref{eq:cross_cutoff_limit0} (without the prefactor) into
\begin{align}
\notag
& \sum_{n \ge m} 2^{-\frac{\beta}{2}\left[2(\alpha - \gamma ) - \beta\right]n} 
 \sum_{k \ge 1} 
\mathbb{E} \left[ 
\int_{(2^{-n}\delta/k)^2}^1 \frac{du}{2\pi u} \bdec{x}{x}{u}\left[ \mathcal{I}\left(\lambda F_{\gamma, (p, n, \beta)}^{\{x, y\}}(\mathbf{b})\right)1_{\mathcal{H}_k}\right]
 \right]\\
\label{eq:cross_cutoff_limit1}
& + \sum_{n \ge m} 2^{-\frac{\beta}{2}\left[2(\alpha - \gamma ) - \beta\right]n} 
 \sum_{k \ge 1} 
\mathbb{E} \left[ 
\int_0^{(2^{-n}\delta/k)^2} \frac{du}{2\pi u} \bdec{x}{x}{u}\left[ \mathcal{I}\left(\lambda F_{\gamma, (p, n, \beta)}^{\{x, y\}}(\mathbf{b})\right)1_{\mathcal{H}_k}\right]
 \right].
\end{align}

\noindent The first double sum is easily bounded by
\begin{align*}
\sum_{n \ge m} 2^{-\frac{\beta}{2}\left[2(\alpha - \gamma ) - \beta\right]n} 
 \sum_{k \ge 1}  \pdec{0}{0}{1}(\mathcal{H}_k) \log (2^{n} k / \delta) \lesssim m 2^{-\frac{\beta}{2}\left[2(\alpha - \gamma ) - \beta\right]m} 
\end{align*}

\noindent uniformly in $\lambda > 0$ and vanishes as $m \to \infty$ provided that $\beta \in (0, 2(\alpha - \gamma))$. To treat the remaining double sum in \eqref{eq:cross_cutoff_limit1}, recall for each $k, n \in \mathbb{N}$ and on the event $\mathcal{H}_k$ that the Brownian bridge $\mathbf{b}$ satisfies (under the probability measure $\bdec{x}{x}{u}$)
\begin{align*}
\mathbf{b}_s \in B(x, k\sqrt{u}) \subset B(x, 2^{-n} \delta)  
\end{align*}

\noindent and in particular $|\mathbf{b}_s - y| \ge \delta$ and $d(\mathbf{b}_s, \partial D)  \ge \frac{\kappa}{2}$ for all $s \le u \le (2^{-n} \delta / k)^2$.

\noindent and observe that $B(x, \delta) \cap B(y, \delta) = \emptyset$ by our choice of $\delta$. We may therefore assume (up to a re-definition) that the constant $C_\kappa$ in \eqref{eq:Ckappa} also satisfies
\begin{align*}
\left|\mathbb{E}[h_{2^{-n}}(x) h(\mathbf{b}_s)] - \log \left(2^{-n}\right)\right| &\le C_\kappa 
\qquad \text{and} \qquad
\left|\mathbb{E}[h_{2^{-n}}(y) h(\mathbf{b}_s)] \right| \le C_\kappa 
\end{align*}

\noindent for all $n \ge m$ and any $s$. This means, in particular, that
\begin{align}\label{eq:two-to-one}
e^{-\gamma(\gamma+\beta)C_\kappa}  \le 
\frac{F_{\gamma, (x,n, \beta)}^{\{x, y\}}(\mathbf{b})}{2^{\gamma \beta n} F_{\gamma}^{\{x\}}(\mathbf{b})}
\le e^{\gamma(\gamma+\beta)C_\kappa} 
\quad \text{and} \quad 
e^{-\gamma(\gamma+\beta)C_\kappa}  \le 
\frac{F_{\gamma, (y,n, \beta)}^{\{x, y\}}(\mathbf{b})}{F_{\gamma}^{\{x\}}(\mathbf{b})}
\le e^{\gamma(\gamma+\beta)C_\kappa} 
\end{align}

\noindent and hence
\begin{align} 
\notag
& \mathbb{E} \left[ 
\int_0^{(2^{-n}\delta/k)^2} \frac{du}{2\pi u} \bdec{x}{x}{u}\left[ \mathcal{I}\left(\lambda F_{\gamma, (p, n, \beta)}^{\{x, y\}}(\mathbf{b})\right)1_{\mathcal{H}_k}\right]
 \right]\\
\label{eq:two-to-one2}
& \le 
e^{2\gamma(\gamma+\beta)C_\kappa}  
\times \begin{dcases}
\mathbb{E} \left[ 
\int_0^{(2^{-n}\delta/k)^2} \frac{du}{2\pi u} \bdec{x}{x}{u}\left[  \mathcal{I}\left(\lambda e^{-\gamma(\gamma+\beta)C_\kappa} 2^{\gamma \beta n} F_{\gamma}^{\{x\}}(\mathbf{b})\right)1_{\mathcal{H}_k}\right]
 \right] & \text{for $p=x$},\\
\mathbb{E} \left[ 
\int_0^{(2^{-n}\delta/k)^2} \frac{du}{2\pi u} \bdec{x}{x}{u}\left[  \mathcal{I}\left(\lambda e^{-\gamma(\gamma+\beta)C_\kappa} F_{\gamma}^{\{x\}}(\mathbf{b})\right)1_{\mathcal{H}_k}\right]
 \right] & \text{for $p=y$}.
\end{dcases}
\end{align}

\noindent In either case this can be further upper bounded by $C \pdec{0}{0}{1}(\mathcal{H}_k)$ uniformly in $n, k \in \mathbb{N}$ and $\lambda > 0$ by \Cref{lem:G1_uniform} (as $d(x, \partial D) \ge \kappa \ge 4k \sqrt{u}$ is automatically satisfied). Substituting this back to the second sum in\eqref{eq:cross_cutoff_limit1}, we see that
\begin{align*}
& \sum_{n \ge m} 2^{-\frac{\beta}{2}\left[2(\alpha - \gamma ) - \beta\right]n} 
 \sum_{k \ge 1} 
\mathbb{E} \left[ 
\int_0^{(2^{-n}\delta/k)^2} \frac{du}{2\pi u} \bdec{x}{x}{u}\left[ \mathcal{I}\left(\lambda F_{\gamma, (p, n, \beta)}^{\{x, y\}}(\mathbf{b})\right)1_{\mathcal{H}_k}\right]
 \right]\\
& \qquad  \lesssim \sum_{n \ge m} 2^{-\frac{\beta}{2}\left[2(\alpha - \gamma ) - \beta\right]n} 
 \sum_{k \ge 1} \pdec{0}{0}{1}(\mathcal{H}_k)
\lesssim 2^{-\frac{\beta}{2}\left[2(\alpha - \gamma ) - \beta\right]m} \xrightarrow{m \to \infty} 0
\end{align*}

\noindent which concludes our proof of \eqref{eq:cross_removecutoff}.
\end{proof}

\subsection{Part III: analysis of the diagonal term \eqref{eq:diagonal_limit} } \label{sec:pf_diagonal}
We now consider the diagonal term
\begin{align}
\notag
&  \mathbb{E}\left[
\left( \int_A \mu_\gamma^{\kappa, n_0}(dx)  \int_0^1 \frac{du}{2\pi u} 
\bdec{x}{x}{u}[ \mathcal{I} \left(\lambda  F_{\gamma}(\mathbf{b}) \right)]
\right)^2\right]\\
\notag 
& = \int_{A \times A} 1_{\{d(x, \partial D) \ge \kappa \}}1_{\{d(y, \partial D) \ge \kappa \}}R(x; D)^{\frac{\gamma^2}{2}}R(y; D)^{\frac{\gamma^2}{2}}e^{\gamma^2 G_0^D(x, y)} dxdy\\
\label{eq:diagonal_preDCT}
&\times \mathbb{E} \left[ 1_{\mathcal{G}^{\{x, y\}}_{[n_0, \infty)}(x) \cap \mathcal{G}^{\{x, y\}}_{[n_0, \infty)}(y)}
\int_0^1 \frac{du}{2\pi u}\bdec{x}{x}{u}\left[ \mathcal{I} \left(\lambda  F_{\gamma}^{\{x, y\}}(\mathbf{b}) \right)\right]
\int_0^1 \frac{dv}{2\pi v} \bdec{y}{y}{v}\left[ \mathcal{I} \left(\lambda  F_{\gamma}^{\{x, y\}}(\tilde{\mathbf{b}}) \right)\right]
\right] 
\end{align}

\noindent where $\mathbf{b}$ and $\tilde{\mathbf{b}}$ are two independent Brownian bridges distributed according to $\bdec{x}{x}{u}$ and $\bdec{y}{y}{v}$ respectively.

\subsubsection{Uniform estimates for the diagonal term}
\begin{lem}\label{lem:diagonal_uniform}
Let $\beta > 0$ and $n_0 \in \mathbb{N}$ satisfying $2^{1-n_0} < \kappa$. Then there exists some constant $C = C(\kappa, n_0, \gamma, \alpha, \beta) \in (0, \infty)$ such that 
\begin{align}
\notag
& \mathbb{E} \left[ 
1_{\mathcal{G}^{\{x, y\}}_{[n_0, \infty)}(x) \cap \mathcal{G}^{\{x, y\}}_{[n_0, \infty)}(y)}
\int_0^1 \frac{du}{2\pi u}\bdec{x}{x}{u}\left[ \mathcal{I} \left(\lambda  F_{\gamma}^{\{x, y\}}(\mathbf{b}) \right)\right]
\int_0^1 \frac{dv}{2\pi v} \bdec{y}{y}{v}\left[ \mathcal{I} \left(\lambda  F_{\gamma}^{\{x, y\}}(\tilde{\mathbf{b}})\right) \right]
 \right] \\
\label{eq:diagonal_uniform}
& \qquad \qquad \le C \left[1 - \log |x-y|\right]^2 |x-y|^{(2\gamma - \alpha)\beta- \frac{\beta^2}{2}}
\end{align}

\noindent uniformly in $\lambda > 0$ and $x, y \in D$ satisfying $d(x, \partial D) \wedge d(y, \partial D) \ge \kappa$.
\end{lem}

As we saw earlier, the proof of \Cref{lem:cross_uniform} relies on \Cref{lem:G1_uniform}. The ``two-point" analogue of this estimate is as follows.

\begin{lem}\label{lem:G2_uniform}
Denote by $c(x,y) = c(x,y; \kappa):= \frac{1}{8}\min(|x-y|, \kappa)$. There exists some $C = C(\gamma, \kappa) \in (0, \infty)$ such that
\begin{align}\notag
\mathbb{E}\Bigg[ & 
\left(\int_{0}^{j^{-2} c(x, y)^2} \frac{du}{2\pi u} \bdec{x}{x}{u} [ \mathcal{I} \left(\lambda_1  F_\gamma^{\{x\}}(\mathbf{b})\right)1_{\mathcal{H}_j(\mathbf{b})}]\right) \\
\notag & \qquad \qquad \times 
\left(\int_{0}^{k^{-2} c(x,y)^2} \frac{dv}{2\pi v} \bdec{y}{y}{v} [ \mathcal{I} \left(\lambda_2  F_\gamma^{\{y\}}(\mathbf{b})\right)1_{\mathcal{H}_k(\mathbf{b})}]\right)
\Bigg]\\
\label{eq:G2_uniform}
& \qquad \le C \pdec{0}{0}{1}\left(\mathcal{H}_j\right)\pdec{0}{0}{1}\left(\mathcal{H}_k\right)
\end{align}

\noindent uniformly in $\lambda_1, \lambda_2 > 0$, $j, k \in \mathbb{N}$, and $x, y \in D$ satisfying $d(x, \partial D) \wedge d(y, \partial D) \ge \kappa$.
\end{lem}

\begin{proof}
By Fubini, we rewrite the LHS of \eqref{eq:G2_uniform} as
\begin{align}
\notag 
& \int_0^{j^{-2} c(x, y)^2} \int_0^{k^{-2} c(x,y)^2} \frac{du}{2\pi u}\frac{dv}{2\pi v}\\
\label{eq:G2pf0}
& \qquad\times
\bdec{x}{x}{u} \otimes \bdec{y}{y}{v} \Bigg[\mathbb{E}
\left[\mathcal{I} \left(\lambda_1  F_\gamma^{\{x\}}(\mathbf{b})\right)\mathcal{I} \left(\lambda_2  F_\gamma^{\{y\}}(\tilde{\mathbf{b}})\right)\right]
1_{\mathcal{H}_j(\mathbf{b})} 1_{\mathcal{H}_k(\tilde{\mathbf{b}})}\Bigg]
\end{align}

\noindent where $\mathbf{b}$ and $\tilde{\mathbf{b}}$ are two independent Brownian bridges distributed according to $\bdec{x}{x}{u}$ and $\bdec{y}{y}{v}$ respectively. The rest of our analysis will be divided into two steps, mirroring the structure of the proof of \Cref{lem:G1_uniform}.

\paragraph{Step (i): Gaussian comparison.} We want to derive a two-point analogue of the bound \eqref{eq:G1_uniform_step1}, i.e. for $\max(j\sqrt{u}, k\sqrt{v}) \le c(x,y)$ and on the event $\mathcal{H}_j(\mathbf{b}) \cap \mathcal{H}_k(\tilde{\mathbf{b}})$, we aim to establish an inequality of the form
\begin{align} \label{eq:G2_claim1}
& \mathbb{E}
\left[\mathcal{I} \left(\lambda_1  F_\gamma^{\{x\}}(\mathbf{b})\right)\mathcal{I} \left(\lambda_2  F_\gamma^{\{y\}}(\tilde{\mathbf{b}})\right)\right]
\le C \mathbb{E}
\left[\mathcal{I} \left(\widetilde{\lambda}_1
\overline{F}_\gamma^{\{x\}}(\mathbf{b}; X)\right)\mathcal{I} \left(\widetilde{\lambda}_2  \overline{F}_\gamma^{\{y\}}(\tilde{\mathbf{b}}; X)\right)\right]
\end{align}

\noindent where $X$ is some Gaussian field which shall be defined in \eqref{eq:G2_field}, and $\widetilde{\lambda}_1, \widetilde{\lambda}_2 > 0$ and $C \in (0, \infty)$ will be suitably chosen in \eqref{eq:G2_lambdat} and \eqref{eq:G2_const1} respectively. For now, we just emphasise that the constant $C$ on the RHS of \eqref{eq:G2_claim1} will be independent of $\lambda_1, \lambda_2$ and satisfy the desired uniformity in $j, k \in \mathbb{N}$ and $x, y \in D$ as described in the statement of  \Cref{lem:G2_uniform}.\\

To establish an inequality of the form \eqref{eq:G2_claim1}, we begin by applying Cameron-Martin to the LHS of \eqref{eq:G2_claim1} and rewrite
\begin{align}
\notag 
& \mathbb{E}
\left[\mathcal{I} \left(\lambda_1  F_\gamma^{\{x\}}(\mathbf{b})\right)\mathcal{I} \left(\lambda_2  F_\gamma^{\{y\}}(\tilde{\mathbf{b}})\right)\right]\\
\notag
& = \mathbb{E}
\left[\lambda_1 \lambda_2 F_\gamma^{\{x\}}(\mathbf{b})F_\gamma^{\{y\}}(\tilde{\mathbf{b}}) e^{-\lambda_1F_\gamma^{\{x\}}(\mathbf{b}) - \lambda_2F_\gamma^{\{y\}}(\tilde{\mathbf{b}})}\right]\\
\notag
& = \lambda_1 \lambda_2 \int_0^u ds_1 \int_0^v dt_1  e^{\gamma^2 [G_0^D(x, \mathbf{b}_{s_1}) + G_0^D(y, \tilde{\mathbf{b}}_{t_1}) + G_0^D(\mathbf{b}_{s_1}, \tilde{\mathbf{b}}_{t_1})]} R(\mathbf{b}_{s_1}; D)^{\frac{\gamma^2}{2}} 
 R(\tilde{\mathbf{b}}_{t_1}; D)^{\frac{\gamma^2}{2}}
\\
\notag 
& \qquad \times \mathbb{E}\Bigg[ 
\exp\Bigg(-\lambda_1 \int_{0}^u e^{\gamma^2 [G_0^D(x, \mathbf{b}_{s_2}) + G_0^D(\mathbf{b}_{s_1}, \mathbf{b}_{s_2})
+ G_0^D(\tilde{\mathbf{b}}_{t_1}, \mathbf{b}_{s_2})]}F_\gamma(ds_2; \mathbf{b})\\
\label{eq:G2pf1}
& \qquad \qquad \qquad \qquad 
- \lambda_2 \int_{0}^v e^{\gamma^2 [G_0^D(y, \tilde{\mathbf{b}}_{t_2}) + G_0^D(\tilde{\mathbf{b}}_{t_1}, \tilde{\mathbf{b}}_{t_2})
+ G_0^D(\mathbf{b}_{s_1}, \tilde{\mathbf{b}}_{t_2})]}F_\gamma(dt_2; \tilde{\mathbf{b}})
\Bigg)
\Bigg].
\end{align}

Since $\max(j\sqrt{u}, k \sqrt{v}) \le c(x,y)$, we have $\mathbf{b}_s \in B(x, j\sqrt{u}) \subset B(x, c(x,y))$ and $\tilde{\mathbf{b}}_t \in B(y, k\sqrt{v}) \subset B(y, c(x,y))$ on the event $\mathcal{H}_j(\mathbf{b}) \cap \mathcal{H}_k(\tilde{\mathbf{b}})$. Moreover:
\begin{itemize}[leftmargin=*]
    \item The following estimates apply for all $s_1, s_2 \le u$ and $t_1, t_2 \le v$ from \Cref{cor:gff_exact}:
\begin{align*}
 \left| G_0^D(\mathbf{b}_{s_1}, \mathbf{b}_{s_2}) - \left[ - \log |\mathbf{b}_{s_1} - \mathbf{b}_{s_2}| + \log R(x; D) \right] \right| &\le 4, \\
 \left| G_0^D(\tilde{\mathbf{b}}_{t_1}, \tilde{\mathbf{b}}_{t_2}) - \left[ - \log |\tilde{\mathbf{b}}_{t_1} - \tilde{\mathbf{b}}_{t_2}| + \log R(y; D) \right] \right| &\le 4.
\end{align*}

Let us recall again that these inequalities above imply, in particular, that
\begin{align*}
\begin{array}{r}
 \big| G_0^D(x, \mathbf{b}_{s_1}) - \big[ - \log |x - \mathbf{b}_{s_1}| + \log R(x; D) \big] \big| \le 4\\
 \big| G_0^D(y, \tilde{\mathbf{b}}_{t_1}) - \big[ - \log |y - \tilde{\mathbf{b}}_{t_1}| + \log R(y; D) \big] \big|\le 4
\end{array}
\qquad \text{(by setting $s_2, t_2 = 0$)}
\end{align*}

\noindent as well as
\begin{align*}
\begin{array}{r}
|\log R(\mathbf{b}_{s_1}; D) - \log R(x; D)|\le 4\\
|\log R(\tilde{\mathbf{b}}_{t_1}; D) - \log R(y; D)|\le 4
\end{array}
\qquad \text{(by letting $s_2 \to s_1, t_2 \to t_1$)}
\end{align*}

\noindent for all $s_1 \le u$ and $t_1 \le v$.

\item We have $d(a, \partial D) \wedge d(b, \partial D) \ge \frac{\kappa}{2}$ for all $a \in B(x, j\sqrt{u})$ and $b \in B(x, k\sqrt{v})$. This allows us to apply the estimate \eqref{eq:mollified_cov_precise} several times below; in particular,
\begin{align*}
    \left| G_0^D(\mathbf{b}_{s}, \tilde{\mathbf{b}}_{t}) + \log |\mathbf{b}_{s} - \tilde{\mathbf{b}}_{t}| \right| \le C_\kappa \qquad \forall s \le u, \quad t \le v.
\end{align*}

\item By definition, we also have
\begin{align*}
c(x, y) \le \frac{|x-y|}{8} \le |\mathbf{b}_s - \tilde{\mathbf{b}}_t| \le 2|x-y| \qquad \forall s \le u, \quad t \le v.
\end{align*}
\end{itemize}

\noindent Combining all these estimates, we can upper bound \eqref{eq:G2pf1} with
\begin{align}
\notag 
& \lambda_1 \lambda_2 \int_0^u ds_1 \int_0^v dt_1  \frac{e^{(12 + C_\kappa) \gamma^2} R(x; D)^{\frac{3\gamma^2}{2}} R(y; D)^{\frac{3\gamma^2}{2}}}{|x - \mathbf{b}_{s_1}|^{\gamma^2}|y - \tilde{\mathbf{b}}_{t_1}|^{\gamma^2} c(x, y)^{\gamma^2}}
\\
\notag 
& \quad \times \mathbb{E}\Bigg[ 
\exp\Bigg(-\lambda_1  \frac{2^{-\gamma^2} e^{-(10 + C_\kappa) \gamma^2}R(x; D)^{\frac{5\gamma^2}{2}}}{|x-y|^{\gamma^2}} 
\int_{0}^u \frac{e^{\gamma h(\mathbf{b}_{s_2}) - \frac{\gamma^2}{2} \mathbb{E}[h(\mathbf{b}_{s_2})^2]}ds_2}{|x-\mathbf{b}_{s_2}|^{\gamma^2} |\mathbf{b}_{s_1} - \mathbf{b}_{s_2}|^{\gamma^2} }
\\
\label{eq:G2pf2}
& \quad \qquad  \qquad \qquad 
-\lambda_2  \frac{2^{-\gamma^2} e^{-(10 + C_\kappa) \gamma^2}R(y; D)^{\frac{5\gamma^2}{2}}}{|x-y|^{\gamma^2}}
\int_{0}^v \frac{e^{\gamma h(\tilde{\mathbf{b}}_{t_2}) - \frac{\gamma^2}{2} \mathbb{E}[h(\tilde{\mathbf{b}}_{t_2})^2]}dt_2}{|y-\tilde{\mathbf{b}}_{t_2}|^{\gamma^2} |\tilde{\mathbf{b}}_{t_1} - \tilde{\mathbf{b}}_{t_2}|^{\gamma^2}}
\Bigg)
\Bigg].
\end{align}

Let us now introduce a new (centred) Gaussian field $X(\cdot) = X(\cdot; \kappa)$ on $B(x, c(x,y)) \cup B(y, c(x, y))$ with covariance
\begin{align}
\notag
& \mathbb{E}\left[X(a) X(b)\right] \\
\label{eq:G2_field}
& \quad = \mathbb{E}\left[X_x(a) X_x(b)\right]  1_{\{a, b \in B(x, c(x, y))\}}
+ \mathbb{E}\left[X_y(a) X_y(b)\right]1_{\{a, b \in B(y, c(x,y))\}}
+ \mathbb{E}\left[N_{x, y}^2\right]
\end{align}

\noindent where
\begin{itemize}
\item $X_x (\cdot)$ and $X_y(\cdot)$ are two independent exactly scale invariant Gaussian fields on the two balls $B(x, c(x,y))$ and $B(y, c(x,y))$ respectively, and both of their covariance kernels are of the form
\begin{align*}
(a, b) \mapsto -\log |a-b| + \log c(x,y);
\end{align*}
\item $N_{x, y}$ is an independent Gaussian random variable with zero mean and variance equal to $C_\kappa - \log c(x, y)$.
\end{itemize}

\noindent (The fact that $X_x$ and $X_y$ exist follows from the fact that the kernel $(a, b) \mapsto -\log |a-b|$ is positive definite on the unit ball in dimension $2$.) By construction, we see that
\begin{align*}
\mathbb{E}\left[X(a) X(b)\right]
& = \begin{dcases}
- \log |a-b|  + C_\kappa & \text{if $a, b$ belong to the same ball}\\
 - \log c(|x-y|) + C_\kappa& \text{otherwise}
\end{dcases}\\
& \ge - \log|a-b| + C_\kappa \\
&\ge G_0^D(a, b)
\qquad \forall a, b \in B(x, c(x, y) \cup B(y, c(x,y))
\end{align*}

\noindent where the last inequality follows from the definition of $C_\kappa$ in \eqref{eq:mollified_cov_precise} (sending $\epsilon, \delta$ to $0$). Therefore, by Gaussian comparison we further upper bound \eqref{eq:G2pf2} by
\begin{align}
\notag 
& \lambda_1 \lambda_2 \int_0^u ds_1 \int_0^v dt_1  \frac{e^{(12 + C_\kappa) \gamma^2} R(x; D)^{\frac{3\gamma^2}{2}} R(y; D)^{\frac{3\gamma^2}{2}}}{|x - \mathbf{b}_{s_1}|^{\gamma^2}|y - \tilde{\mathbf{b}}_{t_1}|^{\gamma^2} c(x,y)^{\gamma^2}}
\\
\notag 
& \quad \times \mathbb{E}\Bigg[ 
\exp\Bigg(-\lambda_1  \frac{2^{-\gamma^2} e^{-(10 + C_\kappa) \gamma^2}R(x; D)^{\frac{5\gamma^2}{2}}}{|x-y|^{\gamma^2}} 
\int_{0}^u \frac{e^{\gamma X(\mathbf{b}_{s_2}) - \frac{\gamma^2}{2} \mathbb{E}[X(\mathbf{b}_{s_2})^2]}ds_2}{|x-\mathbf{b}_{s_2}|^{\gamma^2} |\mathbf{b}_{s_1} - \mathbf{b}_{s_2}|^{\gamma^2} }
\\
\label{eq:G2pf3}
& \quad \qquad  \qquad \qquad 
-\lambda_2  \frac{2^{-\gamma^2} e^{-(10 + C_\kappa) \gamma^2}R(y; D)^{\frac{5\gamma^2}{2}}}{|x-y|^{\gamma^2}}
\int_{0}^v \frac{e^{\gamma X(\tilde{\mathbf{b}}_{t_2}) - \frac{\gamma^2}{2} \mathbb{E}[X(\tilde{\mathbf{b}}_{t_2})^2]}dt_2}{|y-\tilde{\mathbf{b}}_{t_2}|^{\gamma^2} |\tilde{\mathbf{b}}_{t_1} - \tilde{\mathbf{b}}_{t_2}|^{\gamma^2}}
\Bigg)
\Bigg].
\end{align}

Finally, recall the RHS of \eqref{eq:G2_claim1}: by Cameron-Martin we have
\begin{align}
\notag &C \mathbb{E}
\left[\mathcal{I} \left(\widetilde{\lambda}_1
\overline{F}_\gamma^{\{x\}}(\mathbf{b}; X)\right)\mathcal{I} \left(\widetilde{\lambda}_2  \overline{F}_\gamma^{\{y\}}(\tilde{\mathbf{b}}; X)\right)\right]\\
\notag 
& \qquad = C \widetilde{\lambda}_1\widetilde{\lambda}_2
\int_0^u ds_1 \int_0^v dt_1 \frac{e^{\gamma^2C_\kappa}}{|x-\mathbf{b}_{s_1}|^{\gamma^2}|y-\tilde{\mathbf{b}}_{t_1}|^{\gamma^2}c(x,y)^{\gamma^2}}\\
\notag
& \qquad \qquad \times \mathbb{E}\Bigg[
\exp \Bigg(-\widetilde{\lambda}_1 \frac{e^{2\gamma^2 C_\kappa}}{c(x,y)^{\gamma^2}}
\int_{0}^u \frac{e^{\gamma X(\mathbf{b}_{s_2}) - \frac{\gamma^2}{2} \mathbb{E}[X(\mathbf{b}_{s_2})^2]}ds_2}{|x-\mathbf{b}_{s_2}|^{\gamma^2} |\mathbf{b}_{s_1} - \mathbf{b}_{s_2}|^{\gamma^2} }\\
\label{eq:G2pf4}
& \qquad \qquad \qquad \qquad \qquad
-\widetilde{\lambda}_2 \frac{e^{2\gamma^2 C_\kappa}}{c(x,y)^{\gamma^2}}
\int_{0}^v \frac{e^{\gamma X(\tilde{\mathbf{b}}_{t_2}) - \frac{\gamma^2}{2} \mathbb{E}[X(\tilde{\mathbf{b}}_{t_2})^2]}dt_2}{|y-\tilde{\mathbf{b}}_{t_2}|^{\gamma^2} |\tilde{\mathbf{b}}_{t_1} - \tilde{\mathbf{b}}_{t_2}|^{\gamma^2} }
\Bigg)
\Bigg].
\end{align}

\noindent Comparing \eqref{eq:G2pf3} and \eqref{eq:G2pf4}, we can now choose
\begin{align}\label{eq:G2_lambdat}
\begin{split}
\widetilde{\lambda}_1 
&= \lambda_1 \left[ e^{-10 - 3C_\kappa} \frac{c(x, y)}{2|x-y|}\right]^{\gamma^2} R(x; D)^{\frac{5\gamma^2}{2}}, \\
\widetilde{\lambda}_2
&= \lambda_2 \left[ e^{-10 - 3C_\kappa} \frac{c(x, y)}{2|x-y|}\right]^{\gamma^2} R(y; D)^{\frac{5\gamma^2}{2}},
\end{split}
\end{align}

\noindent and 
\begin{align} \label{eq:G2_const1}
C = \left[ \frac{2|x-y|}{c(x, y)}\right]^{2\gamma^2} e^{(32 + 6C_\kappa)\gamma^2} R(x; D)^{-\gamma^2}R(y; D)^{-\gamma^2}
\end{align}

\noindent which can be bounded uniformly in $x, y$ satisfying $d(x, \partial D) \wedge d(y, \partial D) \ge \kappa$. This concludes Step (i) of the proof.

\paragraph{Step (ii): scale invariance.} For any $\widetilde{\lambda}_1, \widetilde{\lambda}_2 > 0$ and $j\sqrt{u}, k\sqrt{v} \le c(x,y)$, we aim to establish an identity of the form
\begin{align}
\notag
& \bdec{x}{x}{u} \otimes \bdec{y}{y}{v}  \otimes \mathbb{E}\Bigg[
\mathcal{I} \left(\widetilde{\lambda}_1
\overline{F}_\gamma^{\{x\}}(\mathbf{b}; X)\right)\mathcal{I} \left(\widetilde{\lambda}_2  \overline{F}_\gamma^{\{y\}}(\tilde{\mathbf{b}}; X)\right)1_{\mathcal{H}_j(\mathbf{b}) \cap \mathcal{H}_k(\tilde{\mathbf{b}})}\Bigg]\\
\label{eq:G2_claim2}
& = 
\bdec{0}{0}{1}^{\otimes 2}  \otimes \mathbb{E} \Bigg[
\mathcal{I} \left(\widetilde{\lambda}_1 \mathcal{E}_x(\mathbf{b}) e^{\gamma (B_{1, T_{1}(u)} - (Q-\gamma) T_1(u))}\right)
\mathcal{I} \left(\widetilde{\lambda}_2 \mathcal{E}_y(\tilde{\mathbf{b}})e^{\gamma (B_{2, T_{2}(v)} - (Q-\gamma) T_2(v))}\right)1_{\mathcal{H}_j(\mathbf{b})\cap\mathcal{H}_k(\tilde{\mathbf{b}})}\Bigg]
\end{align}

\noindent where $B_{1, T_1(u)} \sim \mathcal{N}(0, T_1(u))$ and $B_{2, T_2(v)} \sim \mathcal{N}(0, T_2(v))$ are two random variables independent of each other and everything else (including the random variables $\mathcal{E}_x(\mathbf{b})$ and $\mathcal{E}_y(\tilde{\mathbf{b}})$ which will be specified later), with
\begin{align} \label{eq:2ptB}
T_1(u) := -\log \left(\frac{j \sqrt{u}}{ c(x, y)}\right)
\quad \text{and} \quad 
 T_2(v) := -\log \left(\frac{k \sqrt{v}}{c(x,y)}\right)
\end{align}

\noindent which are non-negative for the range of values of $(u, v)$ under consideration.

To commence with, let us recall the definition of the field $X(\cdot)$ in \eqref{eq:G2_field}.  On the event $\mathcal{H}_j(\mathbf{b})\cap\mathcal{H}_k(\tilde{\mathbf{b}})$, we have
\begin{align*}
\overline{F}_\gamma^{\{x\}}(\mathbf{b}; X)
= \overline{F}_\gamma^{\{x\}}(\mathbf{b}; X_x(\cdot) + N_{x,y})
\qquad \text{and} \qquad
\overline{F}_\gamma^{\{y\}}(\tilde{\mathbf{b}}; X)
= \overline{F}_\gamma^{\{y\}}(\tilde{\mathbf{b}}; X_y(\cdot) + N_{x,y}).
\end{align*} 

\noindent Let us standardise our Brownian loops just like what was done in \eqref{eq:BB_rescale}; in other words, we rewrite
\begin{align}
\notag
&\bdec{x}{x}{u} \otimes \bdec{y}{y}{v}  \otimes \mathbb{E}\Bigg[
\mathcal{I} \left(\lambda_1
\overline{F}_\gamma^{\{x\}}(\mathbf{b}; X)\right)\mathcal{I} \left(\lambda_2  \overline{F}_\gamma^{\{y\}}(\tilde{\mathbf{b}}; X)\right)1_{\mathcal{H}_j(\mathbf{b}) \cap \mathcal{H}_k(\tilde{\mathbf{b}})}\Bigg]\\
\notag
& = \bdec{0}{0}{1}^{\otimes 2}  \otimes \mathbb{E}\Bigg[
\mathcal{I} \left(\lambda_1\overline{F}_\gamma^{\{x\}}(x + \sqrt{u}\mathbf{b}_{\cdot / u}; X_x(\cdot) + N_{x,y})\right)\\
\label{eq:G2pf5}
&\qquad \qquad \qquad \qquad \times \mathcal{I} \left(\lambda_2  \overline{F}_\gamma^{\{y\}}(y + \sqrt{v}\tilde{\mathbf{b}}_{\cdot / v}; X_y(\cdot) + N_{x,y})\right)1_{\mathcal{H}_j(\mathbf{b}) \cap \mathcal{H}_k(\tilde{\mathbf{b}})}\Bigg]
\end{align}

\noindent where (based on the same argument in \eqref{eq:Ftilde_std})
\begin{align*}
&\begin{pmatrix}
\overline{F}_\gamma^{\{x\}}(x + \sqrt{u}\mathbf{b}_{\cdot / u}; X_x(\cdot) + N_{x,y})\\
\overline{F}_\gamma^{\{y\}}(y + \sqrt{v}\tilde{\mathbf{b}}_{\cdot / v}; X_y(\cdot) + N_{x,y})
\end{pmatrix}\\
& \qquad = e^{\gamma N_{x, y} - \frac{\gamma^2}{2} \mathbb{E}[N_{x, y}^2]} \times
\begin{pmatrix}
\displaystyle
u^{1-\frac{\gamma^2}{2}} 
\int_0^1 1_{\{\mathbf{b}_s \in B(0, j)\}}\frac{e^{\gamma X_x(x + \sqrt{u} \mathbf{b}_s)- \frac{\gamma^2}{2} \mathbb{E}[X_x(x + \sqrt{u} \mathbf{b}_s)^2]}ds}{|\mathbf{b}_s|^{\gamma^2}}\\
\displaystyle 
 v^{1-\frac{\gamma^2}{2}}
\int_0^1 1_{\{\tilde{\mathbf{b}}_t \in B(0, k)\}}\frac{e^{\gamma X_y(y + \sqrt{v} \tilde{\mathbf{b}}_t)- \frac{\gamma^2}{2} \mathbb{E}[X_y(y + \sqrt{v}\tilde{ \mathbf{b}}_t)^2]}dt}{|\tilde{\mathbf{b}}_t|^{\gamma^2}}
\end{pmatrix}.
\end{align*}

\noindent Now, let $\overline{X}_x, \overline{X}_y$ be two independent Gaussian fields on the unit ball with covariance kernels $(a, b) \mapsto -\log |a-b|$ and recall \eqref{eq:2ptB}. Then for any $a, b \in B(0, 1)$, one has the following equivalence in covariance:
\begin{align*}
\mathbb{E}\left[X_x(x+ j\sqrt{u} a)X_x(x+ j\sqrt{u} b)\right]
& = \mathbb{E}\left[\overline{X}_x(a)\overline{X}_x(b)\right] + \mathbb{E}\left[B_{1,  T_1(u)}^2\right]\\
\text{and} \qquad \mathbb{E}\left[X_y(y+ k\sqrt{v} a)X_y(y+ k\sqrt{v} b)\right]
& = \mathbb{E}\left[\overline{X}_y(a)\overline{X}_y(b)\right] +\mathbb{E}\left[B_{2,T_2(v)}^2\right]
\end{align*}

\noindent and thus
\begin{align*}
& \begin{pmatrix}
\overline{F}_\gamma^{\{x\}}(x + \sqrt{u}\mathbf{b}_{\cdot / u}; X_x(\cdot) + N_{x,y})\\
\overline{F}_\gamma^{\{y\}}(y + \sqrt{v}\tilde{\mathbf{b}}_{\cdot / v}; X_y(\cdot) + N_{x,y})
\end{pmatrix}\\
& \qquad \overset{d}{=} 
e^{\gamma N_{x, y} - \frac{\gamma^2}{2} \mathbb{E}[N_{x, y}^2]}  \times
\begin{pmatrix}
\displaystyle
u^{1-\frac{\gamma^2}{2}} e^{\gamma B_{1, T_1(u)} - \frac{\gamma^2}{2} T_1(u)}
\int_0^1 1_{\{\mathbf{b}_s \in B(0, j)\}}\frac{e^{\gamma \overline{X}_x(\mathbf{b}_s)- \frac{\gamma^2}{2} \mathbb{E}[\overline{X}_x(\mathbf{b}_s)^2]}ds}{|\mathbf{b}_s|^{\gamma^2}}\\
\displaystyle 
 v^{1-\frac{\gamma^2}{2}} e^{\gamma B_{2, T_2(v)} - \frac{\gamma^2}{2} T_2(v)}
\int_0^1 1_{\{\tilde{\mathbf{b}}_t \in B(0, k)\}}\frac{e^{\gamma \overline{X}_y(\tilde{\mathbf{b}}_t)- \frac{\gamma^2}{2} \mathbb{E}[\overline{X}_y(\tilde{ \mathbf{b}}_t)^2]}dt}{|\tilde{\mathbf{b}}_t|^{\gamma^2}}
\end{pmatrix}\\
& \qquad =
e^{\gamma N_{x, y} - \frac{\gamma^2}{2} \mathbb{E}[N_{x, y}^2]}  \times
\begin{pmatrix}
\overline{F}_\gamma^{\{x\}}(\mathbf{b}; \overline{X}_x) \left[c(x,y)/j\right]^{2-\gamma^2} \exp\left(\gamma [B_{1, T_1(u)} - (Q-\gamma) T_1(u)] \right)\\
\overline{F}_\gamma^{\{y\}}(\tilde{\mathbf{b}}; \overline{X}_y) \left[c(x,y)/k\right]^{2-\gamma^2} \exp\left(\gamma [B_{2, T_2(v)} - (Q-\gamma) T_2(v)] \right)
\end{pmatrix}.
\end{align*}

\noindent Substituting this into \eqref{eq:G2pf5}, we conclude that \eqref{eq:G2_claim2} holds with
\begin{align*}
\begin{pmatrix}
\mathcal{E}_x(\mathbf{b}) \\
\mathcal{E}_y(\tilde{\mathbf{b}})
\end{pmatrix}
:=  e^{\gamma N_{x, y} - \frac{\gamma^2}{2} \mathbb{E}[N_{x, y}^2]} \times
\begin{pmatrix}
\overline{F}_\gamma^{\{x\}}(\mathbf{b}; \overline{X}_x) \left[c(x,y)/j\right]^{2-\gamma^2}\\
\overline{F}_\gamma^{\{y\}}(\tilde{\mathbf{b}}; \overline{X}_y) \left[c(x,y)/k\right]^{2-\gamma^2}
\end{pmatrix}.
\end{align*}

\paragraph{Concluding the proof of \Cref{lem:G2_uniform}.} Combining the two claims \eqref{eq:G2_claim1} and \eqref{eq:G2_claim2},  we see that \eqref{eq:G2pf0} is upper-bounded by
\begin{align}
\notag 
& \int_0^{j^{-2} c(x,y)^2} \int_0^{k^{-2} c(x,y)^2} \frac{du}{2\pi u}\frac{dv}{2\pi v}
\bdec{0}{0}{1}^{\otimes 2} \otimes \mathbb{E} \Bigg[
\mathcal{I} \left(\widetilde{\lambda}_1 \mathcal{E}_x(\mathbf{b}) e^{\gamma [B_{1, T_{1}(u)} - (Q-\gamma) T_1(u)]}\right)\\
\label{eq:G2pf6}
& \qquad\times
\mathcal{I} \left(\widetilde{\lambda}_2 \mathcal{E}_y(\tilde{\mathbf{b}})e^{\gamma [B_{2, T_{2}(v)} - (Q-\gamma) T_2(v)]}\right)1_{\mathcal{H}_j(\mathbf{b})\cap\mathcal{H}_k(\tilde{\mathbf{b}})}\Bigg]
\end{align}

\noindent up to a multiplicative constant $C \in (0, \infty)$ inherited from the RHS of \eqref{eq:G2_claim1}.

Note that by definition, the distributions of $\mathcal{E}_x(\mathbf{b})$ and $\mathcal{E}_y(\tilde{\mathbf{b}})$ do not depend on the value of $u$ and $v$. If we now consider the substitution $s = T_1(u)$ and $t = T_2(v)$, then \eqref{eq:G2pf6} can be further rewritten as
\begin{align*}
&\int_0^\infty \int_0^\infty \frac{ds dt}{\pi^2} 
\bdec{0}{0}{1}^{\otimes 2} \otimes \mathbb{E} \Bigg[
\mathcal{I} \left(\widetilde{\lambda}_1 \mathcal{E}_x(\mathbf{b}) e^{\gamma B_{1, s}^{-(Q-\gamma)}}\right)
\mathcal{I} \left(\widetilde{\lambda}_2 \mathcal{E}_y(\tilde{\mathbf{b}})e^{\gamma B_{2, t}^{-(Q-\gamma)}}\right)1_{\mathcal{H}_j(\mathbf{b})\cap\mathcal{H}_k(\tilde{\mathbf{b}})}\Bigg]\\
& \qquad = \frac{1}{\pi^2}
\bdec{0}{0}{1}^{\otimes 2} \otimes \mathbb{E} \Bigg[
\left( \int_0^\infty \mathcal{I} \left(\widetilde{\lambda}_1 \mathcal{E}_x(\mathbf{b}) e^{\gamma B_{1, s}^{-(Q-\gamma)}}\right) ds\right)\\
&\qquad \qquad \times
\left(\int_0^\infty \mathcal{I} \left(\widetilde{\lambda}_2 \mathcal{E}_y(\tilde{\mathbf{b}})e^{\gamma B_{2, t}^{-(Q-\gamma)}}\right)dt\right)
1_{\mathcal{H}_j(\mathbf{b})\cap\mathcal{H}_k(\tilde{\mathbf{b}})}\Bigg]
\end{align*}

\noindent where $(B_{i, t}^{-(Q-\gamma)})_{t \ge 0}$ are two independent Brownian motions with drift $-(Q-\gamma) < 0$ that are independent of everything else. By \Cref{lem:main} (or more precisely the estimate \eqref{eq:uniform_main2}), we see that this expectation is bounded uniformly in $\lambda_1, \lambda_2 > 0$ by
\begin{align*}
\left[\pi c_\gamma \right]^2 \bdec{0}{0}{1}^{\otimes 2} \otimes \mathbb{E} [1_{\mathcal{H}_j(\mathbf{b})\cap\mathcal{H}_k(\tilde{\mathbf{b}})}]
= \left[\pi c_\gamma \right]^2 \pdec{0}{0}{1}\left(\mathcal{H}_j\right)\pdec{0}{0}{1}\left(\mathcal{H}_k\right)
\end{align*}

\noindent which is our desired claim \eqref{eq:G2_uniform}.

\end{proof}

\begin{proof}[Proof of \Cref{lem:diagonal_uniform}]
Recall $c(x, y):= \frac{1}{8} \min(|x-y|, \kappa)$, and consider
\begin{align}\label{eq:diagonal_uniform_pf0}
\begin{split}
& \mathbb{E} \Bigg[ 
1_{\mathcal{G}^{\{x, y\}}_{[n_0, \infty)}(x) \cap \mathcal{G}^{\{x, y\}}_{[n_0, \infty)}(y)}
\int_0^1 \frac{du}{2\pi u}\bdec{x}{x}{u}\left[ \mathcal{I} \left(\lambda  F_{\gamma}^{\{x, y\}}(\mathbf{b}) \right)\right]
\int_0^1 \frac{dv}{2\pi v} \bdec{y}{y}{v}\left[ \mathcal{I} \left(\lambda  F_{\gamma}^{\{x, y\}}(\tilde{\mathbf{b}})\right) \right]
\Bigg] \\
&  \le   \sum_{k \ge 1} \mathbb{E}\Bigg[ 
1_{\mathcal{G}^{\{x, y\}}_{[n_0, \infty)}(x) \cap \mathcal{G}^{\{x, y\}}_{[n_0, \infty)}(y)}
 \left(\int_{k^{-2} c(x,y)^2}^{1} \frac{dv}{2\pi v} \pdec{y}{y}{v} \left(\mathcal{H}_k\right)\right)\\
&\qquad \qquad \qquad \qquad \qquad \times 
\left(\int_{0}^{1} \frac{du}{2\pi u} \bdec{x}{x}{u} [ \mathcal{I} \left(\lambda  F_\gamma^{\{x, y\}}(\mathbf{b})\right)1_{\mathcal{H}_j(\mathbf{b})}]\right)
\Bigg]\\
& \quad + \sum_{j\ge 1} \mathbb{E}\Bigg[ 
1_{\mathcal{G}^{\{x, y\}}_{[n_0, \infty)}(x) \cap \mathcal{G}^{\{x, y\}}_{[n_0, \infty)}(y)}
\left(\int_{j^{-2} c(x, y)^2}^{1} \frac{du}{2\pi u} \pdec{x}{x}{u} \left(\mathcal{H}_j\right)\right)\\
& \qquad \qquad  \qquad \qquad  \qquad \times
\left(\int_{0}^{1} \frac{dv}{2\pi v} \bdec{y}{y}{v} [ \mathcal{I} \left(\lambda  F_\gamma^{\{x, y\}}(\tilde{\mathbf{b}})\right)1_{\mathcal{H}_k(\tilde{\mathbf{b}})}]\right)
\Bigg]\\
& \quad + \sum_{j, k \ge 1} \mathbb{E}\Bigg[ 
1_{\mathcal{G}^{\{x, y\}}_{[n_0, \infty)}(x) \cap \mathcal{G}^{\{x, y\}}_{[n_0, \infty)}(y)}
\left(\int_{0}^{j^{-2} c(x, y)^2} \frac{du}{2\pi u} \bdec{x}{x}{u} [ \mathcal{I} \left(\lambda  F_\gamma^{\{x, y\}}(\mathbf{b})\right)1_{\mathcal{H}_j(\mathbf{b})}]\right) \\
 & \qquad \qquad \qquad \qquad \qquad   \times 
\left(\int_{0}^{k^{-2} c(x, y)^2} \frac{dv}{2\pi v} \bdec{y}{y}{v} [ \mathcal{I} \left(\lambda  F_\gamma^{\{x, y\}}(\tilde{\mathbf{b}})\right)1_{\mathcal{H}_k(\tilde{\mathbf{b}})}]\right)
\Bigg].
\end{split}
\end{align}

The first sum on the RHS is upper bounded by 
\begin{align*}
&\mathbb{E}\Bigg[ 
1_{\mathcal{G}^{\{x, y\}}_{[n_0, \infty)}(x) \cap \mathcal{G}^{\{x, y\}}_{[n_0, \infty)}(y)}
\left(\int_{0}^{1} \frac{du}{2\pi u} \bdec{x}{x}{u} [ \mathcal{I} \left(\lambda  F_\gamma^{\{x, y\}}(\mathbf{b})\right)1_{\mathcal{H}_j(\mathbf{b})}]\right)
\Bigg]
\sum_{k \ge 1}  \left(\log \frac{k}{c(x, y)}\right) \pdec{0}{0}{1}(\mathcal{H}_k)\\
& \qquad  \lesssim 
\mathbb{E}\Bigg[ 
1_{\mathcal{G}^{\{x, y\}}_{[n_0, \infty)}(x) \cap \mathcal{G}^{\{x, y\}}_{[n_0, \infty)}(y)}
\left(\int_{0}^{1} \frac{du}{2\pi u} \bdec{x}{x}{u} [ \mathcal{I} \left(\lambda  F_\gamma^{\{x, y\}}(\mathbf{b})\right)1_{\mathcal{H}_j(\mathbf{b})}]\right)
\Bigg] \left[1 - \log c(x, y)\right].
\end{align*}

\noindent Since the remaining expectation can be controlled by \Cref{lem:cross_uniform}, it follows that the first sum indeed satisfies a bound of the form \eqref{eq:diagonal_uniform}. The same argument applies to the second sum in \eqref{eq:diagonal_uniform_pf0}.

To conclude the proof we must show that the third sum in \eqref{eq:diagonal_uniform_pf0} satisfies a similar bound. We now consider two cases, following arguments similar to that of the proof of \Cref{lem:cross_uniform}.

\paragraph{Case 1: $|x-y| \ge 2^{-n_0}$.} 
Our goal here is to show that the third sum in \eqref{eq:diagonal_uniform_pf0} is bounded uniformly in $\lambda > 0$. (This is enough to conclude the estimate \eqref{eq:diagonal_uniform} as $|x-y|$ is bounded away from $0$.)

For any $\max(j\sqrt{u}, k\sqrt{v}) \le c(x, y)$, we have on the event $\mathcal{H}_j(\mathbf{b}) \cap \mathcal{H}_k(\tilde{\mathbf{b}})$ that
\begin{align*}
\mathbf{b}_{\cdot} \in B(x, j \sqrt{u}) \subset B(x, c(x, y))
\qquad \text{and} \qquad 
\tilde{\mathbf{b}}_{\cdot} \in B(y, k \sqrt{u}) \subset B(y, c(x,y)).
\end{align*}

\noindent Based on the definition of $c(x, y)$, we know that the two balls $B(x, c(x, y))$ and $B(y, c(x,y))$ are at least $|x-y| / 2 \ge 2^{-(n_0 + 1)}$ apart from each other. By the continuity of the Green's function away from the diagonal, there exists some constant $C_D(n_0) < \infty$ such that
\begin{align*}
\max \left( |G_0^D(y, \mathbf{b}_s)|, |G_0^D(x, \tilde{\mathbf{b}}_t)|\right) \le C_D(n_0)
\qquad \forall s \le u, \quad t \le v
\end{align*}

\noindent and hence
\begin{align*}
\mathcal{I}\left(\lambda F_{\gamma}^{\{x, y\}}(\mathbf{b})\right) &\le e^{2\gamma^2 C_D(n_0)}\mathcal{I}\left( \widetilde{\lambda} F_{\gamma}^{\{x\}}(\mathbf{b})\right)\\
 \text{and} \qquad 
\mathcal{I}\left(\lambda F_{\gamma}^{\{x, y\}}(\tilde{\mathbf{b}})\right) &\le e^{2\gamma^2 C_D(n_0)}\mathcal{I}\left( \widetilde{\lambda} F_{\gamma}^{\{y\}}(\tilde{\mathbf{b}})\right)
\end{align*}

\noindent for $\widetilde{\lambda} := \lambda e^{-\gamma^2 C_D(n_0)}$. Putting everything back together, we have
\begin{align*}
& \sum_{j, k \ge 1} \mathbb{E}\Bigg[ 
1_{\mathcal{G}^{\{x, y\}}_{[n_0, \infty)}(x) \cap \mathcal{G}^{\{x, y\}}_{[n_0, \infty)}(y)}
\left(\int_{0}^{j^{-2} c(x, y)^2} \frac{du}{2\pi u} \bdec{x}{x}{u} [ \mathcal{I} \left(\lambda  F_\gamma^{\{x, y\}}(\mathbf{b})\right)1_{\mathcal{H}_j(\mathbf{b})}]\right) \\
 & \qquad \qquad \qquad \qquad \qquad   \times 
\left(\int_{0}^{k^{-2} c(x,y)^2} \frac{dv}{2\pi v} \bdec{y}{y}{v} [ \mathcal{I} \left(\lambda  F_\gamma^{\{x, y\}}(\tilde{\mathbf{b}})\right)1_{\mathcal{H}_k(\tilde{\mathbf{b}})}]\right)
\Bigg]\\
& \le  e^{4\gamma^2 C_D(n_0)} \sum_{j, k \ge 1} \mathbb{E}\Bigg[ 
\left(\int_{0}^{j^{-2} c(x,y)^2} \frac{du}{2\pi u} \bdec{x}{x}{u} [ \mathcal{I} \left(\widetilde{\lambda}  F_\gamma^{\{x\}}(\mathbf{b})\right)1_{\mathcal{H}_j(\mathbf{b})}]\right) \\
\notag & \qquad \qquad\qquad \qquad \qquad  \times 
\left(\int_{0}^{k^{-2} c(x,y)^2} \frac{dv}{2\pi v} \bdec{y}{y}{v} [ \mathcal{I} \left(\widetilde{\lambda}  F_\gamma^{\{y\}}(\tilde{\mathbf{b}})\right)1_{\mathcal{H}_k(\tilde{\mathbf{b}})}]\right)
\Bigg]
\end{align*}

\noindent which is bounded uniformly in $\widetilde{\lambda} > 0$ (and hence $\lambda > 0$) by \Cref{lem:G2_uniform}.

\paragraph{Case 2: $|x-y| < 2^{-n_0}$.}
Recall \eqref{eq:ind_exp_trick} where $n \ge n_0$ is chosen to be the integer satisfying $2^{-(n+1)} \le |x-y| < 2^{-n}$. We have
\begin{align}
\notag
& \mathbb{E}\Bigg[ 
1_{\mathcal{G}^{\{x, y\}}_{[n_0, \infty)}(x) \cap \mathcal{G}^{\{x, y\}}_{[n_0, \infty)}(y)}
\left(\int_{0}^{j^{-2} c(x, y)^2} \frac{du}{2\pi u} \bdec{x}{x}{u} [ \mathcal{I} \left(\lambda  F_\gamma^{\{x, y\}}(\mathbf{b})\right)1_{\mathcal{H}_j(\mathbf{b})}]\right) \\
\notag
 & \qquad \qquad \qquad \qquad \qquad   \times 
\left(\int_{0}^{k^{-2} c(x, y)^2} \frac{dv}{2\pi v} \bdec{y}{y}{v} [ \mathcal{I} \left(\lambda  F_\gamma^{\{x, y\}}(\tilde{\mathbf{b}})\right)1_{\mathcal{H}_k(\tilde{\mathbf{b}})}]\right)
\Bigg]\\
\notag
& \lesssim |x-y|^{(2\gamma - \alpha)\beta- \frac{\beta^2}{2}} 
 \mathbb{E}\Bigg[ e^{-\beta h_{2^{-n}}(x) - \frac{\beta^2}{2} \mathbb{E}[h_{2^{-n}}(x)^2]}
\left(\int_{0}^{j^{-2} c(x, y)^2} \frac{du}{2\pi u} \bdec{x}{x}{u} [ \mathcal{I} \left(\lambda  F_\gamma^{\{x, y\}}(\mathbf{b})\right)1_{\mathcal{H}_j(\mathbf{b})}]\right) \\
\notag
 & \qquad \qquad \qquad \qquad \qquad   \times 
\left(\int_{0}^{k^{-2} c(x, y)^2} \frac{dv}{2\pi v} \bdec{y}{y}{v} [ \mathcal{I} \left(\lambda  F_\gamma^{\{x, y\}}(\tilde{\mathbf{b}})\right)1_{\mathcal{H}_k(\tilde{\mathbf{b}})}]\right)
\Bigg]\\
\notag
& = |x-y|^{(2\gamma - \alpha)\beta- \frac{\beta^2}{2}} 
 \mathbb{E}\Bigg[ 
\left(\int_{0}^{j^{-2} c(x, y)^2} \frac{du}{2\pi u} \bdec{x}{x}{u} [ \mathcal{I} \left(\lambda  F_{\gamma, (n, -\beta)}^{\{x, y\}}(\mathbf{b})\right)1_{\mathcal{H}_j(\mathbf{b})}]\right) \\
\label{eq:diagonal_uniform_sum3}
 & \qquad \qquad \qquad \qquad \qquad   \times 
\left(\int_{0}^{k^{-2} c(x,y)^2} \frac{dv}{2\pi v} \bdec{y}{y}{v} [ \mathcal{I} \left(\lambda  F_{\gamma, (n, -\beta)}^{\{x, y\}}(\tilde{\mathbf{b}})\right)1_{\mathcal{H}_k(\tilde{\mathbf{b}})}]\right)
\Bigg]
\end{align}

\noindent where the notation $F_{\gamma, (n, -\beta)}^{\{x, y\}}(\cdot)$ was defined in \eqref{eq:beta_mu_G1}. \\

By definition, on the event $\mathcal{H}_j(\mathbf{b}) \cap \mathcal{H}_k(\tilde{\mathbf{b}})$ we have
\begin{align*}
\max\left(|\mathbf{b}_s - x|, |\tilde{\mathbf{b}}_t - y|\right) 
\le c(x, y) 
\le \frac{1}{8}|x-y| 
< 2^{-n_0-2} < \frac{\kappa}{4},
\end{align*}

\noindent and in particular $d(\mathbf{b}_s, \partial D) \wedge d(\tilde{\mathbf{b}}_t, \partial D) \ge \frac{\kappa}{2}$ for any $s \le \ell(\mathbf{b}), t \le \ell(\tilde{\mathbf{b}})$. By \eqref{eq:mollified_cov_precise} we have
\begin{gather*}
\left|G_0^D(y, \mathbf{b}_s) + \log|y-\mathbf{b}_s| \right| \le C_\kappa, 
\qquad \left|G_0^D(x, \tilde{\mathbf{b}}_t) + \log|x-\tilde{\mathbf{b}}_t| \right|\le C_\kappa, \\
\left|\mathbb{E}\left[h(\mathbf{b}_s) h_{2^{-n}}(x) \right] + \log (2^{-n})\right|  \le C_\kappa,
\qquad \left|\mathbb{E}\left[h(\tilde{\mathbf{b}}_t) h_{2^{-n}}(x) \right] + \log (2^{-n})\right|  \le C_\kappa.
\end{gather*}

\noindent Combining these estimates with the fact that
\begin{align*}
\max \left\{ 
\left| \log |y-\mathbf{b}_s| - \log |x-y| \right|,
\left| \log |x - \tilde{\mathbf{b}}_t| - \log |x-y| \right|, 
\left| \log(2^{-n}) - \log |x-y| \right| \right\} \le C
\end{align*}

\noindent for some absolute constant $C>0$ (say $C = \log 2$), we obtain both \eqref{eq:unbeta_mu1} and
\begin{align}\label{eq:unbeta_mu2}
\widehat{C}^{-1}
F_{\gamma}^{\{y\}}(\widetilde{\mathbf{b}})
\le |x-y|^{-\gamma(\beta-\gamma)} F_{\gamma, (n, -\beta)}^{\{x, y\}}(\widetilde{\mathbf{b}})
\le \widehat{C}  F_{\gamma}^{\{y\}}(\widetilde{\mathbf{b}})
\end{align}

\noindent where $\widehat{C} = \widehat{C}(\kappa, \beta, \gamma) \in (0, \infty)$. This means \eqref{eq:diagonal_uniform_sum3} can be upper-bounded by
\begin{align*}
&\widehat{C}^4 |x-y|^{(2\gamma - \alpha)\beta- \frac{\beta^2}{2}} 
 \mathbb{E}\Bigg[ 
\left(\int_{0}^{j^{-2} c(x, y)^2} \frac{du}{2\pi u} \bdec{x}{x}{u} [ \mathcal{I} \left(\widehat{\lambda}  F_\gamma^{\{x\}}(\mathbf{b})\right)1_{\mathcal{H}_j(\mathbf{b})}]\right) \\
 & \qquad \qquad \qquad \qquad \qquad   \times 
\left(\int_{0}^{k^{-2} c(x, y)^2} \frac{dv}{2\pi v} \bdec{y}{y}{v} [ \mathcal{I} \left(\widehat{\lambda}  F_\gamma^{\{y\}}(\tilde{\mathbf{b}})\right)1_{\mathcal{H}_k(\tilde{\mathbf{b}})}]\right)
\Bigg]
\end{align*}

\noindent with $\widehat{\lambda} := \lambda \widehat{C}^{-1} |x-y|^{\gamma(\beta - \gamma)}$. This expression can now be controlled uniformly in $\lambda > 0$ and $j, k \in \mathbb{N}$ by \Cref{lem:G2_uniform} and we are done after taking the sum over $j, k \ge 1$. This concludes the proof of \Cref{lem:diagonal_uniform}.
\end{proof}

\subsubsection{Pointwise limit of the diagonal term}
We now state the pointwise limit for our diagonal term.
\begin{lem}\label{lem:diagonal_pointwise}
For any fixed $n_0 \in \mathbb{N}$ satisfying $2^{1-n_0} < \kappa$, 
\begin{align}
\notag
&\mathbb{E} \left[ 
1_{\mathcal{G}^{\{x, y\}}_{[n_0, \infty)}(x) \cap \mathcal{G}^{\{x, y\}}_{[n_0, \infty)}(y)}
\int_0^1 \frac{du}{2\pi u}\bdec{x}{x}{u}\left[ \mathcal{I} \left(\lambda  F_{\gamma}^{\{x, y\}}(\mathbf{b}) \right)\right]
\int_0^1 \frac{dv}{2\pi v} \bdec{y}{y}{v}\left[ \mathcal{I} \left(\lambda  F_{\gamma}^{\{x, y\}}(\tilde{\mathbf{b}})\right) \right]
 \right] \\
\label{eq:L1_diagonal_pointwise}
& \qquad =  c_\gamma^2
\mathbb{P} \left(\widetilde{\mathcal{G}}_{[n, \infty)}(x) \cap \widetilde{\mathcal{G}}_{[n, \infty)}(y)\right)
\end{align}

\noindent for any distinct points $x, y \in D$ satisfying $d(x, \partial D) \wedge d(y, \partial D) \ge \kappa$ and $-\log_2|x-y| \not \in \mathbb{N}$.
\end{lem}

\begin{proof}
The analysis of diagonal term is very similar to that of the cross term performed in \Cref{sec:cross_pointwise}, so we only sketch the arguments here.

\paragraph{Step (i).} We need a ``two-point" analogue of \Cref{lem:cross_pointwise_cutoff}, i.e. we first show that for any $m > 3 + \max(n, -\log_2 |x-y|)$ sufficiently large, 
\begin{align}
\notag
&\lim_{\lambda \to \infty} 
\mathbb{E} \left[ 
1_{\mathcal{G}^{\{x, y\}}_{[n_0, m)}(x) \cap \mathcal{G}^{\{x, y\}}_{[n_0, m)}(y)}
\int_0^1 \frac{du}{2\pi u}\bdec{x}{x}{u}\left[ \mathcal{I} \left(\lambda  F_{\gamma}^{\{x, y\}}(\mathbf{b}) \right)\right]
\int_0^1 \frac{dv}{2\pi v} \bdec{y}{y}{v}\left[ \mathcal{I} \left(\lambda  F_{\gamma}^{\{x, y\}}(\tilde{\mathbf{b}})\right) \right]
 \right] \\
\label{eq:L1_diagonal_pointwise_cutoff}
&\qquad =  c_\gamma^2 \mathbb{P} \left(\mathcal{G}^{\{x, y\}}_{[n_0, m)}(x) \cap \mathcal{G}^{\{x, y\}}_{[n_0, m)}(y)\right).
\end{align}

Let us fix some $\delta \in (0, 2^{-m})$ as before, and define for each $j, k \in \mathbb{N}$
\begin{align*}
I_{j,k} := \mathbb{E} \Bigg[
1_{\mathcal{G}^{\{x, y\}}_{[n_0, m)}(x) \cap \mathcal{G}^{\{x, y\}}_{[n_0, m)}(y)}
&\int_0^{(\delta/j)^2} \frac{du}{2\pi u}\bdec{x}{x}{u}\left[ \mathcal{I} \left(\lambda  F_{\gamma}^{\{x, y\}}(\mathbf{b}) \right)1_{\mathcal{H}_j(\mathbf{b})}\right]\\
& \qquad \times \int_0^{(\delta/k)^2} \frac{dv}{2\pi v} \bdec{y}{y}{v}\left[ \mathcal{I} \left(\lambda  F_{\gamma}^{\{x, y\}}(\tilde{\mathbf{b}})\right)1_{\mathcal{H}_k(\tilde{\mathbf{b}})} \right]
\Bigg].
\end{align*}

\noindent In order to establish \eqref{eq:L1_diagonal_pointwise_cutoff}, it suffices to show
\begin{align*}
\lim_{\lambda \to \infty} \sum_{j, k \ge 1} I_{j,k} =  c_\gamma^2 \mathbb{P} \left(\mathcal{G}_{[n_0, m)}^{\{x, y\}}(x) \cap  \mathcal{G}_{[n_0, m)}^{\{x, y\}}(y)\right)
\end{align*}

\noindent using a similar dominated convergence approach. As in the proof of \Cref{lem:cross_pointwise_cutoff} we just highlight the steps for the upper bound of $I_{j, k}$.
\begin{itemize}
\item By considering the domain Markov property of Gaussian free field  \eqref{eq:GFF_markov} and performing a radial-lateral decomposition of the two independent Gaussian fields 
\begin{align*}
h^{x, \eta}(\cdot) = h^{x, \mathrm{rad}}(\cdot)+ h^{x, \mathrm{lat}}(\cdot) 
\quad \text{and} \quad 
h^{y, \eta}(\cdot) = h^{y, \mathrm{rad}}(\cdot)+ h^{y, \mathrm{lat}}(\cdot),
\end{align*}
one obtains the following analogue of \eqref{eq:Ik_ub}: we have
\begin{align*}
I_{j, k}
\le  & \int_0^{(\delta/j)^2} \frac{du}{2\pi u} \int_0^{(\delta/k)^2} \frac{dv}{2\pi v}\\
&  \qquad\times \mathbb{E} \otimes \bdec{0}{0}{1}^{\otimes 2} \Bigg[
1_{\mathcal{G}^{\{x, y\}}_{[n_0, m)}(x) \cap \mathcal{G}^{\{x, y\}}_{[n_0, m)}(y)} 1_{\mathcal{H}_j(\mathbf{b}) \cap \mathcal{H}_k(\tilde{\mathbf{b}})}E_x(\delta)^{-2}E_y(\delta)^{-2}\\
&\qquad\qquad \qquad \qquad \times
\mathcal{I}\left(\widetilde{\lambda}_x E_x(\delta) 
u \overline{F}_{\gamma}^{\{x\}}(x+\sqrt{u}\mathbf{b}; h^{x, \eta}(\cdot) + \overline{h}(x)) \right)\\
&\qquad \qquad \qquad \qquad \times
\mathcal{I}\left(\widetilde{\lambda}_y E_y(\delta) 
v \overline{F}_{\gamma}^{\{y\}}(y+\sqrt{v}\tilde{\mathbf{b}}; h^{y, \eta}(\cdot) + \overline{h}(y)) \right)
\Bigg]
\end{align*}

where, for $p \in \{x, y\}$, 
\begin{align*}
\widetilde{\lambda}_p
& := \lambda R(p; D)^{\frac{3\gamma^2}{2}}
e^{\gamma^2 G_0^D(x, y)}, 
\qquad 
E_p(\delta)
:= \left[e^{\frac{5 \gamma^2}{2}\delta + \gamma \mathcal{E}_p(\delta) + \frac{\gamma^2}{2} e_p(\delta)}\right]^{-1},
\end{align*}

with
\begin{align*}
\mathcal{E}_p(\delta):= \sup_{z \in B(p, \delta)} | \overline{h}(z) - \overline{h}(p)|,
\qquad {e}_p(\delta):= \sup_{z \in B(p, \delta)} | \mathbb{E}[\overline{h}(z)^2 - \overline{h}(p)^2]|.
\end{align*}

\item We need two (conditional) Gaussian comparisons to replace $h^{p, \mathrm{lat}}$ with the field
\begin{align*}
\mathbb{E}\left[\widehat{X}^p(z_1) \widehat{X}^p(z_2)\right] = \log \frac{|z_1-p| \vee |z_2-p|}{|z_1-z_2|} \qquad \forall z_1, z_2 \in B(p, \delta)
\end{align*}
for each $p \in \{x, y\}$. One can show (with a computation similar to that in \eqref{eq:cross_ptwise_GPerror}) that these replacements would yield an error that is summable in $j, k \in \mathbb{N}$ uniformly in $\lambda > 0$, and negligible as $\delta \to 0^+$. In other words, we just need to study
\begin{align}\label{eq:diagonal_ptwise_GP}
\begin{split}
\int_0^{(\delta/j)^2} &\frac{du}{2\pi u} \int_0^{(\delta/k)^2} \frac{dv}{2\pi v}\\
&  \times \mathbb{E} \otimes \bdec{0}{0}{1}^{\otimes 2} \Bigg[
1_{\mathcal{G}^{\{x, y\}}_{[n_0, m)}(x) \cap \mathcal{G}^{\{x, y\}}_{[n_0, m)}(y)} 1_{\mathcal{H}_j(\mathbf{b}) \cap \mathcal{H}_k(\tilde{\mathbf{b}})}E_x(\delta)^{-2}E_y(\delta)^{-2}\\
&\qquad \qquad \qquad \times
\mathcal{I}\left(\widetilde{\lambda}_x E_x(\delta) 
u \overline{F}_{\gamma}^{\{x\}}(x+\sqrt{u}\mathbf{b}; h^{x, \mathrm{rad}} + \widehat{X}^x + \overline{h}(x)) ) \right)\\
&\qquad \qquad \qquad \times
\mathcal{I}\left(\widetilde{\lambda}_y E_y(\delta) 
v \overline{F}_{\gamma}^{\{y\}}(y+\sqrt{v}\tilde{\mathbf{b}}; h^{y, \mathrm{rad}}(\cdot) +\widehat{X}^y + \overline{h}(y)) \right)
\Bigg].
\end{split}
\end{align}

\item Following the same scaling argument as in \eqref{eq:cross_ptwise_scale}, one can show that \eqref{eq:diagonal_ptwise_GP} is equal to (cf. \eqref{eq:cross_ptwise_final})
\begin{align*}
\begin{split}
& \int_0^{\infty} \int_0^{\infty} \frac{dsdt}{\pi^2}
 \mathbb{E} \otimes \bdec{0}{0}{1}^{\otimes 2} \Bigg[
1_{\mathcal{G}^{\{x, y\}}_{[n_0, m)}(x) \cap \mathcal{G}^{\{x, y\}}_{[n_0, m)}(y)} 1_{\mathcal{H}_j(\mathbf{b}) \cap \mathcal{H}_k(\tilde{\mathbf{b}})}E_x(\delta)^{-2}E_y(\delta)^{-2}\\
&\qquad \qquad \qquad \times
\mathcal{I}\left(\widetilde{\lambda}_x E_x(\delta) 
\mathcal{R}_x e^{\gamma (B_{x, s} - (Q-\gamma)s)} ) \right)
\mathcal{I}\left(\widetilde{\lambda}_y E_y(\delta) 
\mathcal{R}_y e^{\gamma (B_{y, t} - (Q-\gamma)t)} ) \right)
\Bigg]
\end{split}
\end{align*}
where $(B_{x, s})_{s \ge 0}$ and $(B_{y, t})_{t \ge 0}$ are two standard Brownian motions independent of each other and everything else, and we are ready to apply \Cref{lem:main} to obtain a uniform bound (summable over $j, k \ge 1$) as well as the limiting value as $\lambda \to \infty$.
\end{itemize}

\noindent Summarising all the analysis above, one obtains by dominated convergence
\begin{align*}
&\limsup_{\lambda \to \infty} \sum_{j, k \ge 1} I_{j, k}\\
& \le \limsup_{\delta \to 0^+} \sum_{j, k \ge 1} \lim_{\lambda \to \infty} \int_0^{\infty} \int_0^{\infty} \frac{dsdt}{\pi^2}
 \mathbb{E} \otimes \bdec{0}{0}{1}^{\otimes 2} \Bigg[
1_{\mathcal{G}^{\{x, y\}}_{[n_0, m)}(x) \cap \mathcal{G}^{\{x, y\}}_{[n_0, m)}(y)} 1_{\mathcal{H}_j(\mathbf{b}) \cap \mathcal{H}_k(\tilde{\mathbf{b}})}\\
&\qquad \times
E_x(\delta)^{-2}E_y(\delta)^{-2} \mathcal{I}\left(\widetilde{\lambda}_x E_x(\delta) 
\mathcal{R}_x e^{\gamma (B_{x, s} - (Q-\gamma)s)} ) \right)
\mathcal{I}\left(\widetilde{\lambda}_y E_y(\delta) 
\mathcal{R}_y e^{\gamma (B_{y, t} - (Q-\gamma)t)} ) \right)
\Bigg]\\
& = \limsup_{\delta \to 0^+}\sum_{j, k \ge 1}  c_\gamma^2  \mathbb{E}\left[
1_{\mathcal{G}^{\{x, y\}}_{[n_0, m)}(x) \cap \mathcal{G}^{\{x, y\}}_{[n_0, m)}(y)}
E_x(\delta)^{-2}E_y(\delta)^{-2}\right]\pdec{0}{0}{1}\left( \mathcal{H}_j\right)\pdec{0}{0}{1}\left( \mathcal{H}_k\right)\\
& = c_\gamma^2 \mathbb{P}\left(\mathcal{G}^{\{x, y\}}_{[n_0, m)}(x) \cap \mathcal{G}^{\{x, y\}}_{[n_0, m)}(y)\right),
\end{align*}

\noindent and when combined with an analogous lower bound this concludes the proof of \eqref{eq:L1_diagonal_pointwise_cutoff}.
\paragraph{Step (ii).} We want to establish a ``two-point" analogue of \eqref{eq:cross_removecutoff}, i.e. 
\begin{align}\notag
\lim_{m \to \infty}\limsup_{\lambda \to \infty} 
\mathbb{E} \Bigg[
1_{\mathcal{G}^{\{x, y\}}_{[m, \infty)}(p)^c} &
\left(\int_0^1 \frac{du}{2\pi u} \bdec{x}{x}{u}\left[ \mathcal{I}\left(\lambda F_{\gamma}^{\{x, y\}}(\mathbf{b})\right)\right]\right)\\
\label{eq:diagonal_removecutoff}
& \qquad \times 
\left(\int_0^1 \frac{du}{2\pi v} \bdec{y}{y}{v}\left[ \mathcal{I}\left(\lambda F_{\gamma}^{\{x, y\}}(\tilde{\mathbf{b}})\right)\right]\right)
\Bigg] =  0
\end{align}

\noindent for $p \in \{x, y\}$ and any fixed and distinct $x, y \in  D$ satisfying $d(x, \partial D) \wedge d(y, \partial D) \ge \kappa$. 

To do so, we first use \eqref{eq:ind_exp_trick1} and follow the argument in \eqref{eq:cross_cutoff_limit0} to bound the expectation in \eqref{eq:diagonal_removecutoff} by
\begin{align}
\notag
\frac{e^{(\frac{\beta^2}{2} + 2\beta \gamma )C_{\kappa}}}{|x-y|^\gamma} 
\sum_{n \ge m} 2^{-\frac{\beta}{2}\left[2(\alpha - \gamma ) - \beta\right]n} 
\mathbb{E} \Bigg[&
\left(\int_0^1 \frac{du}{2\pi u} \bdec{x}{x}{u}\left[ \mathcal{I}\left(\lambda F_{\gamma, (p, n, \beta)}^{\{x, y\}}(\mathbf{b})\right)\right]\right)\\
\label{eq:diagonal_cutoff_limit0}
& \qquad \times \left(\int_0^1 \frac{dv}{2\pi v} \bdec{y}{y}{v}\left[ \mathcal{I}\left(\lambda F_{\gamma, (p, n, \beta)}^{\{x, y\}}(\tilde{\mathbf{b}})\right)\right]\right)
\Bigg]
\end{align}

\noindent where $F_{\gamma, (p, n, \beta)}^{\{x, y\}}(\cdot)$ was defined in \eqref{eq:Frare}, and $\beta \in (0, 2(\alpha - \gamma))$ is fixed.

Recall $c(x, y) := \frac{1}{8} \min(|x-y|, \kappa)$. Based on a splitting analysis similar to that in \eqref{eq:cross_cutoff_limit1}, the proof is complete if we can show, for some $\delta \in (0,  c(x, y))$, that
\begin{align}
\notag
\limsup_{m \to \infty} \limsup_{\lambda \to \infty}
\sum_{n \ge m} & 2^{-\frac{\beta}{2}\left[2(\alpha - \gamma ) - \beta\right]n}\\
\notag
&\times  \sum_{j, k \ge 1} \mathbb{E} \Bigg[
\left(\int_0^{(2^{-n}\delta / j)^2} \frac{du}{2\pi u} \bdec{x}{x}{u}\left[ \mathcal{I}\left(\lambda F_{\gamma, (p, n, \beta)}^{\{x, y\}}(\mathbf{b})\right) 1_{\mathcal{H}_j(\mathbf{b})}\right]\right)\\
\label{eq:diagonal_cutoff_limit1}
& \qquad \times \left(\int_0^{(2^{-n}\delta / k)^2} \frac{dv}{2\pi v} \bdec{y}{y}{v}\left[ \mathcal{I}\left(\lambda F_{\gamma, (p, n, \beta)}^{\{x, y\}}(\tilde{\mathbf{b}})\right)1_{\mathcal{H}_k(\tilde{\mathbf{b}})}\right]\right)
\Bigg]=0.
\end{align}

But by \eqref{eq:two-to-one}, one can check easily that
\begin{align*}
&\mathbb{E} \Bigg[\left(\int_0^{(2^{-n}\delta / j)^2} \frac{du}{2\pi u} \bdec{x}{x}{u}\left[ \mathcal{I}\left(\lambda F_{\gamma, (p, n, \beta)}^{\{x, y\}}(\mathbf{b})\right) 1_{\mathcal{H}_j(\mathbf{b})}\right]\right)\\
& \qquad \qquad \qquad \times \left(\int_0^{(2^{-n}\delta / k)^2} \frac{dv}{2\pi v} \bdec{y}{y}{v}\left[ \mathcal{I}\left(\lambda F_{\gamma, (p, n, \beta)}^{\{x, y\}}(\tilde{\mathbf{b}})\right)1_{\mathcal{H}_k(\tilde{\mathbf{b}})}\right]\right)
\Bigg]\\
&\qquad \lesssim
\mathbb{E} \Bigg[\left(\int_0^{(2^{-n}\delta / j)^2} \frac{du}{2\pi u} \bdec{x}{x}{u}\left[ \mathcal{I}\left(\widetilde{\lambda}_{x, p} F_{\gamma}^{\{x\}}(\mathbf{b})\right) 1_{\mathcal{H}_j(\mathbf{b})}\right]\right)\\
& \qquad \qquad \qquad \times \left(\int_0^{(2^{-n}\delta / k)^2} \frac{dv}{2\pi v} \bdec{y}{y}{v}\left[ \mathcal{I}\left(\widetilde{\lambda}_{y, p} F_{\gamma}^{\{y\}}(\tilde{\mathbf{b}})\right)1_{\mathcal{H}_k(\tilde{\mathbf{b}})}\right]\right)
\Bigg]
\end{align*}

\noindent for some suitable $\widetilde{\lambda}_{x, p}, \widetilde{\lambda}_{y, p} > 0$  (cf. \eqref{eq:two-to-one2}), and the above inequality is $ \lesssim \pdec{0}{0}{1}\left(\mathcal{H}_j\right)\pdec{0}{0}{1}\left(\mathcal{H}_k\right)$ by \Cref{lem:G2_uniform}. Thus \eqref{eq:diagonal_cutoff_limit1} is upper bounded (up to a multiplicative factor) by
\begin{align*}
&\limsup_{m \to \infty} \limsup_{\lambda \to \infty}
\sum_{n \ge m}  2^{-\frac{\beta}{2}\left[2(\alpha - \gamma ) - \beta\right]n} \sum_{j, k \ge 1} \pdec{0}{0}{1}\left(\mathcal{H}_j\right)\pdec{0}{0}{1}\left(\mathcal{H}_k\right) \\
&\qquad  \lesssim \limsup_{m \to \infty} 
2^{-\frac{\beta}{2}\left[2(\alpha - \gamma ) - \beta\right]m}  = 0
\end{align*}

\noindent and this concludes the proof of Lemma \ref{lem:diagonal_pointwise}. 
Combining with the other estimates in this section, this also concludes the proof of Theorem \ref{theo:Weyl}.
\end{proof}

\subsection{Proof of Theorems \ref{T:Weyl_intro} and \ref{theo:lhk_laplace}.}

Given Theorem \ref{theo:Weyl}, the proof of Theorem \ref{theo:lhk_laplace} proceeds as explained in Section \ref{SS:mainidea}.  In short, Theorem \ref{theo:Weyl} and the bridge decomposition establish that
$$
\int_0^\infty e^{-\lambda u} u\mathbf{S}_{\gamma}(u) du \sim \frac{c_\gamma \mu_{\gamma}(D)}{\lambda}
\qquad \text{as} \qquad \lambda \to \infty
$$
which is \eqref{eq:2nd_tau}. By an application of the Tauberian theorem (Theorem \ref{theo:tauberian}) this implies 
$$
\int_0^t u \mathbf{S}_{\gamma}(u) du \sim c_{\gamma} \mu_{\gamma}(D) t
\qquad \text{as} \qquad t \to 0^+,
$$
which is \eqref{eq:2nd_trans}. Lemma \ref{lem:asympdiff} implies that 
$$
t     \mathbf{S}_{\gamma}(t) \to  {c_{\gamma}  \mu_{\gamma}(D) }
$$
in probability, as desired for Theorem \ref{theo:lhk_laplace}.

Since $\mathbf{S}_{\gamma}(t)$ is the Laplace transform of the eigenvalue counting function $\mathbf{N}_{\gamma}(\lambda)$, Theorem \ref{T:Weyl_intro} follows again from an application of the probabilistic Tauberian theorem (Theorem \ref{theo:tauberian}).
\hfill \qed

\section{Pointwise heat kernel asymptotics} \label{sec:proof_ptwise}

\subsection{Proof of \Cref{T:HK}}\label{sec:proofHK}
Based on a similar scaling argument as before, let us assume that $\mathrm{diam}(D) < \frac{1}{2}$, and we shall continue to write $c_\gamma = c_\gamma(Q-\gamma; \mathcal{I})$ throughout \Cref{sec:proofHK} without risk of confusion. By standard approximation argument, it suffices to establish \Cref{T:HK} for test functions $f$ that are uniformly bounded and Lipschitz, and without loss of generality suppose
\begin{align}\label{eq:testfn_regularity}
\sup_{x \in \overline{D}, u \in \RR_+} |f(x, u)| + \sup_{x \in \overline{D}} \left[ \sup_{u, v \in \RR_+} \left| \frac{f(x, u) - f(x, v)}{u-v}\right|\right] \le 1.
\end{align}

To begin with, we apply the bridge decomposition and rewrite the LHS of \eqref{E:HK_Taub} as
\begin{align*}
& \mathbb{E}\left[ \int_D \mu_\gamma (dx)  f ( x, J_\gamma^\lambda(x) ) \right]\\
& = \mathbb{E}\left[ \int_D \mu_\gamma (dx) f \left(x, \int_0^\infty \frac{du}{2\pi u} \bdec{x}{x}{u}[ \mathcal{I} \left(\lambda  F_{\gamma}(\mathbf{b})\right)1_{\{u< \tau_D(\mathbf{b}) \}}] \right) \right]\\
& = \int_D R(x; D)^{\frac{\gamma^2}{2}} dx \mathbb{E}\left[ f \left(x, \int_0^\infty \frac{du}{2\pi u} \bdec{x}{x}{u}[ \mathcal{I} \left(\lambda  F_{\gamma}^{\{x\}}(\mathbf{b})\right)1_{\{u< \tau_D(\mathbf{b}) \}}] \right)  \right].
\end{align*}

\noindent Since $f$ is uniformly bounded, the expectation in the integrand above is bounded, and by dominated convergence we just need to show that
\begin{align*}
\lim_{\lambda \to \infty} \mathbb{E}\left[ f \left(x, \int_0^\infty \frac{du}{2\pi u} \bdec{x}{x}{u}[ \mathcal{I} \left(\lambda  F_{\gamma}^{\{x\}}(\mathbf{b})\right)1_{\{u< \tau_D(\mathbf{b}) \}}] \right) \right]
= \mathbb{E} [ f(x, J_\gamma^\infty) ]
\end{align*}

\noindent for any $\kappa > 0$ and $x \in D$ satisfying $d(x, \partial D) \ge 2\kappa$ (see \Cref{lem:Jinfty_limit} for the definition of $J_\gamma^\infty$).

\subsubsection{Step 1: truncating the time integral}
Let $\delta_1 \in (0, 1)$ be some fixed but arbitrary number (possibly dependent on $x$). Similar to our proof of \Cref{theo:Weyl} we would first like to truncate the $u$-integral:
\begin{lem}
Let $x \in D$ satisfying $d(x, \partial D) \ge 2\kappa$. We have
\begin{align}
\notag
& \limsup_{\lambda \to \infty} \Bigg| \mathbb{E}\left[ f \left(x, \int_0^\infty \frac{du}{2\pi u} \bdec{x}{x}{u}[ \mathcal{I} \left(\lambda  F_{\gamma}^{\{x\}}(\mathbf{b})\right)1_{\{u< \tau_D(\mathbf{b}) \}}] \right) \right]\\
\label{eq:ptwise_truncate}
& \qquad \qquad -  \mathbb{E}\left[ f \left(x, \int_0^{\delta_1^2} \frac{du}{2\pi u} \bdec{x}{x}{u}[ \mathcal{I} \left(\lambda  F_{\gamma}^{\{x\}}(\mathbf{b})\right)1_{\{u< \tau_D(\mathbf{b}) \}}] \right) \right]\Bigg| = 0.
\end{align}
\end{lem}

\begin{proof}
Thanks to the Lipschitz control \eqref{eq:testfn_regularity}, the LHS of \eqref{eq:ptwise_truncate} (before taking the limit $\lambda \to \infty$) is bounded by
\begin{align*}
\EE\left[ \int_{\delta_1^2}^\infty   \frac{du}{2\pi u} \bdec{x}{x}{u}[ \mathcal{I} \left(\lambda  F_{\gamma}^{\{x\}}(\mathbf{b})\right)1_{\{u< \tau_D(\mathbf{b}) \}}] \right].
\end{align*}

\noindent Since $\Ia(\cdot) \le 1$, we know from \Cref{cor:bb_bound} (with the assumption $\mathrm{diam}(D) < \frac{1}{2}$) that
\begin{align*}
\bdec{x}{x}{u}[ \mathcal{I} \left(\lambda  F_{\gamma}^{\{x\}}(\mathbf{b})\right)1_{\{u< \tau_D(\mathbf{b}) \}}]  
& \le \pdec{x}{x}{u}\left( \mathbf{b}_s \in  D ~ \forall s \le u\right)\\
& \le \pdec{x}{x}{u}\left( |\mathbf{b}_s - x | \le 1 ~ \forall s \le u\right)
\le 1 \wedge \frac{2}{u}
\end{align*}

\noindent which is integrable with respect to $du/2\pi u$ on $[\delta_1^2, \infty)$. As $\mathcal{I} \left(\lambda  F_{\gamma}^{\{x\}}(\mathbf{b})\right) \xrightarrow{\lambda \to \infty} 0$ almost surely, it follows from dominated convergence that
\begin{align*}
\lim_{\lambda \to \infty} \EE\left[ \int_{\delta_1^2}^\infty   \frac{du}{2\pi u} \bdec{x}{x}{u}[ \mathcal{I} \left(\lambda  F_{\gamma}^{\{x\}}(\mathbf{b})\right)1_{\{u< \tau_D(\mathbf{b}) \}}] \right] = 0
\end{align*}

\noindent which leads to the desired claim \eqref{eq:ptwise_truncate}.
\end{proof}

\subsubsection{Step 2: restricting the range of Brownian bridge}
The next step would be to restrict the range of our Brownian bridge $\mathbf{b}$. Unlike the proof of \Cref{theo:Weyl} where we needed to partition the probability space, here we introduce a cutoff parameter $n \in \mathbb{N}$ and assume from now that $\delta_1$ is small enough such that $4n \delta_1 <\kappa$.
\begin{lem}
Let
\begin{align*}
\overline{\mathcal{H}}_n
= \overline{\mathcal{H}}_n(\mathbf{b})
=  \left\{ \max_{s\le \ell(\mathbf{b})} \frac{|\mathbf{b}_s -\iota(\mathbf{b})|}{\sqrt{\ell(\mathbf{b})}} < n \right\}
= \bigcup_{k=1}^n \mathcal{H}_k.
\end{align*}

\noindent Then for any $x \in D$ satisfying $d(x, \partial D) \ge 2\kappa$, we have
\begin{align}
\notag
\limsup_{n \to \infty} \limsup_{\delta_1 \to 0^+}  \limsup_{\lambda \to \infty} 
& \Bigg| \mathbb{E}\left[ f \left(x, \int_0^{\delta_1^2} \frac{du}{2\pi u} \bdec{x}{x}{u}[ \mathcal{I} \left(\lambda  F_{\gamma}^{\{x\}}(\mathbf{b})\right)1_{\{u< \tau_D(\mathbf{b}) \}}] \right) \right] \\
\label{eq:ptwise_restrict}
& \qquad 
 -\mathbb{E}\left[ f \left(x, \int_0^{\delta_1^2} \frac{du}{2\pi u} \bdec{x}{x}{u}\left[ \mathcal{I} \left(\lambda  F_{\gamma}^{\{x\}}(\mathbf{b})\right)1_{\overline{\mathcal{H}}_n}\right] \right) \right] \Bigg| = 0.
\end{align}
\end{lem}

\begin{proof}
The LHS of \eqref{eq:ptwise_restrict} (before taking any of the limit) is bounded by
\begin{align*}
& \mathbb{E}\left[ \int_0^{\delta_1^2} \frac{du}{2\pi u} \bdec{x}{x}{u}\left[ \mathcal{I} \left(\lambda  F_{\gamma}^{\{x\}}(\mathbf{b})\right)1_{\{u< \tau_D(\mathbf{b}) \}\cap \overline{\mathcal{H}}_n^c}\right] \right]\\
& \qquad\qquad  \le  \sum_{k  \ge n+1} \int_{\delta_1^2 k^{-2}}^{\delta_1^2} \frac{du}{2\pi u} \mathbb{E}\left[  \bdec{x}{x}{u}\left[ \mathcal{I} \left(\lambda  F_{\gamma}^{\{x\}}(\mathbf{b})\right)1_{\mathcal{H}_k}\right] \right]\\
&  \qquad \qquad \qquad \qquad +  \sum_{k  \ge n+1} \int_0^{\delta_1^2 k^{-2}} \frac{du}{2\pi u} \mathbb{E}\left[  \bdec{x}{x}{u}\left[ \mathcal{I} \left(\lambda  F_{\gamma}^{\{x\}}(\mathbf{b})\right)1_{\mathcal{H}_k}\right] \right].
\end{align*}

\noindent where $\mathcal{H}_k$ was defined in \eqref{eq:eventH}. The first sum is upper bounded by
\begin{align*}
\sum_{k  \ge n+1} \int_{\delta_1^2 k^{-2}}^{\delta_1^2} \frac{du}{2\pi u} \mathbb{E}\left[  \bdec{x}{x}{u}\left[ \mathcal{I} \left(\lambda  F_{\gamma}^{\{x\}}(\mathbf{b})\right)1_{\mathcal{H}_k}\right] \right]
& \le  \sum_{k  \ge n+1} \int_{\delta_1^2 k^{-2}}^{\delta_1^2} \frac{du}{2\pi u} \pdec{x}{x}{u} \left( \mathcal{H}_k\right) \\
& \le \sum_{k  \ge n+1} \int_{\delta_1^2 k^{-2}}^{\delta_1^2} \frac{du}{2\pi u} \cdot 4 e^{-\frac{(k-1)^2}{2}}\\
& \le \sum_{k  \ge n+1} 2 e^{-\frac{(k-1)^2}{2}} \log k
\end{align*}

\noindent  where the second last inequality follows from \Cref{cor:bb_bound}. This vanishes as $n \to \infty$ uniformly in $\lambda$ and $\delta_1$. 

Let us look at the second sum. Since $4k\sqrt{u} \le 4k \sqrt{\delta_1^2 k^{-2}} = 4 \delta_1 \le \kappa \le d(x, \partial D)$ for $u \in [0, \delta_1^2k^{-2}]$, we obtain
\begin{align*}
&  \sum_{k  \ge n+1} \int_0^{\delta_1^2 k^{-2}} \frac{du}{2\pi u} \mathbb{E}\left[  \bdec{x}{x}{u}\left[ \mathcal{I} \left(\lambda  F_{\gamma}^{\{x\}}(\mathbf{b})\right)1_{\mathcal{H}_k}\right] \right]\\
& \qquad \le \sum_{k  \ge n+1} \int_0^{1} \frac{du}{2\pi u} 1_{\{d(x, \partial D) \ge 4k \sqrt{u}\}} \mathbb{E}\left[  \bdec{x}{x}{u}\left[ \mathcal{I} \left(\lambda  F_{\gamma}^{\{x\}}(\mathbf{b})\right)1_{\mathcal{H}_k}\right] \right]\\
& \qquad \le C \sum_{k \ge n+1} \pdec{0}{0}{1}\left(\mathcal{H}_k\right) = C \pdec{0}{0}{1}\left(\mathcal{H}_{n}^c\right) 
\end{align*}

\noindent where the last inequality follows from \Cref{lem:G1_uniform} with $C > 0$  independent of $\lambda$. This bound again vanishes uniformly in $\lambda$ and $\delta_1$ as $n\to \infty$, and this concludes the proof of \eqref{eq:ptwise_restrict}.
\end{proof}

\subsubsection{Step 3: decomposition of Gaussian free field}
We now need to argue that the Gaussian free field $h(\cdot)$ locally behaves like an exactly scale invariant field. In the proof of \Cref{theo:Weyl}, this was achieved by Gaussian interpolation/comparison. It is not clear how this method could be adapted to the analysis here, though, since we are dealing with arbitrary test functions $f$. We shall therefore pursue a different strategy based on the decomposition of Gaussian fields.

Applying the domain Markov property of Gaussian free field similar to that in \eqref{eq:GFF_markov}, we can write
\begin{align*}
h(\cdot) = \overline{h}(\cdot) + h^{x, \eta}(\cdot) + h^{y, \eta}(\cdot)
\end{align*}

\noindent but here we choose $\eta \in (\kappa/2, \kappa)$ (and in particular $\delta_2 := n\delta_1 < \eta$). Since the random variable $F_{\gamma}^{\{x\}}(\mathbf{b})$ (recall \eqref{eq:Girsanov1}) only depends on  $h(\cdot)$ on $\overline{B}(x, \delta_2)$ on the event $\overline{H}_n$ when we restrict $u \in [0, \delta_1^2]$ and \eqref{eq:liouville_clock} can be rewritten as
\begin{align}\label{eq:liouville_clock_dec}
F_{\gamma}(ds; \mathbf{b}) := 
e^{\gamma \overline{h}(\mathbf{b}_s) - \frac{\gamma^2}{2} \mathbb{E}[\overline{h}(\mathbf{b}_s)^2]} 
e^{\gamma h^{x, \eta}(\mathbf{b}_s) - \frac{\gamma^2}{2} \mathbb{E}[h^{x, \eta}(\mathbf{b}_s)^2]}
R(\mathbf{b}_s; D)^{\frac{\gamma^2}{2}} 1_{\{\mathbf{b}_s \in \overline{B}(x, \delta_2)\}}ds.
\end{align}

\noindent We shall perform further decomposition with the help of \Cref{lem:GFF_dec}, and write
\begin{align*}
h^{p, \eta}(\cdot) = X^{p, \eta}(\cdot) - Y^{p, \eta}(\cdot) \qquad \text{on $B(p, \eta)$}
\end{align*}

\noindent for $p \in \{x, y\}$, where $X^{p, \eta}(\cdot) \overset{d}{=} X^{\eta \mathbb{D}}(\cdot - p)$ and $Y^{p, \eta}(\cdot) \overset{d}{=} Y^{\eta \mathbb{D}}(\cdot - p)$ in the notation of \eqref{eq:GFF_dec}.
We claim that when $\delta_1$ (and hence $\delta_2$) is small, $F_{\gamma}^{\{x\}}(\mathbf{b})$ is approximately equal to $R(x; D)^{\frac{3\gamma^2}{2}}e^{\gamma \overline{h}(x) - \frac{\gamma^2}{2} \mathbb{E}[\overline{h}(x)^2]} \overline{F}_{\gamma}^{\{x\}}(\mathbf{b};  X^{x, \eta})$ where (recalling  \eqref{eq:overlineF})
\begin{align*}
\overline{F}_{\gamma}^{\{x\}}(\mathbf{b};  X^{x, \eta})
& :=
\int_{0}^{\ell(\mathbf{b})} e^{\gamma  X^{x, \eta}(\mathbf{b}_s) - \frac{\gamma^2}{2}\mathbb{E}[X^{x, \eta}(\mathbf{b}_s)^2]} \frac{1_{\{\mathbf{b}_s \in \overline{B}(x, \delta_2)\}}ds}{| \mathbf{b}_s - x |^{\gamma^2}}.
\end{align*}

\begin{lem}
For any $x \in D$ satisfying $d(x, \partial D) \ge 2\kappa$, we have
\begin{align}
\notag
&  \limsup_{\delta_1 \to 0^+}  \limsup_{\lambda \to \infty} 
\Bigg| \mathbb{E}\left[ f \left(x, \int_0^{\delta_1^2} \frac{du}{2\pi u} \bdec{x}{x}{u}\left[ \mathcal{I} \left(\lambda  F_{\gamma}^{\{x\}}(\mathbf{b})\right)1_{\overline{\mathcal{H}}_n}\right] \right) \right.\\
\label{eq:local_exact}
& \quad 
 -\left. f \left(x, \int_0^{\delta_1^2} \frac{du}{2\pi u} \bdec{x}{x}{u}\left[ \mathcal{I} \left(\lambda  
R(x; D)^{\frac{3\gamma^2}{2}}e^{\gamma \overline{h}(x) - \frac{\gamma^2}{2} \mathbb{E}[\overline{h}(x)^2]} \overline{F}_{\gamma}^{\{x\}}(\mathbf{b};  X^{x, \eta})
\right)1_{\overline{\mathcal{H}}_n}\right] \right) \right] \Bigg| = 0.
\end{align}
\end{lem}

\begin{proof}
Fix $\epsilon \in (0, 1)$, and suppose $\delta_1 > 0$ (and hence $\delta_2 := n \delta_1 > 0$) is sufficiently small such that
\begin{align*}
(1+\epsilon)^{-1} R(x; D) \le R(w; D) \le (1+\epsilon) R(x; D) \qquad \forall w \in \overline{B}(x, \delta_2)
\end{align*}

\noindent as well as 
\begin{align*}
\bigg| G_0^D(x, w) - \left[ -\log|x-w| + \log R(x; D)\right]\bigg| \le \epsilon \qquad \forall w \in  \overline{B}(x, \delta_2)
\end{align*}

\noindent which is possible by \Cref{lem:Green_estimate}. We also introduce the event
\begin{align*}
\mathcal{O}_\epsilon(x, \delta_2) 
&:= \left\{ \left|\left(\gamma \overline{h}(w) - \frac{\gamma^2}{2} \mathbb{E}[\overline{h}(w)^2]\right) - \left(\gamma \overline{h}(x) - \frac{\gamma^2}{2} \mathbb{E}[\overline{h}(x)^2] \right)\right| \le \epsilon \quad \forall w \in \overline{B}(x, \delta_2)  \right\}\\
& \qquad \cap \left\{ \left|\gamma Y^{x, \eta}(w) - \frac{\gamma^2}{2} \mathbb{E}[Y^{x, \eta}(w)^2] \right| \le \epsilon \quad \forall w \in \overline{B}(x, \delta_2)  \right\}
\end{align*}

\noindent and bound the LHS of \eqref{eq:local_exact} by
\begin{align}
\notag
\mathbb{P}(O_\epsilon(x, \delta_2)^c)
& + \mathbb{E} \bigg\{ 1_{O_\epsilon(x, \delta_2)}\bigg| \int_0^{\delta_1^2} \frac{du}{2\pi u} \bdec{x}{x}{u}\left[ \mathcal{I} \left(\lambda  F_{\gamma}^{\{x\}}(\mathbf{b})\right)1_{\overline{\mathcal{H}}_n}\right]\\
\label{eq:local_exact_bound}
&  -
\int_0^{\delta_1^2} \frac{du}{2\pi u} \bdec{x}{x}{u}\left[ \mathcal{I} \left(\lambda  
R(x; D)^{\frac{3\gamma^2}{2}}e^{\gamma \overline{h}(x) - \frac{\gamma^2}{2} \mathbb{E}[\overline{h}(x)^2]} \overline{F}_{\gamma}^{\{x\}}(\mathbf{b};  X^{x, \eta})
\right)1_{\overline{\mathcal{H}}_n}\right] 
\bigg|\bigg\}.
\end{align}

Let us further rewrite \eqref{eq:liouville_clock_dec} (on the event $\mathcal{O}_\epsilon(x, \delta_2)$ and $\overline{\mathcal{H}}_n$) as
\begin{align*}
F_{\gamma}(ds; \mathbf{b}) 
& :=  e^{\gamma \overline{h}(\mathbf{b}_s) - \frac{\gamma^2}{2} \mathbb{E}[\overline{h}(\mathbf{b}_s)^2]} 
\left[e^{\gamma Y^{x, \eta}(\mathbf{b}_s) - \frac{\gamma^2}{2} \mathbb{E}[Y^{x, \eta}(\mathbf{b}_s)^2]}\right]^{-1}\\
& \qquad  \qquad \times 
e^{\gamma X^{x, \eta}(\mathbf{b}_s) - \frac{\gamma^2}{2} \mathbb{E}[X^{x, \eta}(\mathbf{b}_s)^2]}
R(\mathbf{b}_s; D)^{\frac{\gamma^2}{2}} 1_{\{\mathbf{b}_s \in \overline{B}(x, \delta_2)\}}ds.
\end{align*}

\noindent Then based on the definition of $\epsilon$ as well as the event $\mathcal{O}_\epsilon(x, \delta_2)$, it is straightforward to verify that
\begin{align*}
C(\epsilon)^{-1} 
\le  \frac{F_{\gamma}^{\{x\}}(\mathbf{b}) }{R(x; D)^{\frac{3\gamma^2}{2}}e^{\gamma \overline{h}(x) - \frac{\gamma^2}{2} \mathbb{E}[\overline{h}(x)^2]} \overline{F}_{\gamma}^{\{x\}}(\mathbf{b};  X^{x, \eta})}
\le C(\epsilon) 
\end{align*}

\noindent where $C(\epsilon) = (1+\epsilon)^{\frac{\gamma^2}{2}} e^{(\gamma^2 +2) \epsilon}$. Combining this two-sided control with the fact that
\begin{align*}
|\mathcal{I}(u) - \mathcal{I}(v)| 
& = \left| ue^{-u} - v e^{-v}\right|
\le \int_v^u \left |(1-s) e^{-s}\right|ds
\le 2(u-v) e^{-\frac{v}{2}}
\end{align*}

\noindent for any $u \ge v \ge 0$, one can check that
\begin{align*}
& \left| \mathcal{I} \left(\lambda  F_{\gamma}^{\{x\}}(\mathbf{b})\right)
- \mathcal{I} \left(\lambda  
R(x; D)^{\frac{3\gamma^2}{2}}e^{\gamma \overline{h}(x) - \frac{\gamma^2}{2} \mathbb{E}[\overline{h}(x)^2]} \overline{F}_{\gamma}^{\{x\}}(\mathbf{b};  X^{x, \eta})
\right)
\right|\\
& \qquad \le 
4\left[C(\epsilon) - 1\right] C(\epsilon) \mathcal{I}\left(
\frac{\lambda}{2C(\epsilon)} R(x; D)^{\frac{3\gamma^2}{2}}e^{\gamma \overline{h}(x) - \frac{\gamma^2}{2} \mathbb{E}[\overline{h}(x)^2]} \overline{F}_{\gamma}^{\{x\}}(\mathbf{b};  X^{x, \eta})
\right).
\end{align*}

Summarising everything so far, the estimate \eqref{eq:local_exact_bound} can be bounded by
\begin{align} \label{eq:local_exact_bound2}
&\mathbb{P}(O_\epsilon(x, \delta_2)^c)
 + 4\left[C(\epsilon) - 1\right] C(\epsilon) \mathbb{E} \left[ \int_0^{\delta_1^2} \frac{du}{2\pi u} \bdec{x}{x}{u}\left[ 
\mathcal{I}\left(
\lambda C_x(\epsilon) \overline{F}_{\gamma}^{\{x\}}(\mathbf{b};  X^{x, \eta})
\right)
1_{\overline{\mathcal{H}}_n}\right]\right]
\end{align}

\noindent with $\displaystyle C_x(\epsilon) := \frac{1}{2C(\epsilon)} R(x; D)^{\frac{3\gamma^2}{2}}e^{\gamma \overline{h}(x) - \frac{\gamma^2}{2} \mathbb{E}[\overline{h}(x)^2]}.$

We now perform a space-time rescaling of the Brownian bridge \eqref{eq:BB_rescale}, and write
\begin{align}\notag
& \mathbb{E} \otimes\bdec{x}{x}{u}\left[ 
\mathcal{I}\left(
\lambda C_x(\epsilon)\overline{F}_{\gamma}^{\{x\}}(\mathbf{b};  X^{x, \eta})
\right)
1_{\overline{\mathcal{H}}_n}\right] \\
\notag
&\qquad  = \mathbb{E} \otimes\bdec{0}{0}{1}\left[ 
\mathcal{I}\left(
\lambda C_x(\epsilon) 
\overline{F}_{\gamma}^{\{x\}}(x + \sqrt{u}\mathbf{b}_{\cdot/u};  X^{x, \eta})
\right)
1_{\overline{\mathcal{H}}_n}\right]\\
\label{eq:pre-exact0}
& \qquad = \mathbb{E} \otimes\bdec{0}{0}{1}\left[ 
\mathcal{I}\left(
\lambda C_x(\epsilon) 
\overline{F}_{\gamma}^{\{0\}}(\sqrt{u}\mathbf{b}_{\cdot/u};  X^{\eta \mathbb{D}})
\right)
1_{\overline{\mathcal{H}}_n}\right]
\end{align}

\noindent where
\begin{align}
\notag
\overline{F}_{\gamma}^{\{0\}}(\sqrt{u}\mathbf{b}_{\cdot/u};  X^{\eta \mathbb{D}}) 
& =  \int_0^u  \frac{e^{\gamma X^{\eta \mathbb{D}}(\sqrt{u}\mathbf{b}_{s/u}) - \frac{\gamma^2}{2} \mathbb{E}[X^{\eta \mathbb{D}}( \sqrt{u}\mathbf{b}_{s/u})^2]}ds}{| \sqrt{u}\mathbf{b}_{s/u}|^{\gamma^2}}\\
\label{eq:pre-exact1}
& = u ^{1 - \frac{\gamma^2}{2}}\int_0^1 \frac{e^{\gamma X^{\eta \mathbb{D}}(\sqrt{u}\mathbf{b}_{s}) - \frac{\gamma^2}{2} \mathbb{E}[X^{\eta \mathbb{D}}(\sqrt{u}\mathbf{b}_{s})^2]}ds}{|\mathbf{b}_{s}|^{\gamma^2}}.
\end{align}

\noindent Since
\begin{align*}
&\mathbb{E}\left[X^{\eta \mathbb{D}}( n\sqrt{u} x_1)X^{\eta \mathbb{D}}(n\sqrt{u} x_2) \right]\\
&\quad = -\log|x_1 - x_2| - \log \frac{n\sqrt{u}}{\eta}
= \mathbb{E}\left[ X^{\mathbb{D}}(x_1) X^{\mathbb{D}}(x_2)\right] + \mathbb{E}[B_{\widetilde{T}(u, n)}^2]
\quad \forall x_1, x_2 \in \mathbb{D}
\end{align*}

\noindent where $\widetilde{T}(u, n) := -\log \frac{n\sqrt{u}}{\eta} > 0$ (as $2n\sqrt{u} \le 2n\delta_1 < \frac{\kappa}{2} < \eta$) and $B_{\widetilde{T}(u, n)} \sim \mathcal{N}(0, \widetilde{T}(u, n))$ is independent of $X^{\mathbb{D}}$, we see that \eqref{eq:pre-exact1} (on the event $\overline{H}_n$) is equal in distribution to
\begin{align*}
& u^{1-\frac{\gamma^2}{2}} e^{\gamma B_{\widetilde{T}(u, n)} - \frac{\gamma^2}{2} \widetilde{T}(u, n)}
\int_0^1 \frac{e^{\gamma X^{ \mathbb{D}}(n^{-1}\mathbf{b}_{s}) - \frac{\gamma^2}{2} \mathbb{E}[X^{\eta \mathbb{D}}(n^{-1}\mathbf{b}_{s})^2]}ds}{|\mathbf{b}_{s}|^{\gamma^2}}\\
& \qquad = u^{1-\frac{\gamma^2}{2}}  e^{\gamma B_{\widetilde{T}(u, n)} - \frac{\gamma^2}{2} \widetilde{T}(u, n)} n^{-\gamma^2} \overline{F}_{\gamma}^{\{0\}}(n^{-1}\mathbf{b};  X^{\mathbb{D}})\\
& \qquad = e^{\gamma (B_{\widetilde{T}(u, n)} - (Q-\gamma)\widetilde{T}(u, n))} (n/\eta)^{-(2-\gamma^2)}n^{-\gamma^2} \overline{F}_{\gamma}^{\{0\}}(n^{-1}\mathbf{b};  X^{\mathbb{D}}).
\end{align*}

\noindent Setting $\mathcal{E}:= C_x(\epsilon) (n/\eta)^{-(2-\gamma^2)}n^{-\gamma^2} \overline{F}_{\gamma}^{\{x\}}(\mathbf{b};  X^{\mathbb{D}})$, we obtain
\begin{align*}
&\mathbb{E} \left[ \int_0^{\delta_1^2} \frac{du}{2\pi u} \bdec{x}{x}{u}\left[ 
\mathcal{I}\left(
\lambda C_x(\epsilon) \overline{F}_{\gamma}^{\{x\}}(\mathbf{b};  X^{x, \eta})
\right)
1_{\overline{\mathcal{H}}_n}\right]\right]\\
& \qquad =
\mathbb{E} \otimes \bdec{0}{0}{1} \left[ \int_0^{\delta_1^2} \frac{du}{2\pi u} 
\mathcal{I}\left(
\lambda \mathcal{E} e^{\gamma (B_{\widetilde{T}(u; n)} - (Q-\gamma) \widetilde{T}(u; n))}
\right)
1_{\overline{\mathcal{H}}_n}
\right] \\
& \qquad \le
\mathbb{E} \otimes \bdec{0}{0}{1} \left[ \int_0^{\infty} \frac{dt}{\pi}
\mathcal{I}\left(
\lambda \mathcal{E} e^{\gamma (B_{t} - (Q-\gamma) t)}\right)
1_{\overline{\mathcal{H}}_n}
\right]
\le c_\gamma
\end{align*}

\noindent where the last inequality follows from \eqref{eq:uniform_main1} of \Cref{lem:main}. Therefore, \eqref{eq:local_exact_bound2} is uniformly bounded in $\lambda \to \infty$ by
\begin{align*}
\mathbb{P}(O_\epsilon(x, \delta_2)^c)
 + 4\left[C(\epsilon) - 1\right] C(\epsilon) \cdot c_\gamma.
\end{align*}

\noindent As $\delta_1 \to 0^+$ (and hence $\delta_2 \to 0^+$), we have $\mathbb{P}(O_\epsilon(x, \delta_2)^c) \to 0$ by the continuity of the Gaussian fields $\overline{h}(\cdot)$ and $Y^{x, \eta}(\cdot)$ in a neighbourhood of $x$. Since $\epsilon > 0$ is arbitrary, we can send $\epsilon \to 0^+$ and conclude that \eqref{eq:local_exact} holds.

\end{proof}

\subsubsection{Step 4: identifying the limiting random variable $J_\gamma^\infty$}
All that remains to be done is to establish the pointwise limit.
\begin{lem} \label{lem:Jinfty_limit}
Let $C_x:=R(x; D)^{\frac{3\gamma^2}{2}}e^{\gamma \overline{h}(x) - \frac{\gamma^2}{2} \mathbb{E}[\overline{h}(x)^2]}$. For any $x \in D$ satisfying $d(x, \partial D) \ge 2\kappa$, we have
\begin{align*}
\lim_{n \to \infty}\lim_{\delta_1 \to 0^+} \lim_{\lambda \to \infty} 
\mathbb{E}\left[f\left(x, \int_0^{\delta_1^2} \frac{du}{2\pi u} \bdec{x}{x}{u}\left[  \mathcal{I} \left(
\lambda  C_x \overline{F}_{\gamma}^{\{x\}}(\mathbf{b};  X^{x, \eta}) 
\right)1_{\overline{\mathcal{H}}_n}\right]\right) \right]
= \mathbb{E}\left[ f(x, J_\gamma^\infty)\right]
\end{align*}

\noindent with
\begin{align*}
J_\gamma^\infty
:= 
\int_{-\infty}^\infty \frac{dt}{\pi} \bdec{0}{0}{1}\Bigg[  \mathcal{I} \bigg(
\int_0^1 e^{-\gamma \beta_{t - \log|\mathbf{b}_s|}^{Q-\gamma}}
\frac{e^{\gamma \widehat{X}(e^{-t}\mathbf{b}_s) - \frac{\gamma^2}{2}\mathbb{E}[\widehat{X}(e^{-t}\mathbf{b}_s)^2]}ds}{|\mathbf{b}_s|^2}
 \bigg)\Bigg]
\end{align*}

\noindent where
\begin{itemize}
\item $\widehat{X}(\cdot)$ is a scale-invariant Gaussian field defined on $\mathbb{R}^2 \cong \mathbb{C}$ with covariance kernel
\begin{align*}
\mathbb{E} \left[\widehat{X}(x_1) \widehat{X}(x_2)\right]=  \log \frac{|x_1| \vee |x_2|}{|x_1 - x_2|};
\end{align*}
\item $(\beta_t^{Q-\gamma})_{t \in \mathbb{R}}$ is the $\gamma$-quantum cone, i.e. the two-sided stochastic process defined in \eqref{eq:beta_process} with $m = Q-\gamma$.
\end{itemize}
\end{lem}

\begin{proof}
We begin by standardising our Brownian bridge like \eqref{eq:pre-exact0}, i.e.
\begin{align*}
& \mathbb{E}\left[f\left(x, \int_0^{\delta_1^2} \frac{du}{2\pi u} \bdec{x}{x}{u}\left[  \mathcal{I} \left(
\lambda  C_x \overline{F}_{\gamma}^{\{x\}}(\mathbf{b};  X^{x, \eta}) \right)1_{\overline{\mathcal{H}}_n}\right]\right) \right]\\
& \qquad = \mathbb{E}\left[f\left(x, \int_0^{\delta_1^2} \frac{du}{2\pi u} \bdec{0}{0}{1}\left[  \mathcal{I} \left(
\lambda  C_x \overline{F}_{\gamma}^{\{0\}}( \sqrt{u}\mathbf{b}_{\cdot/u};  X^{\eta \mathbb{D}} \right)1_{\overline{\mathcal{H}}_n}\right]\right) \right].
\end{align*}

\noindent Unlike the proof of the last lemma where the exact scaling relation of $X^{\eta \mathbb{D}}$ was used, we have to proceed with the radial-lateral decomposition here: for $x_1, x_2 \in B(0, \eta)$ recall
\begin{align*}
\mathbb{E}\left[X^{\eta \mathbb{D}}(x_1)X^{\eta \mathbb{D}}(x_2) \right]
& = -\log \left| \frac{x_1}{\eta}\right| \vee \left| \frac{x_2}{\eta}\right|
+ \log \frac{|x_1| \vee |x_2|}{|x_1 - x_2|}\\
& = \mathbb{E}[B_{\widehat{T}(x_1)}B_{\widehat{T}(x_2)}] + \mathbb{E} \left[\widehat{X}(x_1) \widehat{X}(x_2)\right]
\end{align*}

\noindent where $\widehat{T}(\cdot) = -\log |\cdot/\eta|$ and $(B_t)_{t \ge 0}$ is a Brownian motion independent of $\widehat{X}(\cdot)$. Then \eqref{eq:pre-exact1} is equal to
\begin{align*}
&\overline{F}_{\gamma}^{\{0\}}(\sqrt{u}\mathbf{b}_{\cdot/u};  X^{\eta \mathbb{D}}) \\
& \quad = \eta^{2-\gamma^2} \int_0^1 
|\sqrt{u} \mathbf{b}_s / \eta|^{2-\gamma^2}
\frac{e^{\gamma X^{\eta \mathbb{D}}(\sqrt{u}\mathbf{b}_{s}) - \frac{\gamma^2}{2} \mathbb{E}[X^{\eta \mathbb{D}}(\sqrt{u}\mathbf{b}_{s})^2]}ds}{|\mathbf{b}_{s}|^{2}}\\
& \quad=\eta^{2-\gamma^2} \int_0^1 e^{\gamma \left[B_{\widehat{T}(\sqrt{u}\mathbf{b}_s)} -(Q-\gamma) \widehat{T}(\sqrt{u}\mathbf{b}_s)\right]}
\frac{e^{\gamma \widehat{X}(\sqrt{u}\mathbf{b}_s) - \frac{\gamma^2}{2}\mathbb{E}[\widehat{X}(\sqrt{u}\mathbf{b}_s)^2]}ds}{|\mathbf{b}_s|^2}
\end{align*}

\noindent and thus
\begin{align}
\notag
&\int_0^{\delta_1^2} \frac{du}{2\pi u} \bdec{0}{0}{1}\left[  \mathcal{I} \left(
\lambda  C_x \overline{F}_{\gamma}^{\{0\}}( \sqrt{u}\mathbf{b}_{\cdot/u};  X^{\eta \mathbb{D}} \right)1_{\overline{\mathcal{H}}_n}\right]\\
\notag
&  = \int_{-\log \delta_1}^\infty \frac{dt}{\pi} \bdec{0}{0}{1}\Bigg[  \mathcal{I} \bigg(
\lambda  C_x \eta^{2-\gamma^2} \\
\notag
&\qquad \qquad \times 
\int_0^1 e^{\gamma \left[B_{t - \log|\mathbf{b}_s / \eta|} -(Q-\gamma)(t - \log|\mathbf{b}_s / \eta|)\right]}
\frac{e^{\gamma \widehat{X}(e^{-t}\mathbf{b}_s) - \frac{\gamma^2}{2}\mathbb{E}[\widehat{X}(e^{-t}\mathbf{b}_s)^2]}ds}{|\mathbf{b}_s|^2}
 \bigg)
1_{\overline{\mathcal{H}}_n}\Bigg]\\
\notag
& =
\int_{-\log(\delta_1 / \eta) - \widetilde{\tau}}^\infty \frac{dt}{\pi} \bdec{0}{0}{1}\Bigg[  \mathcal{I} \bigg(
\int_0^1 e^{\gamma \left[B_{t - \log|\mathbf{b}_s|+\widetilde{\tau}} -(Q-\gamma)(t - \log|\mathbf{b}_s| + \widetilde{\tau})\right]
-\gamma \left[B_{\widetilde{\tau}} -(Q-\gamma)\widetilde{\tau}\right]}\\
\label{eq:pre-beta}
&\qquad \qquad  \times 
\frac{e^{\gamma \widehat{X}(\eta e^{\widetilde{\tau}} e^{-t}\mathbf{b}_s) - \frac{\gamma^2}{2}\mathbb{E}[\widehat{X}(\eta e^{\widetilde{\tau}} e^{-t}\mathbf{b}_s)^2]}ds}{|\mathbf{b}_s|^2}
 \bigg)
1_{\overline{\mathcal{H}}_n}\Bigg]
\end{align}

\noindent with
\begin{align*}
\widetilde{\tau}
:= \widetilde{\tau}_{\lambda C_x \eta^{2-\gamma^2}}
:= \inf\left\{u > 0: e^{\gamma [B_u - (Q-\gamma)u]} = (\lambda C_x \eta^{2-\gamma^2})^{-1}\right\}.
\end{align*}

\noindent Since $\eta e^{\widetilde{\tau}}$ is independent of the scale invariant field $\widehat{X}$, we see that \eqref{eq:pre-beta} has the same distribution as
\begin{align}\label{eq:pre-limit}
\int_{-\log(\delta_1 / \eta) - \widetilde{L}}^\infty \frac{dt}{\pi} \bdec{0}{0}{1}\Bigg[  \mathcal{I} \bigg(
\int_0^1 e^{-\gamma \beta_{t - \log|\mathbf{b}_s|}^{Q-\gamma}}
\frac{e^{\gamma \widehat{X}(e^{-t}\mathbf{b}_s) - \frac{\gamma^2}{2}\mathbb{E}[\widehat{X}(e^{-t}\mathbf{b}_s)^2]}ds}{|\mathbf{b}_s|^2}
 \bigg)
1_{\overline{\mathcal{H}}_n}\Bigg]
\end{align}

\noindent where $\widetilde{L} := \widetilde{L}_{\lambda C_x \eta^{2-\gamma^2}} := \sup \left\{u > 0: \beta_{-u}^{Q-\gamma} = \lambda C_x \eta^{2-\gamma^2}\right\}$ by \Cref{lem:time_reversal}. As everything inside $ \bdec{0}{0}{1}[\cdot]$ in \eqref{eq:pre-limit} is non-negative and independent of $\lambda C_x$, and $\widetilde{L} \xrightarrow{\lambda \to \infty} \infty$ a.s., it follows from monotone convergence that \eqref{eq:pre-limit} converges as $\lambda \to \infty$ to
\begin{align*}
\int_{-\infty}^\infty \frac{dt}{\pi} \bdec{0}{0}{1}\Bigg[  \mathcal{I} \bigg(
\int_0^1 e^{-\gamma \beta_{t - \log|\mathbf{b}_s|}^{Q-\gamma}}
\frac{e^{\gamma \widehat{X}(e^{-t}\mathbf{b}_s) - \frac{\gamma^2}{2}\mathbb{E}[\widehat{X}(e^{-t}\mathbf{b}_s)^2]}ds}{|\mathbf{b}_s|^2}
 \bigg)
1_{\overline{\mathcal{H}}_n}\Bigg].
\end{align*}

\noindent Now that the above expression is independent of $\delta_1 > 0$, we may first send $\delta_1 \to 0^+$ and then $n \to \infty$  (so that the condition $4n\delta_1 < \kappa$ remains satisfied)  to conclude the proof by monotone convergence and continuous mapping theorem.
\end{proof}

\begin{proof}[Proof of \Cref{T:HK}]
Combining all the analysis from Step 1--4 above, we are only left with the final task of verifying $\mathbb{E}[J_\gamma^\infty] = c_{\gamma}$. A direct computation would not be straightforward, and we shall proceed instead by reversing the sequence of arguments in the proof of \Cref{lem:Jinfty_limit} and making use of
\begin{align}\label{eq:checkmean}
\mathbb{E}[J_{\gamma}^\infty]
& = \lim_{n \to \infty} \lim_{\lambda \to \infty}
\int_0^{\delta_1^2} \frac{du}{2\pi u} \mathbb{E}\otimes\bdec{0}{0}{1}\Bigg[  \mathcal{I} \left(
\lambda  C_x \overline{F}_{\gamma}^{\{0\}}( \sqrt{u}\mathbf{b}_{\cdot/u};  X^{\eta \mathbb{D}} \right)1_{\overline{\mathcal{H}}_n}\Bigg]
\end{align}

\noindent as a result of monotone convergence. Note that the evaluation needed on the RHS is independent of the choice of $\delta_1$ (which is allowed to depend on $n$), and in particular we may take $\delta_1 = \eta / n$ to ensure that $\mathbf{b}_{\cdot} \in B(0, \eta) = \eta \mathbb{D}$ on the event $\overline{H}_n$. We then follow the strategy in \Cref{sec:proofWeyl} and invoke the scaling behaviour of $X^{\eta \mathbb{D}}$, leading us to
\begin{align*}
\overline{F}_{\gamma}^{\{0\}}(\sqrt{u}\mathbf{b}_{\cdot/u};  X^{\eta \mathbb{D}}) 
&\overset{d}{=} 
u^{1 - \frac{\gamma^2}{2}} e^{\gamma B_T - \frac{\gamma^2}{2} T}\int_0^1 \frac{e^{\gamma X^{\eta \mathbb{D}}(\frac{\eta}{n}\mathbf{b}_{s}) - \frac{\gamma^2}{2} \mathbb{E}[X^{\eta \mathbb{D}}(\frac{\eta}{n}\mathbf{b}_{s})^2]}ds}{|\mathbf{b}_{s}|^{\gamma^2}}\\
& \overset{d}{=} 
  e^{\gamma \left(B_T - (Q-\gamma)T\right)}\underbrace{(\eta/n)^{2-\gamma^2}\int_0^1 \frac{e^{\gamma X^{ \mathbb{D}}(n^{-1}\mathbf{b}_{s}) - \frac{\gamma^2}{2} \mathbb{E}[X^{\mathbb{D}}(n^{-1}\mathbf{b}_{s})^2]}ds}{|\mathbf{b}_{s}|^{\gamma^2}}}_{=:\mathcal{E}_{x, n}}.
\end{align*}

\noindent where $B_T \sim \mathcal{N}(0, T)$ with $T = T(u; n, \eta) := -\log \left(n \sqrt{u}/\eta\right)$ is independent of everything else. Substituting this back to \eqref{eq:checkmean}, we obtain
\begin{align*}
\mathbb{E}[J_{\gamma}^\infty]
& = \lim_{n \to \infty} \lim_{\lambda \to \infty}
 \int_0^\infty \frac{dt}{\pi}\mathbb{E}\otimes\bdec{0}{0}{1}\Bigg[  \mathcal{I} \left(
\lambda  C_x e^{\gamma \left(B_t - (Q-\gamma)t\right)} \mathcal{E}_{x, n}\right)1_{\overline{\mathcal{H}}_n}\Bigg]\\
& = \lim_{n \to \infty} c_{\gamma}\mathbb{E}\otimes\bdec{0}{0}{1}[1_{\overline{\mathcal{H}}_n}] 
= c_\gamma
\end{align*}

\noindent by \Cref{lem:main}, and the proof of \Cref{T:HK} is now complete.
\end{proof}

\subsection{Evaluating the constant $c_\gamma(m)$: proof of \Cref{theo:constant}}
\begin{proof}[Proof of  \Cref{theo:constant}]
Recall from \Cref{lem:const_finite} that $c_\gamma(m)$ defined by the probabilistic representation \eqref{eq:constant} or equivalently \eqref{eq:constant2} is finite for any $\gamma, m > 0$. Moreover, from \Cref{lem:main} we may write
\begin{align*}
\pi c_{\gamma}(m)
& = \lim_{\lambda \to \infty} \mathbb{E}\left[ \int_0^\infty \mathcal{I}(\lambda e^{\gamma (B_t - mt)})dt\right]\\
& = \lim_{\lambda \to \infty} \int_0^\infty dt  \int_0^\infty \lambda ue^{-\lambda u} \mathbb{P}(e^{\gamma (B_t - mt)} \in du) \\
& = \lim_{\lambda \to \infty}\lambda   \int_0^\infty  ue^{-\lambda u} \underbrace{\left[ \int_0^\infty \frac{1}{u \gamma \sqrt{2\pi t}} \exp\left(-\frac{1}{2\gamma^2 t} (\log u +\gamma m t)^2\right) dt \right]}_{(*)} du
\end{align*}

\noindent where $(*)$ is integrable for any $u > 0$. By the standard Hardy-Littlewood-Karamata Tauberian Theorem (i.e.  \Cref{theo:tauberian} in the deterministic setting), we also have
\begin{align*}
\pi c_{\gamma}(m)
& = \lim_{\lambda \to \infty}\lambda   \int_0^{1/\lambda} u \underbrace{\left[ \int_0^\infty \frac{1}{u \gamma \sqrt{2\pi t}} \exp\left(-\frac{1}{2\gamma^2 t} (\log u +\gamma m t)^2\right) dt \right]}_{(*)} du\\
& =  \lim_{\lambda \to \infty} \mathbb{E}\left[ \int_0^\infty \widetilde{\mathcal{I}}(\lambda e^{\gamma (B_t - mt)})dt\right]
\end{align*}

\noindent with $\widetilde{\mathcal{I}}(x) := x 1_{\{x \le 1\}}$ and in particular $\widetilde{\mathcal{I}}(x) = 0$ for $x > 1$. Introducing the stopping time $\widetilde{\tau}_\lambda := \inf \{t > 0: e^{\gamma (B_t - mt)} = 1/\lambda \}$,  we have for any $\lambda > 0$ that
\begin{align*}
\mathbb{E}\left[ \int_0^\infty \widetilde{\mathcal{I}}(\lambda e^{\gamma (B_t - mt)})dt\right]
& = \mathbb{E}\left[ \int_{\widetilde{\tau}_{\lambda}}^\infty \widetilde{\mathcal{I}}(\lambda e^{\gamma (B_t - mt)})dt\right]\\
& = \mathbb{E}\left[ \int_{\widetilde{\tau}_{\lambda}}^\infty \widetilde{\mathcal{I}}(e^{\gamma [(B_t - mt) - (B_{\widetilde{\tau}_\lambda} - m\widetilde{\tau}_\lambda )]})dt\right]
= \mathbb{E}\left[ \int_0^\infty \widetilde{\mathcal{I}}(e^{\gamma (B_t - mt)})dt\right]
\end{align*}

\noindent by the strong Markov property. If we denote by $\Phi(\cdot)$ the cumulative distribution function of standard Gaussian random variables, then
\begin{align*}
c_{\gamma}(m)
& = \frac{1}{\pi} \int_0^\infty e^{(\frac{\gamma^2}{2} - \gamma m) t} \mathbb{E}\left[ e^{\gamma B_t - \frac{\gamma^2}{2}t} 1_{\{B_t - mt \le 0\}}\right]dt \\
&= \frac{1}{\pi} \int_0^\infty e^{(\frac{\gamma^2}{2} - \gamma m) t} \Phi\left( (m - \gamma)\sqrt{t}\right) dt\\
&= \frac{2}{\pi \gamma(\gamma - 2m)} \left\{\left[ e^{(\frac{\gamma^2}{2} - \gamma m) t} \Phi\left( (m - \gamma)\sqrt{t}\right) \right]_{0}^\infty -  \int_0^\infty e^{(\frac{\gamma^2}{2} - \gamma m) t} \partial_t \Phi\left( (m - \gamma)\sqrt{t}\right) dt\right\}\\
&= \frac{1}{\pi \gamma(\gamma - 2m)} \left[- 1 -  2(m-\gamma)\int_0^\infty e^{(\frac{\gamma^2}{2} - \gamma m) s^2} e^{-\frac{(m-\gamma)^2 s^2}{2}} \frac{ds}{\sqrt{2\pi}}\right]\\
&= \frac{1}{\pi \gamma(\gamma - 2m)} \left[- 1 -  \frac{m-\gamma}{m}\right]
= \frac{1}{\pi \gamma m}
\end{align*}

\noindent which is our desired result.
\end{proof}

\appendix
\section{Probabilistic asymptotics}\label{app:probasy}
This appendix collects some probabilistic generalisations of common asymptotic results that are suitable in the context of convergence in probability. The first one concerns ``asymptotic differentiations".

\begin{lem}\label{lem:asympdiff}
Let $\alpha, \beta > 0$ be fixed, and $\varphi(u): \mathbb{R}_+ \mapsto \mathbb{R}_+$ a random non-increasing function. Suppose there exists some a.s. positive random variable $C$ such that
\begin{align*}
    t^{-\beta} \int_0^t u^{\alpha - 1} \varphi(u) du
    \xrightarrow[t \to 0^+]{p} C,
\end{align*}

\noindent then
\begin{align*}
    t^{\alpha - \beta} \varphi(t) \xrightarrow[t \to 0^+]{p} \beta C.
\end{align*}
\end{lem}

\begin{proof}
Without loss of generality suppose $C = 1$ almost surely. We start with the upper bound, i.e. we would like to establish
\begin{align*}
    \lim_{t \to 0^+} \mathbb{P}( t^{\alpha - \beta} \varphi(t) - \beta > \epsilon) = 0 \qquad \forall \epsilon > 0.
\end{align*}

\noindent For this, consider, for fixed $b > 1$, the deterministic inequality
\begin{align*}
\int_{b^{-1}t}^{t} u^{\alpha - 1} \varphi(u) du
\ge \varphi(t) \int_{b^{-1}t}^{t} u^{\alpha-1} du
= t^\alpha \varphi(t) \frac{1 -b^{-\alpha}}{\alpha}.
\end{align*}

\noindent Then for any $\epsilon' > 0$, we have
\begin{align}
\notag
 &\lim_{t \to 0^+} \mathbb{P}( t^{\alpha - \beta} \varphi(t) - \beta > \epsilon)\\
\notag
 & \quad \le \lim_{t \to 0^+} \mathbb{P}\left(\left(\frac{\alpha}{1-b^{-\alpha}}\right) t^{-\beta}\int_{b^{-1}t}^t u^{\alpha - 1}\varphi(u) du - \beta \ge \epsilon \right)\\
\notag
 & \quad \le \lim_{t \to 0^+} \mathbb{P}\left(\left|t^{-\beta} \int_0^t u^{\alpha-1} \varphi(u) du - 1\right| > \epsilon' \right)
 +\lim_{t \to 0^+} \mathbb{P}\left(\left |(b^{-1}t)^{-\beta} \int_0^{b^{-1}t} u^{\alpha-1} \varphi(u) du - 1\right| > \epsilon' \right)\\
\notag
 &\qquad + 1\left\{\left(\frac{\alpha}{1-b^{-\alpha}}\right) \left[ (1+\epsilon') - b^{-\beta}(1-\epsilon')\right] - \beta > \epsilon\right\}\\
\label{eq:Pdiff_indicator}
 & \quad = 1\left\{\frac{\alpha(1-b^{-\beta})}{1-b^{-\alpha}} - \beta  + \frac{\alpha(1+b^{-\beta})}{1-b^{-\alpha}} \epsilon'> \epsilon\right\}.
\end{align}

Given that
\begin{align*}
    \lim_{b \to 1} \frac{\alpha(1-b^{-\beta})}{1-b^{-\alpha}} - \beta = 0,
\end{align*}

\noindent we can choose $b$ sufficiently close to $1$ and then $\epsilon' > 0$ sufficiently small such that
\begin{align*}
   \left| \frac{\alpha(1-b^{-\beta})}{1-b^{-\alpha}} - \beta\right| < \frac{\epsilon}{2}
   \qquad \text{and} \qquad
   \frac{\alpha(1+b^{-\beta})}{1-b^{-\alpha}} \epsilon'
    < \frac{\epsilon}{2},
\end{align*}

\noindent in which case the indicator function in \eqref{eq:Pdiff_indicator} is always evaluated to $0$. By a similar argument, one may obtain the lower bound
\begin{align*}
    \lim_{t \to 0^+} \mathbb{P}( t^{\alpha - \beta} \varphi(t) - \beta < \epsilon) = 0
\end{align*}

\noindent by considering the integral $\int_t^{bt} u^{\alpha - 1} \varphi(u) du$. This concludes the proof.
\end{proof}

The next result is a probabilistic generalisation of the Hardy--Littlewood Tauberian theorem. The version we are stating is slightly more general than what is needed here as it could be of independent interest. Recall that a function $L: (0, \infty) \to (0, \infty)$ is slowly varying at zero if $\lim_{t \to 0^+} L(xt)/L(t) = 1$ for any $x > 0$.\footnote{One can also talk about slow variation at infinity by considering the analogous ratio limit as $t \to \infty$.}
\begin{theo}\label{theo:tauberian}
Let $\nu(d \cdot)$ be a non-negative random measure on $\mathbb{R}_+$, $\nu(t) := \int_0^t \nu(ds)$, and suppose the Laplace transform
\begin{align*}
    \hat{\nu}(\lambda) := \int_0^\infty e^{-\lambda s} \nu(ds)
\end{align*}

\noindent exists almost surely for any $\lambda > 0$. If
\begin{itemize}\setlength\itemsep{0em}
\item $\rho \in [0, \infty)$ is fixed;
\item $L: (0, \infty) \to (0, \infty)$ is a deterministic slowly varying function at $0$; and 
\item $C_\nu$ is some non-negative (finite) random variable,
\end{itemize}

\noindent then we have:
\begin{align}\label{eq:tauberian}
    \frac{\lambda^\rho}{L(\lambda^{-1})} \hat{\nu}(\lambda) \xrightarrow[\lambda \to \infty]{p} C_\nu
    \qquad \Rightarrow \qquad 
    \frac{t^{-\rho}}{L(t)} \nu(t) \xrightarrow[t \to 0^+]{p} \frac{C_\nu}{\Gamma(1+\rho)}.
\end{align}

\noindent The same implication also holds when one considers the asymptotics as $\lambda \to 0^+$ and $t \to \infty$ in \eqref{eq:tauberian} (but with $L$ being slowly varying at infinity) instead.
\end{theo}

Following \cite[Chapter I, Section 15]{Kor2004} as well as \cite[Appendix A]{NPS2023}, our proof of \Cref{theo:tauberian} is based on adapting Karamata's argument to the probabilistic setting, and the main ingredient is the following deterministic approximation lemma.
\begin{lem}\label{lem:polyapprox}
For each $\alpha \ge 0$ and $\epsilon \in (0, 1/2e)$, there exist some constant $C = C(\alpha) < \infty$ independent of $\epsilon$ and polynomials $\Pa_\pm(\cdot)$ without constant terms (i.e. $\Pa_\pm(0) = 0$) such that $\Pa_-(x) \le 1_{[e^{-1}, 1]}(x) \le \Pa_+(x)$ for any $x \in [0,1]$ and
\begin{align} \label{eq:polyapprox}
\int_0^1 \left|\Pa_\pm(x) - 1_{[e^{-1}, 1]}(x) \right| \alpha \left(\log \frac{1}{x}\right)^{\alpha-1}\frac{dx}{x} \le C(\alpha) \epsilon.
\end{align}
\end{lem}

\begin{proof}
We focus on the construction of $\Pa_+$ since the other one is similar. To begin with, define a continuous function $h: [0, 1] \to \RR_+$ by
\begin{align*}
h(x) = \begin{cases}
0  & \text{if $x \in [0, e^{-1} - \epsilon]$},\\
\epsilon^{-1} [x - (e^{-1} - \epsilon)] & \text{if $x \in [e^{-1} - \epsilon, e^{-1}]$},\\
1 & \text{if $x \in [e^{-1}, 1]$.}
\end{cases}
\end{align*}

\noindent It is straightforward to see that $h(x) \ge 1_{[e^{-1}, 1]}(x)$ for all $x \in [0,1]$ and
\begin{align*}
\int_0^1 \left[h(x) - 1_{[e^{-1}, 1]}(x)\right]  \alpha \left(\log \frac{1}{x}\right)^{\alpha-1}\frac{dx}{x} 
&\le \int_{e^{-1} - \epsilon}^{e^{-1}}  [x - (e^{-1} - \epsilon)]^2  \alpha \left(\log \frac{1}{x}\right)^{\alpha-1}\frac{dx}{x}\\
& \le \alpha e \left(\log (2e)\right)^{\alpha} \epsilon^2.
\end{align*}

Next, using Weierstrass theorem, there exists some polynomial $\widetilde{\Pa}(\cdot)$ such that
\begin{align*}
\left|\widetilde{\Pa}(x) - \left(\frac{h(x)}{x} + \epsilon\right)\right| \le \epsilon \qquad \forall x \in [0, 1].
\end{align*}

\noindent This means in particular that $\Pa_+(x) := x \widetilde{\Pa}(x)$ (which is a polynomial without constant term) satisfies $\Pa_+(x) \ge h(x) \ge 1_{[e^{-1}, 1]}(x)$ for all $x \in [0,1]$ and
\begin{align*}
\int_0^1 \left[\Pa_+(x) - h(x)\right]  \alpha \left(\log \frac{1}{x}\right)^{\alpha-1}\frac{dx}{x}
\le 2 \epsilon \alpha \int_0^1 \left(\log 1/x\right)^{\alpha-1} dx
= 2\Gamma(\alpha + 1) \epsilon.
\end{align*}

\noindent Combining everything, we arrive at
\begin{align*}
\int_0^1 \left|\Pa_\pm(x) - 1_{[e^{-1}, 1]}(x) \right| \alpha \left(\log \frac{1}{x}\right)^{\alpha-1}\frac{dx}{x} \le
\left[ \alpha e \left(\log (2e)\right)^{\alpha}  + 2\Gamma(\alpha+1)\right] \epsilon
\end{align*}

\noindent which concludes the proof.
\end{proof}

\begin{proof}[Proof of \Cref{theo:tauberian}]
We shall focus on the claim \eqref{eq:tauberian}, as the other case (i.e. the same implication but with $\lambda \to 0^+$ and $t \to \infty$) follows from the arguments below ad verbatim. To begin with, observe that for each $k \in \mathbb{N}$,
\begin{align}\notag 
\frac{t^{-\rho}}{L(t)} \int_0^\infty e^{-\frac{k}{t}{s}} \nu(ds)
& = k^{-\rho} \frac{L(t/k)}{L(t)} \left[ \frac{(k/t)^{\rho}}{L(t/k)} \hat{\nu}(k/t)\right]\\
\label{eq:tauproof_step0}
& \xrightarrow[t \to 0^+]{p} k^{-\rho} C_\nu
= \frac{C_{\nu}}{\Gamma(1+\rho)} \int_0^\infty e^{-ks} d(s^\rho).
\end{align}

Let us fix some $\epsilon > 0$ to be chosen later, and find a polynomial $\Pa_+(x) = \sum_{k=1}^m p_k x^k$ satisfying the conditions in \Cref{lem:polyapprox}. Since $m = m(\epsilon) > 0$ is finite, \eqref{eq:tauproof_step0} combined with a simple union bound argument suggests that
\begin{align*}
\frac{t^{-\rho}}{L(t)} \int_0^\infty \Pa_+(e^{-s/t}) \nu(ds)
& = \frac{t^{-\rho}}{L(t)} \sum_{k=1}^m p_k \int_0^\infty e^{-\frac{k}{t}s} \nu(ds)\\
& \xrightarrow[t \to 0^+]{p}
\frac{C_\nu}{\Gamma(1+\rho)} \sum_{k=1}^m p_k \int_0^\infty e^{-ks} d(s^\rho)
= \frac{C_\nu}{\Gamma(1+\rho)} \int_0^\infty \Pa_+(e^{-s}) d(s^\rho).
\end{align*}

On the other hand,
\begin{align*}
\nu(t) = \int_0^\infty 1_{[e^{-1}, 1]}(e^{-s/t}) \nu(ds)
\le \int_0^\infty \Pa_+(e^{-s/t}) \nu(ds).
\end{align*}

\noindent Thus for any $\delta > 0$, we have
\begin{align*}
& \limsup_{t \to 0^+} \mathbb{P}\left( \frac{t^{-\rho}}{L(t)} \nu(t) - \frac{C_\nu}{\Gamma(1+\rho)} > \delta\right)\\
& \qquad \le \limsup_{t \to 0^+} \mathbb{P}\left( \frac{t^{-\rho}}{L(t)}\int_0^\infty \Pa_+(e^{-s/t})\nu(ds) - \frac{C_\nu}{\Gamma(1+\rho)} > \delta\right)\\
& \qquad \le \limsup_{t \to 0^+} \mathbb{P}\left( \frac{t^{-\rho}}{L(t)}\int_0^\infty \Pa_+(e^{-s/t})\nu(ds) - \frac{C_\nu}{\Gamma(1+\rho)} \int_0^\infty \Pa_+(e^{-s}) d(s^\rho) > \frac{\delta}{2} \right)\\
& \qquad \qquad + \mathbb{P}\left( \frac{C_\nu}{\Gamma(1+\rho)} \left[ \int_0^\infty \Pa_+(e^{-s}) d(s^\rho)-  1 \right]> \frac{\delta}{2} \right) \\
& \qquad = \mathbb{P}\left( \frac{C_\nu}{\Gamma(1+\rho)} \int_0^\infty  \left[\Pa_+(e^{-s}) - 1_{[e^{-1}, 1]}(e^{-s})\right]d(s^\rho)> \frac{\delta}{2} \right)
\le \mathbb{P}\left( \frac{C_\nu}{\Gamma(1+\rho)} \cdot C(\rho) \epsilon > \frac{\delta}{2} \right)
\end{align*}

\noindent where $C(\rho) \epsilon$ comes from the deterministic bound \eqref{eq:polyapprox}. Since $\epsilon > 0$ is arbitrary, we can send $\epsilon \to 0^+$ and obtain
\begin{align*}
\limsup_{t \to 0^+} \mathbb{P}\left( \frac{t^{-\rho}}{L(t)} \nu(t) - \frac{C_\nu}{\Gamma(1+\rho)} > \delta\right) = 0.
\end{align*}

\noindent Similarly, using the polynomial approximation $\Pa_-(\cdot)$ we can also obtain
\begin{align*}
\limsup_{t \to 0^+} \mathbb{P}\left(\frac{C_\nu}{\Gamma(1+\rho)} - \frac{t^{-\rho}}{L(t)} \nu(t) > \delta\right) = 0
\end{align*}

\noindent and the proof is complete.

\end{proof}

\bibliographystyle{alpha}
\bibliography{main}

\newcommand{\etalchar}[1]{$^{#1}$}
\begin{thebibliography}{DKRV16}

\bibitem[AGZ10]{AGZ}
Greg~W. Anderson, Alice Guionnet, and Ofer Zeitouni.
\newblock {\em An introduction to random matrices}, volume 118 of {\em
  Cambridge Studies in Advanced Mathematics}.
\newblock Cambridge University Press, Cambridge, 2010.

\bibitem[AK16]{AK2016}
Sebastian Andres and Naotaka Kajino.
\newblock Continuity and estimates of the {L}iouville heat kernel with
  applications to spectral dimensions.
\newblock {\em Probability Theory and Related Fields}, 166:713--752, 2016.

\bibitem[ANR{\etalchar{+}}98]{ABNRW1998}
Jan Ambj{\o}rn, Jakob~L Nielsen, Juri Rolf, Dimitrij Boulatov, and Yoshiyuki
  Watabiki.
\newblock The spectral dimension of 2d quantum gravity.
\newblock {\em Journal of High Energy Physics}, 1998(02):010, 1998.

\bibitem[Ber72]{Berezin}
Felix~A Berezin.
\newblock Covariant and contravariant symbols of operators.
\newblock {\em Mathematics of the USSR-Izvestiya}, 6(5):1117, 1972.

\bibitem[Ber77]{Ber1977}
Michael~V Berry.
\newblock Regular and irregular semiclassical wavefunctions.
\newblock {\em Journal of Physics A: Mathematical and General}, 10(12):2083,
  1977.

\bibitem[Ber15]{Ber2015}
Nathana\"{e}l Berestycki.
\newblock Diffusion in planar {L}iouville quantum gravity.
\newblock {\em Ann. Inst. Henri Poincar\'{e} Probab. Stat.}, 51(3):947--964,
  2015.

\bibitem[Ber17]{Ber2017}
Nathana\"{e}l Berestycki.
\newblock An elementary approach to {G}aussian multiplicative chaos.
\newblock {\em Electron. Commun. Probab.}, 22:Paper No. 27, 12, 2017.

\bibitem[BG22]{BerestyckiGwynne}
Nathana{\"e}l Berestycki and Ewain Gwynne.
\newblock Random walks on mated-{CRT} planar maps and {L}iouville {B}rownian
  motion.
\newblock {\em Communications in Mathematical Physics}, 395(2):773--857, 2022.

\bibitem[BGRV16]{HKPZ}
Nathana\"{e}l Berestycki, Christophe Garban, R\'{e}mi Rhodes, and Vincent
  Vargas.
\newblock K{PZ} formula derived from {L}iouville heat kernel.
\newblock {\em J. Lond. Math. Soc. (2)}, 94(1):186--208, 2016.

\bibitem[BGS84]{BGS1984}
O.~Bohigas, M.~J. Giannoni, and C.~Schmit.
\newblock Spectral properties of the {L}aplacian and random matrix theories.
\newblock {\em Journal de Physique Lettres}, 45(21):1015--1022, 1984.

\bibitem[BP24]{BP}
Nathana\"el Berestycki and Ellen Powell.
\newblock {\em Gaussian free field and {L}iouville quantum gravity}.
\newblock Cambridge Series in Advanced Mathematics. Cambridge University Press,
  2024+.

\bibitem[BSS14]{BerestyckiSheffieldSun}
Nathana{\"e}l Berestycki, Scott Sheffield, and Xin Sun.
\newblock Equivalence of {L}iouville measure and {G}aussian free field.
\newblock {\em Preprint arXiv:1410.5407}, 2014.

\bibitem[CCH17]{CharmoyCroydonHambly}
Philippe Charmoy, David Croydon, and Ben Hambly.
\newblock Central limit theorems for the spectra of classes of random fractals.
\newblock {\em Transactions of the American Mathematical Society},
  369(12):8967--9013, 2017.

\bibitem[CH08]{CH2008}
David Croydon and Ben Hambly.
\newblock Self-similarity and spectral asymptotics for the continuum random
  tree.
\newblock {\em Stochastic processes and their applications}, 118(5):730--754,
  2008.

\bibitem[CH10]{CH2010}
David~A Croydon and Ben~M Hambly.
\newblock Spectral asymptotics for stable trees.
\newblock {\em Electronic Journal of Probability}, 15(57):1772--1801, 2010.

\bibitem[Cha84]{Chavel}
Isaac Chavel.
\newblock {\em Eigenvalues in {R}iemannian geometry}, volume 115 of {\em Pure
  and Applied Mathematics}.
\newblock Academic Press, Inc., Orlando, FL, 1984.
\newblock Including a chapter by Burton Randol, With an appendix by Jozef
  Dodziuk.

\bibitem[DKRV16]{DKRV}
Fran{\c{c}}ois David, Antti Kupiainen, R{\'e}mi Rhodes, and Vincent Vargas.
\newblock Liouville quantum gravity on the {R}iemann sphere.
\newblock {\em Communications in Mathematical Physics}, 342(3):869--907, 2016.

\bibitem[DMS21]{DuplantierMillerSheffield}
Bertrand Duplantier, Jason Miller, and Scott Sheffield.
\newblock Liouville quantum gravity as a mating of trees.
\newblock {\em Ast{\'e}risque}, 427, 2021.

\bibitem[DS11]{DS2011}
Bertrand Duplantier and Scott Sheffield.
\newblock Liouville quantum gravity and {KPZ}.
\newblock {\em Invent. Math.}, 185(2):333--393, 2011.

\bibitem[FLPS23]{FLPS2023}
Nikolay Filonov, Michael Levitin, Iosif Polterovich, and David~A Sher.
\newblock P{\'o}lya's conjecture for euclidean balls.
\newblock {\em Inventiones mathematicae}, pages 1--41, 2023.

\bibitem[FW00]{ForresterWitte}
P.~J. Forrester and N.~S. Witte.
\newblock Exact {W}igner surmise type evaluation of the spacing distribution in
  the bulk of the scaled random matrix ensembles.
\newblock {\em Lett. Math. Phys.}, 53(3):195--200, 2000.

\bibitem[Gau61]{Gau1961}
Michel Gaudin.
\newblock Sur la loi limite de l'espacement des valeurs propres d'une matrice
  al{\'e}atoire.
\newblock {\em Nuclear Physics}, 25:447--458, 1961.

\bibitem[GH20]{GwynneHutchcroft}
Ewain Gwynne and Tom Hutchcroft.
\newblock Anomalous diffusion of random walk on random planar maps.
\newblock {\em To appear in Probab. Theory Relat. Fields}, 2020.

\bibitem[GM17]{GwynneMiller_rw}
Ewain Gwynne and Jason Miller.
\newblock Random walk on random planar maps: spectral dimension, resistance,
  and displacement.
\newblock {\em Preprint arXiv:1711.00836}, 2017.

\bibitem[GRV14]{GRV_hk}
Christophe Garban, R\'{e}mi Rhodes, and Vincent Vargas.
\newblock On the heat kernel and the {D}irichlet form of {L}iouville {B}rownian
  motion.
\newblock {\em Electron. J. Probab.}, 19:no. 96, 25, 2014.

\bibitem[GRV16]{GRV}
Christophe Garban, R\'{e}mi Rhodes, and Vincent Vargas.
\newblock Liouville {B}rownian motion.
\newblock {\em Ann. Probab.}, 44(4):3076--3110, 2016.

\bibitem[GWW92]{GWW1992}
Carolyn Gordon, David Webb, and Scott Wolpert.
\newblock Isospectral plane domains and surfaces via riemannian orbifolds.
\newblock {\em Inventiones mathematicae}, 110(1):1--22, 1992.

\bibitem[Ham00]{Hambly}
Ben~M. Hambly.
\newblock On the asymptotics of the eigenvalue counting function for random
  recursive {S}ierpinski gaskets.
\newblock {\em Probability theory and related fields}, 117:221--247, 2000.

\bibitem[Ivr16]{Ivrii100}
Victor Ivrii.
\newblock 100 years of {W}eyl's law.
\newblock {\em Bull. Math. Sci.}, 6(3):379--452, 2016.

\bibitem[Kac66]{Kac1966}
Mark Kac.
\newblock Can one hear the shape of a drum?
\newblock {\em The american mathematical monthly}, 73(4P2):1--23, 1966.

\bibitem[Kaj13]{Kaj2013}
Naotaka Kajino.
\newblock On-diagonal oscillation of the heat kernels on post-critically finite
  self-similar fractals.
\newblock {\em Probability Theory and Related Fields}, 156(1-2):51--74, 2013.

\bibitem[KL93]{kigami1993weyl}
Jun Kigami and Michel~L Lapidus.
\newblock Weyl's problem for the spectral distribution of {L}aplacians on pcf
  self-similar fractals.
\newblock {\em Communications in mathematical physics}, 158:93--125, 1993.

\bibitem[Kor04]{Kor2004}
Jacob Korevaar.
\newblock {\em Tauberian theory: a century of developments}.
\newblock Springer, 2004.

\bibitem[KPZ88]{KPZ}
Vadim~G. Knizhnik, Alexander~M Polyakov, and Alexander~B. Zamolodchikov.
\newblock Fractal structure of 2d quantum gravity.
\newblock {\em Modern Physics Letters A}, 3(08):819--826, 1988.

\bibitem[KRV20]{DOZZ}
Antti Kupiainen, R\'{e}mi Rhodes, and Vincent Vargas.
\newblock Integrability of {L}iouville theory: proof of the {DOZZ} formula.
\newblock {\em Ann. of Math. (2)}, 191(1):81--166, 2020.

\bibitem[LRV22]{LacoinRhodesVargas}
Hubert Lacoin, R{\'e}mi Rhodes, and Vincent Vargas.
\newblock The semiclassical limit of {L}iouville conformal field theory.
\newblock {\em Annales de la Facult{\'e} des sciences de Toulouse:
  Math{\'e}matiques}, 31(4):1031--1083, 2022.

\bibitem[LY83]{LiYau}
Peter Li and Shing-Tung Yau.
\newblock On the {S}chr{\"o}dinger equation and the eigenvalue problem.
\newblock {\em Communications in Mathematical Physics}, 88(3):309--318, 1983.

\bibitem[Meh04]{Meh2004}
Madan~Lal Mehta.
\newblock {\em Random matrices}.
\newblock Elsevier, 3rd edition, 2004.

\bibitem[Mil64]{Mil1964}
John Milnor.
\newblock Eigenvalues of the {L}aplace operator on certain manifolds.
\newblock {\em Proceedings of the National Academy of Sciences},
  51(4):542--542, 1964.

\bibitem[MRVZ16]{MRVZ}
Pascal Maillard, R\'emi Rhodes, Vincent Vargas, and Ofer Zeitouni.
\newblock {Liouville heat kernel: Regularity and bounds}.
\newblock {\em Annales de l'Institut Henri Poincar\'e, Probabilit\'es et
  Statistiques}, 52(3):1281 -- 1320, 2016.

\bibitem[NPS23]{NPS2023}
Joseph Najnudel, Elliot Paquette, and Nick Simm.
\newblock Secular coefficients and the holomorphic multiplicative chaos.
\newblock {\em The Annals of Probability}, 51(4):1193--1248, 2023.

\bibitem[P{\'o}l54]{Pol1954}
Georg P{\'o}lya.
\newblock {\em Mathematics and plausible reasoning (2 volumes)}.
\newblock Princeton University Press, 1954.

\bibitem[P{\'o}l61]{Polya_tilingdomains}
Georg P{\'o}lya.
\newblock On the eigenvalues of vibrating membranes.
\newblock {\em Proceedings of the London Mathematical Society},
  s3-11(1):419--433, 1961.

\bibitem[RP81]{RP1981}
L.~C.~G. Rogers and J.~W. Pitman.
\newblock Markov functions.
\newblock {\em The Annals of Probability}, 9(4), 1981.

\bibitem[RS94]{RudnickSarnak}
Ze\'{e}v Rudnick and Peter Sarnak.
\newblock The behaviour of eigenstates of arithmetic hyperbolic manifolds.
\newblock {\em Comm. Math. Phys.}, 161(1):195--213, 1994.

\bibitem[RV14]{RhodesVargas_spectral}
R{\'e}mi Rhodes and Vincent Vargas.
\newblock Spectral dimension of liouville quantum gravity.
\newblock In {\em Annales Henri Poincar{\'e}}, volume~15, pages 2281--2298.
  Springer, 2014.

\bibitem[Sar03]{Sarnak}
Peter Sarnak.
\newblock Spectra of hyperbolic surfaces.
\newblock {\em Bull. Amer. Math. Soc. (N.S.)}, 40(4):441--478, 2003.

\bibitem[She16]{zipper}
Scott Sheffield.
\newblock Conformal weldings of random surfaces: {SLE} and the quantum gravity
  zipper.
\newblock {\em Ann. Probab.}, 44(5):3474--3545, 2016.

\bibitem[Wey11]{Wey1911}
Hermann Weyl.
\newblock {\"U}ber die asymptotische {V}erteilung der {E}igenwerte.
\newblock {\em Nachrichten von der Gesellschaft der Wissenschaften zu
  G{\"o}ttingen, Mathematisch-Physikalische Klasse}, 1911:110--117, 1911.

\bibitem[Wil74]{Wil1974}
David Williams.
\newblock Path decomposition and continuity of local time for one-dimensional
  diffusions, {I}.
\newblock {\em Proceedings of the London Mathematical Society}, 3(4):738--768,
  1974.

\bibitem[Zel00]{Zelditch}
Steve Zelditch.
\newblock Spectral determination of analytic bi-axisymmetric plane domains.
\newblock {\em Geometric \& Functional Analysis GAFA}, 10(3):628--677, 2000.

\end{thebibliography}

\end{document}